\documentclass[preprint]{elsarticle}
\usepackage[T1]{fontenc}
\usepackage{color}
\usepackage{xcolor}
\usepackage{amsmath, amssymb, amsfonts}
\usepackage{psfrag}

\usepackage{subfig}
\usepackage{thmtools} %fuer Theorem im Appendix
\usepackage{thm-restate} %s.o.
\usepackage{hyperref}
\usepackage{amsthm}
\usepackage{bm}
\usepackage{mathtools}

\theoremstyle{definition}
\newtheorem{definition}{Definition}

\newtheorem{example}{Example}
\theoremstyle{plain}
\newtheorem{theorem}{Theorem}
\newtheorem{corollary}{Corollary}
\newtheorem{lemma}{Lemma}

\newcommand{\RFPairPartition}{\textsc{One-out-of-a-pair-partition}}

\newcommand{\ProblemName}{\textsc{RobT$\equiv$}}
\newcommand{\FirstProblem}{\textsc{RobMCF$\equiv$}}

\newcommand{\fett}[1]{\boldsymbol{#1}}

\DeclareMathOperator{\freehelp}{free}
\DeclareMathOperator{\fixhelp}{fix}
\newcommand{\free}{A^{\freehelp}}
\newcommand{\fix}{A^{\fixhelp}}

\newcommand{\sourceSPGraph}{o}
\newcommand{\sinkSPGraph}{q}

\definecolor{myred}{RGB}{255,0,0}
\definecolor{mygreen}{RGB}{84, 159,0}
\definecolor{myblue}{RGB}{0, 84, 159}
\definecolor{myorange}{RGB}{236, 111, 0}
\definecolor{mybrown}{RGB}{159,0,5}
\definecolor{rwth-llblue}{cmyk}{.25,.125,0,0}
\definecolor{rwth-blue}{cmyk}{1,.5,0,0}
\definecolor{rwth-lblue}{cmyk}{.5,.25,0,0}
\definecolor{rwth-red}{cmyk}{0,1,.5,0}
\definecolor{rwth-lred}{cmyk}{0,.5,.25,0}
\definecolor{Rwth-lred}{cmyk}{0,.5,.25,0}
\definecolor{rwth-green}{cmyk}{.5,0,1,0}
\definecolor{rwth-lgreen}{cmyk}{.25,0,.5,0}
\definecolor{rwth-orange}{cmyk}{0,.4,1,0}
\definecolor{rwth-blue}{RGB}{0,84,159}
\definecolor{rwth-lblue}{RGB}{142,186,226}
\definecolor{rwth-lgreen}{RGB}{189,205,0}
\definecolor{Rwth-lgreen}{RGB}{189,205,0}
\definecolor{rwth-green}{RGB}{87,171,39}
\definecolor{rwth-orange}{RGB}{246,168,0}
\definecolor{rwth-red}{RGB}{204,7,30}

\newcommand{\schwarzKnoten}[4]{\node[smarkKnoten, #4] #1 at #2 {};\node at #2 {#3};}

\usepackage{tikzscale}
\usepackage{graphicx,adjustbox}

\newenvironment{lyxlist}[1]
{\begin{list}{}
		{\settowidth{\labelwidth}{#1}
			\setlength{\leftmargin}{\labelwidth}
			\addtolength{\leftmargin}{\labelsep}
			}}
	{\end{list}}
\usepackage[Algorithm, section]{algorithm}
\usepackage[noend]{algpseudocode}

\makeatletter
\def\BState{\State\hskip-\ALG@thistlm}

\makeatother

\usepackage[colorinlistoftodos]{todonotes}

\usepackage{tikz}
\usetikzlibrary{shapes,arrows,backgrounds}
\usetikzlibrary{calc}
\usepackage{calc} 
\usetikzlibrary{decorations}
\usetikzlibrary{matrix}
\usetikzlibrary{positioning}
\usetikzlibrary{intersections}
\usetikzlibrary{snakes}
\usetikzlibrary{patterns}
\usetikzlibrary{trees,arrows}
\usetikzlibrary{arrows.meta,calc,decorations.markings,math,arrows.meta}
\usetikzlibrary{positioning}
\usepackage{stmaryrd} % für Widerspruchspfeil in Matheumgebung
\usepackage{wasysym} % für Widerspruchspfeil in Textumgebung

\usepackage{arydshln}
\usepackage{scalefnt}

\usepackage{multirow}
\usepackage{adjustbox}

\makeatother
\usetikzlibrary{quotes,angles,babel,calc}

\usepackage{pgfplots}

\definecolor{Darkgreen}{RGB}{0,100,0}
\definecolor{Darkblue}{RGB}{25,25,200}
\definecolor{Darkred}{RGB}{100,2,10}
\definecolor{Brombeer}{RGB}{160,80,160}
\newcommand{\gruen}{green!70!black}
\newcommand{\mycyan}{cyan!50!black}
\definecolor{sand}{RGB}{255,165,79} 
\definecolor{olive}{RGB}{107,142,35} 
\definecolor{newgray}{RGB}{112,128,144} 
\definecolor{hellesbrombeer}{rgb}{0.8,0.4,0.8} 
\definecolor{aquamarine}{RGB}{102,205,170} 
\definecolor{SteelBlue}{RGB}{70,130,180} 
\definecolor{amber}{rgb}{1.0, 0.75, 0.0}
\definecolor{mustard}{rgb}{1.0, 0.86, 0.35}
\definecolor{blond}{rgb}{0.98, 0.94, 0.75}

\definecolor{yellowpoly}{RGB}{255,255,0}
\definecolor{greenpoly}{RGB}{50,205,50}
\definecolor{bluepoly}{RGB}{0,0,205}

\tikzset{knoten/.style={shape=circle, ultra thick, fill = white, minimum size = 0.5cm, draw}}
\tikzset{klknoten/.style={shape=circle,minimum size = 0.01, draw}}
\tikzset{mknoten/.style={shape=circle,minimum size = 0.01, draw}}
\tikzset{miniknoten/.style={shape=circle,minimum size = 0.001, draw}}
\tikzset{klknotenB/.style={klknoten, fill=black}}
\tikzset{knotenB/.style={knoten, fill=black}}
\tikzset{markKnoten/.style={knoten, draw=red}}
\tikzset{bmarkKnoten/.style={knoten, draw=blue}}
\tikzset{gmarkKnoten/.style={knoten, draw=green!70!black}}
\tikzset{unKante/.style={ultra thick}}
\tikzset{markunKante/.style={unKante, red}}
\tikzset{bmarkunKante/.style={unKante, blue}}
\tikzset{gmarkunKante/.style={unKante, green!70!black}}
\tikzset{vmarkunKante/.style={unKante, violet!70!white}}
\tikzset{graumarkunKante/.style={unKante, black!40!white}}
\tikzset{grknoten/.style={shape=circle, ultra thick, fill = white, minimum size = 0.85cm, draw}}

\tikzset{geKante/.style={ultra thick, ->}}
\tikzset{markgeKante/.style={geKante, red}}
\tikzset{bmarkgeKante/.style={geKante, blue}}
\tikzset{gmarkgeKante/.style={geKante, green!70!black}}
\tikzset{graumarkgeKante/.style={geKante, black!40!white}}
\tikzset{miniKante/.style={very thick, -{>[scale=0.7]}, shorten >=2pt, shorten <=2pt}}
\tikzset{klKante/.style={very thick, -{>[scale=0.9]}, shorten >=3.3pt, shorten <=3.3pt}}

\tikzset{dgeKante/.style={ultra thick, <->}}
\tikzset{markdgeKante/.style={dgeKante, red}}
\tikzset{bmarkdgeKante/.style={dgeKante, blue}}
\tikzset{gmarkdgeKante/.style={dgeKante, green}}

\tikzset{coledge/.style={yellow,line width=3pt}}
\newcommand{\miniKnoten}[1]{
	\draw[fill] (#1) circle [radius=0.07];
}

\newcommand{\Knoten}[4]{
\node[knoten, #4] #1 at #2 {};
\node at #2 {#3};
}

\newcommand{\rotKnoten}[4]{
\node[markKnoten, #4] #1 at #2 {};
\node at #2 {#3};
}

\newcommand{\blauKnoten}[4]{
\node[bmarkKnoten, #4] #1 at #2 {};
\node at #2 {#3};
}

\definecolor{mygray}{gray}{.50}

\tikzset{koordinate/.style={thick, ->, draw = black}}
\tikzset{RLines/.style={ultra thick, draw = black}}
\tikzset{RLiner/.style={RLines, draw = red}}
\tikzset{RLineb/.style={RLines, draw = blue}}
\tikzset{Objective/.style={ultra thick, draw = red}}
\tikzset{point/.style={shape=circle, ultra thick, fill = black, minimum size = 0.01cm, scale=0.5, draw = black}}
\tikzset{pointr/.style={point, fill = red, draw = red}}
\tikzset{pointb/.style={point, fill = blue, draw = blue}}
\tikzset{pointg/.style={point, fill = green, draw = green}}
\tikzset{pointgr/.style={point, fill = mygray, draw = mygray}}
\tikzset{markPoints/.style={fill = yellow}}
\tikzset{hilfspfeil/.style={->, thick}}
\tikzset{zeigepfeil/.style={->, >=stealth}}
\tikzset{
	hexagon/.style={signal,signal to=east and west}
}

\tikzset{Pivotelement/.style={red}}

\tikzset{Bin/.style={thin}}

\tikzset{Factory/.style={shape=rectangle, minimum size=12pt, inner sep=3pt, fill=black, draw, very thick}}
\tikzset{FactoryR/.style={Factory, draw=red, ultra thick}}
\tikzset{Customer/.style={shape=circle, minimum size=5pt, inner sep=3pt, fill=black, draw}}

\newcommand{\Job}[5]{
	\draw[ultra thick, fill=#3] #1 rectangle node[midway](#5){#4} #2 ;
}

\newcommand{\Kreuz}[4]{
	\begin{scope}[shift={#3}, scale=#4]
		\draw[line width=#1, #2] (0,0) -- (1,1.5);
		\draw[line width=#1, #2] (1,0) -- (0,1.5);
	\end{scope}
}

\definecolor{olive}{RGB}{107,142,35} 
\tikzstyle{decision} = [ultra thick, diamond, draw=rwthblue, fill=white, text width=5em, text badly centered, node distance=3cm, inner sep=0pt]
\tikzstyle{block} = [ultra thick, rectangle, draw=rwthblue, fill=white, text width=5em, text centered, minimum height=4em]
\tikzstyle{result} = [ultra thick,rectangle, draw=rwthblue, fill=white, text width=5em, text centered, rounded corners=5mm, minimum height=4em]
\tikzstyle{line} = [ultra thick, draw=rwthblue, -latex']
\tikzstyle{cloud} = [ultra thick,draw=rwthblue, ellipse,fill=white, node distance=3cm,
minimum height=2em]

\newcommand{\Ellipse}[4]{
	\draw[unKante,#4] (#1) ellipse (#2cm and #3cm);
}

\newcommand{\Pfeil}[3]{
	\begin{scope}[shift={#1}, scale=#2]
		\draw[unKante, #3 ] (1,-3) -- (2,-3) -- (2,-3.25) -- (2.5,-2.75) -- (2,-2.25) -- (2, -2.5) -- (1,-2.5) -- (1,-3) --cycle;
	\end{scope}
}

\newcommand{\Shape}[3]{
	\draw[ultra thick, #2, pattern color = #2, pattern = #3] plot [smooth cycle] coordinates {#1};
}

\newcommand{\FirstReview}[1] {\textcolor{black}{#1}}
\usepackage{lineno}
\modulolinenumbers[5]
\usepackage[onehalfspacing]{setspace}

\usepackage[
left=2cm,
right=2cm,
top=2.5cm,
bottom=2cm,
]{geometry}

\makeatletter
\let\@afterindenttrue\@afterindentfalse
\def\ps@pprintTitle{%
	\let\@oddhead\@empty
	\let\@evenhead\@empty
	\def\@oddfoot{}%
	\let\@evenfoot\@oddfoot}
\makeatother
\makeatletter
\usepackage{environ}
\makeatletter
\newsavebox{\measure@tikzpicture}
\NewEnviron{scaletikzpicturetowidth}[1]{%
	\def\tikz@width{#1}%
	\def\tikzscale{1}\begin{lrbox}{\measure@tikzpicture}%
		\BODY
	\end{lrbox}%
	\pgfmathparse{#1/\wd\measure@tikzpicture}%
	\edef\tikzscale{\pgfmathresult}%
	\BODY
}
\makeatother

\begin{document}
\begin{frontmatter}

\title{Robust transshipment problem under consistent flow constraints}

\author{Christina B\"using}
\author{Arie M.C.A. Koster}
\author{Sabrina Schmitz\corref{mycorrespondingauthor}}

\ead{schmitz@combi.rwth-aachen.de}
\cortext[mycorrespondingauthor]{Corresponding author}
\address{Combinatorial Optimization, RWTH Aachen University, Germany}

\begin{abstract}
	In this paper, we study robust transshipment under consistent flow constraints.
	We consider demand uncertainty represented by a finite set of scenarios and characterize a subset of arcs as so-called fixed arcs.
	In each scenario, we require an integral flow that satisfies the respective flow balance constraints.
	In addition, on each fixed arc, we require equal flow for all scenarios.
	The objective is to minimize the maximum cost occurring among all scenarios.
	
	We show that the problem is strongly $\mathcal{NP}$-hard on acyclic digraphs by a reduction from the $(3,B2)$-\textsc{Sat} problem.
	Furthermore, we prove that the problem is weakly $\mathcal{NP}$-hard on series-parallel digraphs by a reduction from a special case of the \textsc{Partition} problem. 
	If in addition the number of scenarios is constant, we observe the pseudo-polynomial-time solvability of the problem. 
	We provide polynomial-time algorithms for three special cases on series-parallel digraphs.
	Finally, we present a polynomial-time algorithm for pearl digraphs.
\end{abstract}

\end{frontmatter}

\section{Introduction}
\label{Sec:Introduction}
In this paper, we consider the \textit{robust transshipment problem under consistent flow constraints} (\ProblemName{}). 
The problem is motivated by long-term decisions on transshipment that have to be made despite uncertainties in demand. 
For instance, in logistic applications the transshipment is often agreed in advance by long-term contracts with subcontractors. 
Once an agreement is signed, the transshipment needs to be performed even if demand fluctuates.  
A solution to the \ProblemName{} problem facilitates cost-efficient decision-making which is robust against demand uncertainty.

The \ProblemName{} problem is the uncapacitated version of the robust minimum cost flow problem under consistent flow constraints (\FirstProblem{}), introduced in our previous work~\cite{buesing2020robust}.
The uncapacitated version is of interest, even if logistic applications might have a limited transshipment capacity, as the limitations are often sufficiently large. 
Since the complexity results for the \FirstProblem{} problem rely on the arc capacities of the network, the question raises whether the problem is solvable in polynomial time if the capacity restrictions are neglected.
This is the case, for example, for the integral multi-commodity flow problem, which is $\mathcal{NP}$-hard in general but  solvable in polynomial time for uncapacitated networks~\cite{korte2012combinatorial}.

As in the transshipment problem~\cite{korte2012combinatorial}, we consider an uncapacitated network in the \ProblemName{} problem.
To represent demand uncertainty, we consider vertex balances for a finite number of scenarios.
Furthermore, we characterize a subset of arcs as so-called fixed arcs. 	
In each scenario, we require an integral flow that satisfies the respective flow balance constraints.
The flow on a fixed arc represents the transshipment, for example, by subcontractors. 
For this reason, on each fixed arc, we require equal flow for all scenarios to determine a transshipment that is robust to demand uncertainty. 
The objective is to minimize the maximum cost that may occur among all scenarios. 
We note that the integral requirement for the flow is necessary, even though the balances of all scenarios are integral, as Dantzig and Fulkerson’s Integral Flow Theorem~\cite{korte2012combinatorial} does not hold for the \ProblemName{} problem.   

The main contribution of this paper is summarized as follows. 
We prove that finding a feasible solution to the \ProblemName{} problem is strongly $\mathcal{NP}$-complete on acyclic digraphs, even if only two scenarios are considered that have the same unique source and unique sink.
On series-parallel \FirstReview{(SP)} digraphs, we provide the following results. 
First, we prove that the decision version of the \ProblemName{} problem is weakly $\mathcal{NP}$-complete, even if only two scenarios are considered that have the same unique source and single but different sinks.
Second, we observe on the basis of our previous work~\cite{buesing2020robust} the pseudo-polynomial-time solvability for the special case of a constant number of scenarios. 
Third, we propose a polynomial-time algorithm, also on the basis of our previous work~\cite{buesing2020robust}, for the special case that all scenarios have the same unique source and unique sink. 
Fourth, we present a polynomial-time algorithm for the special case of a unique source and parallel sinks. 
We derive the same result for the special case of parallel sources and a unique sink. 
Furthermore, we present that the algorithm is extendable for the special case of parallel sources and parallel sinks if there exists a path between each source and each sink. 
Finally, we present a polynomial-time algorithm for the special case of pearl digraphs, independent of the number of sources and sinks. 
Initial results are published in a preliminary version in the proceedings of INOC $2022$~\cite{busing2022complexity}.

The outline of this paper is as follows. 
In Section~\ref{Sec:RelatedWork}, we provide an overview of related work. 
In Section~\ref{Sec:Problemdef}, we define the problem and introduce notations. 
In Section~\ref{Sec:ComplexityAcyclic}, we analyze the complexity of the \ProblemName{} problem on acyclic digraphs.
In Section~\ref{Sec:SPDef}, we analyze the complexity of the \ProblemName{} problem on SP digraphs in general, on SP digraphs with a unique source and a unique sink, on SP digraphs with a unique source and parallel sinks, and on pearl digraphs. 
In Section~\ref{Sec:Conclusion}, we conclude our results.

%%%-----------------------------------------------------------------------------------------
\section{Related work}
\label{Sec:RelatedWork}
In the literature, a variety of logistic applications are represented by minimum cost flow (MCF) models. 
Adapted to specific applications, several extensions of the MCF problem are analyzed.
For instance, Seedig~\cite{seedig2011network} introduces the minimum cost flow with minimum quantities problem (MCFMQ). 
In addition to the requirements of the MCF problem, the flow on all outgoing arcs of the source must either take the value zero or a minimum quantity given. 
Krumke and Thielen~\cite{krumke2011minimum} generalize the MCFMQ problem such that the minimum quantity property holds for all arcs.  
Another extension of the MCF problem is, for example, the MCF version of the maximum flow problem with disjunctive constraints, studied by Pferschy and Schauer~\cite{pferschy2013maximum}. 
A maximum flow is sought whose flow is positive for at least one arc of every arc pair included in a predetermined arc set. 
We note that the concepts of these two problems are transferable to the concept of the \ProblemName{} problem. 
We could represent a minimum transshipment or at least one out of two transshipments performed by subcontractors. 
Due to the equal flow requirement of the \ProblemName{} problem, we focus in the following on literature with equal flow requirements. 

There are several extensions to the maximum flow (MF) and MCF problem which consider equal flow requirements on specified arc sets. 
For instance, the integral flow with homologous arcs problem (\textsc{homIF}), introduced by Sahni~\cite{sahni1974computationally}, aims at a maximum flow whose flow is equal on specified arcs. 
Sahni proves the $\mathcal{NP}$-hardness of the problem by a reduction from the \textsc{Non-Tautology} problem.
The MCF version of the \textsc{homIF} problem is known as the (integer) equal flow problem (\textsc{EF}).
By standard techniques, the complexity results of the \textsc{homIF} problem are transferred to the \textsc{EF} problem~\citep{ahuja1988network}. 
Meyers and Schulz~\cite{meyers2009integer} discuss the uncapacitated version of the \textsc{EF} problem. 
For instance, they prove the strong $\mathcal{NP}$-hardness of the problem by a reduction from the \textsc{Exact Cover by $3$-sets} problem, even in the case of a single source and sink. 
In addition, they show that there exists no $2^{n(1-\epsilon)}$-approximation algorithm for any fixed $\epsilon>0$ (on a digraph with $n$ vertices), even if a nontrivial solution is guaranteed to exist, unless $\mathcal{P}=\mathcal{NP}$.
There are further studies for both the \textsc{homIF} and \textsc{EF} problem in the literature~\citep{calvete2003network,meyers2009integer,morrison2013network,10.1007/3-540-45655-4_55}.  
Ali et al.~\cite{ali1988equal} investigate a special case of the \textsc{EF} problem where all sets have cardinality two. 
An integral MCF is sought whose flow is equal on a predetermined set of arc pairs.
The problem finds application in, for example, crew scheduling~\citep{carraresi1984network}.
Therefore, Ali et al.~present a heuristic algorithm based on Lagrangian relaxation.  
Meyers and Schulz~\cite{meyers2009integer} refer to this special case as paired integer equal flow problem (\textsc{pEFP}). 
They also consider the uncapacitated version of the \textsc{pEFP} problem and prove the strong $\mathcal{NP}$-hardness. 
Furthermore, they prove that there exists no $2^{n(1-\epsilon)}$-approximation algorithm for any fixed $\epsilon>0$ (on a digraph with $n$ vertices), unless $\mathcal{P}=\mathcal{NP}$. 
The statement holds true even if a nontrivial solution is guaranteed to exist.

Unlike the research referenced above, we do not consider demand and supply for only one scenario in the \ProblemName{} problem. 
Like in the \FirstProblem{} problem~\cite{buesing2020robust}, we consider several scenarios to represent demand uncertainty. 
We stress that the equal flow requirements are only of importance for more than one scenario.
The flow on a fixed arc has to be equal among all scenarios.
In turn, the flow on two different fixed arcs may differ in one scenario.
In our previous study~\cite{buesing2020robust}, we prove the $\mathcal{NP}$-hardness of the \FirstProblem{} problem on acyclic and SP digraphs.  
Furthermore, we present a polynomial-time algorithm for a special case on SP digraphs. 

To the best of our knowledge, equal flow requirements and demand uncertainty are combined in our previous study~\cite{buesing2020robust} for the first time.  
Demand uncertainty is frequently studied in the context of (uncapacitated) network design. 
Three examples are provided as follows.
Guti\'{e}rrez et al.~\cite{gutierrez1996robustness} present a robustness approach to uncapacitated network design problems.
To solve the problem they develop algorithms based on Benders decomposition methodology.
Lien et al.~\cite{lien2011efficient} provide an efficient and robust design for transshipment networks by chain configurations. 
Holmberg and Hellstrand~\cite{holmberg1998solving} concentrate on finding an optimal solution to the uncapacitated network design problem for commodities with a single source and sink by a Lagrangian heuristic within a branch-and-bound framework.

%%%%%%%%%%-------------------------------------------------------------------------------------
\section{Definition \& notations} 
\label{Sec:Problemdef}
The \ProblemName{} problem is the uncapacitated version of the \FirstProblem{} problem. 
We define the problem on the basis of our previous work~\cite{buesing2020robust}.
Let \textit{digraph} $G=(V,A)$ be given with vertex set $V$ and arc set $A$. 
The set of arcs $A$ is divided into two disjoint sets $\fix$ and $\free$, termed \textit{fixed} and \textit{free arcs}, respectively. 
\FirstReview{If not explicitly defined, }we specify the sets of vertices, arcs, fixed arcs, and free arcs of a digraph $G$ by $V(G)$, $A(G)$, $\fix(G)$, and $\free(G)$, respectively.  
Let arc \textit{cost} $c:A\rightarrow \mathbb{Z}_{\geq 0}$ be given.
The demand uncertainty is represented by the finite set of scenarios $\Lambda$. 
For every scenario $\lambda\in \Lambda$, vertex \textit{balances} $b^\lambda:V\rightarrow \mathbb{Z}$ with $\sum_{v\in V}b^\lambda(v)=0$ are given that define the supply and demand realizations, denoted by $\boldsymbol{b}=(b^1,\ldots,b^{|\Lambda|})$.
A vertex with a positive or negative balance is termed \textit{source} or \textit{sink}, respectively. 
In general, the source (sink) vertices do not necessarily have to be the same in all scenarios.
If there does not exist a path between each two sources (sinks), we say that the problem has \textit{parallel sources (sinks)}. 
If each scenario has only one vertex with a positive (negative) balance, we refer to this source (sink) as single source (sink). 
If the single sources (sinks) of all scenarios are defined by the same vertex, we say that the problem has a \textit{unique source (sink)}. 
Overall, we obtain the \textit{network} $(G=(V,A=\fix \cup \free),c,\boldsymbol{b})$. 

For a single scenario $\lambda\in \Lambda$, a \textit{$b^\lambda$-flow} in digraph $G$ is defined by a function $f^\lambda : A\rightarrow  \mathbb{Z}_{\geq 0}$ that satisfies the \textit{flow balance constraints}
\[
\sum_{a=(v,w)\in A}f^\lambda(a)-\sum_{a=(w,v)\in A}f^\lambda(a)=b^\lambda(v)
\]
at every vertex $v\in V$. 
The cost of a $b^\lambda$-flow $f^\lambda$ is defined by $$c(f^\lambda)=\sum_{a\in A} c(a)\cdot f^\lambda(a).$$
For the entire set of scenarios $\Lambda$, %a \textit{robust $\boldsymbol{b}$-flow} is defined 
a \textit{robust $\boldsymbol{b}$-flow} $\boldsymbol{f}=(f^1,\ldots,f^{|\Lambda|})$ is defined by a $|\Lambda|$-tuple of integral $b^\lambda$-flows $f^\lambda: A\rightarrow  \mathbb{Z}_{\geq 0} $ that satisfy the \textit{consistent flow constraints} $f^\lambda(a)=f^{\lambda^\prime}(a)$ on all fixed arcs $a\in A^\text{fix}$ for all scenarios $\lambda, \lambda^\prime \in \Lambda$. 
The cost of a robust $\boldsymbol{b}$-flow $\boldsymbol{f}$ is defined by
$$c(\boldsymbol{f})=\max_{\lambda\in \Lambda}c(f^\lambda).$$
Finally, the \ProblemName{} problem is \FirstReview{defined} as follows.
\begin{definition}[\ProblemName{} problem]
	Given a network $(G=(V,A=A^{\text{fix}}\cup A^{\text{free}}),c,\boldsymbol{b})$, the \textit{robust transshipment problem under consistent flow constraints} aims at a robust $\boldsymbol{b}$-flow $\boldsymbol{f}=(f^1,\ldots, \allowbreak f^{|\Lambda|})$ of minimum cost.
\end{definition}
We note that in the case of a single scenario, i.e., $|\Lambda|=1$, the \ProblemName{} problem corresponds to the transshipment problem~\cite{korte2012combinatorial}.  
Analogous to the \FirstProblem{} problem, we stress that the Integral Flow Theorem of Dantzig and Fulkerson~\cite{korte2012combinatorial} does not hold for the \ProblemName{} problem. 
Although the balances are integral, in general, the solution of the continuous relaxation of the \ProblemName{} problem is not integral, as shown in the following example.
\begin{example}
	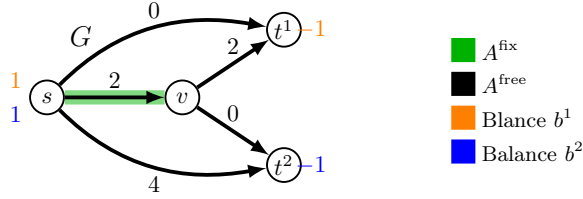
\begin{figure}
		\centering\scalebox{0.9}{\begin{tikzpicture}

	\Knoten{(1)}{(0,0)}{$s$}{thick}
	\Knoten{(2)}{(2,0)}{$v$}{thick}
	\Knoten{(3)}{(3.5,1)}{$t^1$}{thick}
	\Knoten{(4)}{(3.5,-1)}{$t^2$}{thick}
	\node at (0.5,0.9){\large$G$}; 

	\draw[unKante,line width=6,  opacity=0.5,\gruen] (1) -- (2);
	\draw[unKante, -latex] (1) -- node[above]{$2$}(2); 
	\draw[unKante, -latex] (2) -- node[above]{$2$}(3); 
	\draw[unKante, -latex] (2) -- node[above]{$0$}(4); 
	\draw[unKante, -latex] (1) to [out=45, in=170] node [above]{$0$} (3); 
	\draw[unKante, -latex] (1) to [out=-45, in=190] node[below]{$4$}  (4); 

	%Balances
	\node[orange, left] at ($(1)+(-0.25,0.25)$){$1$}; 
	\node[blue, left] at ($(1)-(0.25,0.25)$){$1$}; 
	\node[orange, left] at ($(3)+(0.75,0)$){$-1$}; 
	\node[blue, left] at ($(4)+(0.75,0)$){$-1$}; 

% Legende
\begin{scope}[shift={(8.5,3.55)}]
	\begin{scope}[shift={(-2.5,-3)}]
		\Job{(0,0)}{(0.3,0.3)}{\gruen, \gruen}{}{}
		\node[right] at (0.3,0.15){\small$A^{\text{fix}}$}; 
	\end{scope}
	\begin{scope}[shift={(-2.5,-3.5)}]
		\Job{(0,0)}{(0.3,0.3)}{black, black}{}{}
		\node[right] at (0.3,0.15){\small$A^{\text{free}}$}; 
	\end{scope}
	\begin{scope}[shift={(-2.5,-4)}]
		\Job{(0,0)}{(0.3,0.3)}{orange, orange}{}{}
		\node[right] at (0.3,0.15){\small$\text{Blance } b^1$}; 
	\end{scope}
	\begin{scope}[shift={(-2.5,-4.5)}]
		\Job{(0,0)}{(0.3,0.3)}{blue, blue}{}{}
		\node[right] at (0.3,0.15){\small$\text{Balance } b^2$}; 
	\end{scope}
\end{scope}
\end{tikzpicture}}
		\caption{If the integral flow requirement is neglected, a non-integral robust $\boldsymbol{b}$-flow is the only optimal solution}\label{RFexample:nonIntegralityUncapacitated}
	\end{figure}
	For a set of two scenarios $\Lambda = \{1, 2\}$, let a \ProblemName{} instance $(G,c,\fett{b})$ be given. 
	Digraph $G$, the arc cost $c$, and the non-zero balances $\fett{b}$ are visualized in Figure~\ref{RFexample:nonIntegralityUncapacitated}. 
	An optimal integral robust $\boldsymbol{b}$-flow $\fett{f} = ( f^1, f^2)$ is given for instance as follows.
	The first scenario flow $f^1$ sends one unit along arc $(s,t^1)$ and the second scenario flow $f^2$ sends one unit along arc $(s,t^2)$. 
	The cost is $c(\fett{f}) = 4$ as $c( f^1) = 0$ and $c( f^2) = 4$ hold. 
	However, if the integral flow requirement is neglected, an optimal robust $\boldsymbol{b}$-flow $\fett{\tilde{f}}= ( \tilde{f}^1, \tilde{f}^2)$ is given as follows. 
	The first scenario flow $\tilde{f}^1$ sends half a unit each along arc $(s,t^1)$ and path $svt^1$. 
	The second scenario flow $\tilde{f}^2$ sends half a unit each along arc $(s,t^2)$ and path $svt^2$. 
	The cost is $c(\fett{\tilde{f}}) = 3$ as $c( \tilde{f}^1) = 0.5\cdot 0 + 0.5 \cdot (2+2) = 2$ and $c( \tilde{f}^2) = 0.5\cdot 4 + 0.5 \cdot (2+0) = 3$ hold. 
\end{example}
If integral flows were not required, the \ProblemName{} problem could be solved in polynomial time by a linear program. 
Before concluding this section, we note the following three results on the basis of our previous work~\cite{buesing2020robust}. 
First, if the values to be sent along fixed arcs are given, the \ProblemName{} problem can be solved in polynomial time by $|\Lambda|$ separate transshipment problems. 
Second, if the number of fixed arcs is constant, the \ProblemName{} problem can be solved in polynomial time by a mixed integer program with a corresponding constant number of integer variables. 
Third, if the \ProblemName{} problem has a unique source and a unique sink, the cost of an optimal robust flow is determined by the maximum of the cost of the scenarios with the maximum and minimum supply. 
For more details and proofs, we refer to our previous work~\cite{buesing2020robust}.

\section{Complexity for acyclic digraphs}	
\label{Sec:ComplexityAcyclic}
In this section, we analyze the complexity of the \ProblemName{} problem on acyclic digraphs. 
We perform a reduction from the strongly $\mathcal{NP}$-complete $(3,B2)$-\textsc{Sat} problem, introduced by Berman et al.~\citep{berman2004approximation}.
The $(3,B2)$-\textsc{Sat} problem is a special case of the $3$-\textsc{Sat} problem~\cite{johnson1979computers}, where every literal occurs exactly twice.
We use the notation $[n]:=\{1,\ldots,n\}$.
\begin{theorem}\label{Theorem:ReductionAcyclicGraphs2}
	Deciding whether or not a feasible solution exists to the \ProblemName{} problem on acyclic digraphs is strongly $\mathcal{NP}$-complete, even if only two scenarios are considered that have the same unique source and unique sink.
\end{theorem}
\begin{proof}
	The \ProblemName{} problem is contained in $\mathcal{NP}$ as we can check in polynomial time whether the flow balance and consistent flow constraints are satisfied for every scenario. 
	Let $\{x_1,\ldots,x_n\}$ be the set of variables and $C_1,\ldots,C_m$ be the clauses of the $(3,B2)$-\textsc{Sat} instance $\mathcal{I}$. 
	For a set of two scenarios $\Lambda=\{1,2\}$, we construct a \ProblemName{} instance $\widetilde{\mathcal{I}}= (G, c, \boldsymbol{b})$. 
	Figure~\ref{ReductionSingleSourceSingleSinkWithoutCapacities} visualizes a \ProblemName{} instance corresponding to an example of a $(3,B2)$-\textsc{Sat} instance with four clauses and three variables. 
	\begin{figure*}[h]
		\begin{adjustbox}{max width=1.\textwidth}
			\input{Figures/ReductionSingleSourceSingleSinkWithoutCapacities}
		\end{adjustbox}
		\caption{Construction of \ProblemName{} instance $\widetilde{\mathcal{I}}$}
		\label{ReductionSingleSourceSingleSinkWithoutCapacities}
	\end{figure*}
	In general, the \ProblemName{} instance is based on a digraph $G=(V,A)$ defined as follows.
	% % % % % % %Vertices
	The vertex set $V$ includes one vertex $v_i$ per variable $x_i$, $i\in[n]$, one dummy vertex $v_{n+1}$, and one vertex $u_j$ per clause $C_j$, $j\in[m]$. 
	For every literal $x_i$ ($\overline{x}_i$), $i\in[n]$, four auxiliary vertices $w^\ell_i$ ($\overline{w}^\ell_i$), $\ell\in[4]$ are included.
	Furthermore, set $V$ includes one auxiliary vertex $t$ and vertices $r_\ell$ for $\ell\in [2(2n+2)]$.
	% % % % % %Arcs
	Arc set $A$ includes arcs that connect two successive variable vertices $v_i$, $v_{i+1}$, $i\in[n]$ by two parallel paths $p_i$ and $\overline{p}_i$ defined along the auxiliary vertices, i.e., $p_i=v_i w^1_i w^2_i w^3_i w^4_i v_{i+1}$ and $\overline{p}_i=v_i\overline{w}^1_i \overline{w}^2_i \overline{w}^3_i \overline{w}^4_i v_{i+1}$ for $i\in[n]$. 
	Path $p_i$ represents the positive literal $x_i$ and path $\overline{p}_i$ the negative literal $\overline{x}_i$ of instance $\mathcal{I}$.
	As each literal occurs exactly twice, we identify two arcs of paths $p_i$ and $\overline{p}_i$ each with the literals.
	More precisely, let $x^k_i$ ($\overline{x}^k_i$) denote literal $x_i$ ($\overline{x}_i$), $i\in [n]$ which occurs the $k$-th time, $k\in [2]$ in the formula. 
	Arc $(w^{2k-1}_i,w^{2k}_i)$ ($(\overline{w}^{2k-1}_i,\overline{w}^{2k}_i)$) corresponds to literal $x^k_i$ ($\overline{x}^k_i$), $i\in [n]$, $k\in [2]$, referred to as literal arc. 
	Using this correspondence, we add arc $(w^{2k}_i, u_{j})$ ($(\overline{w}^{2k}_i, u_{j})$) for every literal $x^k_i$ ($\overline{x}^k_i$), $i\in[n]$, $k\in [2]$ included in clause $C_{j}$, $j\in[m]$. 
	% % % % %Path
	In the next step, we create a path $\widetilde{p}$ from vertex $v_{n+1}=:r_0$ along vertices $r_\ell$, $\ell\in [2(2n+2)-1]$ to vertex $z:=r_{2(2n+2)}$. 
	% % % % %Hilfskanten
	Before introducing the last arcs included in arc set $A$, we identify all literal arcs and every second arc of path $\widetilde{p}$ as the only fixed arcs in the network, i.e., 
	\begin{align*}
		A^\text{fix}= \big\{(w^\ell_i ,w^{\ell+1}_i), \linebreak[3] (\overline{w}^\ell_i , \overline{w}^{\ell+1}_i) \mid \ell\in \{1,3\},\ i\in [n]\big\}
		\cup \big\{(r_\ell, r_{\ell+1})\mid \ell\in \{1,3,\ldots,2(2n+2)-1\}\big\}.
	\end{align*}
	We add arcs that connect vertex $v_1$ with every literal arc and every literal arc with auxiliary vertex $t$, i.e., $(v_1,w^\ell_i)$, $(v_1,\overline{w}^\ell_i)$ for $\ell\in\{1,3\}$ and $(w^\ell_i,t)$, $(\overline{w}^\ell_i,t)$ for $\ell\in\{2,4\}$.
	\FirstReview{The clause vertices are connected} with the first $m$ fixed arcs of path $\widetilde{p}$, i.e., $(u_j, r_{2j-1})$ for all $ j\in [m]$. 
	The auxiliary vertex $t$ is connected with the successive $2n-m$ fixed arcs of path $\widetilde{p}$ by $(t,r_{2\ell-1})$, $\ell\in \{m+1,\ldots,2n\}$.
	We add arcs $(v_1,w^4_n)$, $(v_1,\overline{w}^4_n)$, $(w^4_n, r_{4n+1})$, and $(\overline{w}^4_n, r_{4n+3})$.
	Finally, we connect all $2n+2$ fixed arcs of path $\widetilde{p}$ with vertex $z$, i.e., $(r_\ell,z)$ for all $\ell\in \{2,4,\ldots, 4n+2\}$. 
	We set the cost $c\equiv0$ and define the balances $\boldsymbol{b}=(b^1,b^2)$ by
	\begin{align*}
		b^1(v)
		= \left\{ \begin{array}{ll}
			1 			& \mbox{if $v=v_1$}, \\ 
			-1 			& \mbox{if $v=z$}, \\ 
			0 			& \mbox{otherwise},  \end{array} \right.
		& \ 
		b^2(v)= \left\{ \begin{array}{ll}
			2n+2 			& \mbox{if $v=v_1$}, \\ 
			-(2n+2) 	& \mbox{if $v=z$}, \\
			0 			& \mbox{otherwise}.  \end{array} \right.
	\end{align*}
	Vertices $v_1$ and $z$ specify the unique source and unique sink, respectively.
	Overall, we obtain a feasible \ProblemName{} instance $\widetilde{\mathcal{I}}= (G, c, \boldsymbol{b})$ which is constructed in polynomial time. Hence, it remains to show that $\mathcal{I}$ is a Yes-instance if and only if a feasible robust $\boldsymbol{b}$-flow exists for instance $\widetilde{\mathcal{I}}$.
	
	Let $x_1,\ldots,x_n$ be a satisfying truth assignment for instance $\mathcal{I}$. 
	We define the first scenario flow $f^1$ of instance $\widetilde{\mathcal{I}}$ as follows
	\begin{align*}
		f^1(a)=\left\{ \begin{array}{ll}
			1 &\text{for all } a\in A(p_i)	 \mbox{ if $x_i=\textsc{True}$}, \\
			1 &\text{for all } a\in A(\overline{p}_i)	 \mbox{ if $x_i=\textsc{False}$}, \\
			1 &\text{for all } a\in A(\widetilde{p}),\\
			0 &\mbox{otherwise}.  \end{array} \right.
	\end{align*}
	Flow $f^1$ uses either path $p_i$ or $\overline{p}_i$, $i\in [n]$ to send one unit from source $v_1$ to vertex $v_{n+1}$. 
	The unit is forwarded from vertex $v_{n+1}$ to sink $z$ along path $\widetilde{p}$.
	As $x_1,\ldots,x_n$ is a satisfying truth assignment, there exists one  designated verifying literal $x^k_i$ or $\overline{x}^k_i$, $i\in [n]$, $k\in [2]$ for each clause $C_j$, $j\in [m]$. 
	Using this, we define the first part of the second scenario flow $f^2$ as follows
	\begin{align*}
		f^2(a)=
		\begin{cases}
			1 		&\text{for all } a\in A(q^k_i) \text{ with }q^k_i=v_1w^{2k-1}_i w^{2k}_iu_{j}r_{2j-1} r_{2j}z 
			\mbox{ if $x^k_i\in C_{j}$ is designated as verifying literal}, \\
			1 		&\text{for all } a\in A(q^k_i) \text{ with }q^k_i=v_1\overline{w}^{2k-1}_i\overline{w}^{2k}_i u_{j}r_{2j-1} r_{2j}z
			\mbox{ if $\overline{x}^k_i\in C_{j}$ is designated as verifying literal},\\
			0 		&\mbox{otherwise}. 
		\end{cases}
	\end{align*}   
	Flow $f^2$ sends $m$ units from the source $v_1$ to the clause vertices $u_1,\ldots,u_m$ along the literal arcs corresponding to the designated verifying literals. 
	The $m$ units are forwarded along the subsequent fixed arcs to sink $z$. 
	Further, we define the second part of the second scenario flow $f^2$ that sends $2n-m$ units along the remaining literal arcs to vertex $t$. 
	The flow is forwarded to sink $z$ such that we set
	\begin{align*}
		f^2(a)=1 \mbox{ for all } 
		\begin{cases}
			a\in A(q^k_i) 	&\text{with }q^k_i=v_1w^{2k-1}_iw^{2k}_i t 
			\mbox{ if $x^k_i=\textsc{True}$ and $x^k_i$ is not designated as a verifying literal}, \\
			a\in A(q^k_i) 	&\text{with }q^k_i=v_1\overline{w}^{2k-1}_i\overline{w}^{2k}_i t	
			\mbox{ if $x^k_i=\textsc{False}$ and $x^k_i$ is not designated as a verifying literal},  \\
			a\in A(q_\ell) 	&\text{with }q_\ell=tr_{2\ell-1} r_{2\ell}z,\ell\in \{m+1,\ldots, 2n\}.
		\end{cases}
	\end{align*}
	Finally, one further unit is sent along path $v_1w^4_nr_{4n+1}r_{4n+2}z$ and one along path $v_1\overline{w}^4_nr_{4n+3}z$. 
	We have constructed a feasible robust $\boldsymbol{b}$-flow $\boldsymbol{f}=(f^1,f^2)$ for \ProblemName{} instance $\widetilde{\mathcal{I}}$. 
	
	Conversely, let $\boldsymbol{f}=(f^1,f^2)$ be a feasible robust $\boldsymbol{b}$-flow for \ProblemName{} instance $\widetilde{\mathcal{I}}$.
	Flows $f^1$ and $f^2$ send one and $2n+2$ units from source $v_1$ to sink $z$, respectively.
	By construction of the network, the only option to reach the sink requires the usage of at least two fixed arcs, namely one literal arc and one fixed arc of path $\widetilde{p}$ (except for the two paths $v_1w^4_nr_{4n+1}r_{4n+2}z$ and $v_1\overline{w}^4_nr_{4n+3}z$ which include each only one fixed arc of path $\widetilde{p}$). 
	Due to the integral flow $f^1$ which sends one unit within the acyclic digraph, it holds $f^1(a)= f^2(a)\in \{0,1\}$ for all fixed arcs $a\in A^\text{fix} $. 
	Consequently, flows $f^1$ and $f^2$ use at least $4n+2$ fixed arcs to meet the demand of flow $f^2$.
	To use the required $4n+2$ fixed arcs in the first scenario, flow $f^1$ sends the unit along either path $p_i$ or $\overline{p}_i$ for all $i\in[n]$ (but due to the integral requirement and the acyclic construction never both simultaneously) and subsequently along path $\widetilde{p}$. 
	If flow $f^1$ sends the unit along path $p_i$, $i\in[n]$, we set $x_i=\textsc{True}$. 
	If flow $f^1$ sends the unit along path $\overline{p}_i$, $i\in[n]$, we set $x_i=\textsc{False}$.  
	To use the required $4n+2$ fixed arcs in the second scenario, flow $f^2$ sends one unit via each clause vertex and $2n-m$ units via vertex $t$. 
	Depending on the first scenario flow, flow $f^2$ sends one unit along either path $v_1w^\ell_iw^{\ell+1}_iu_{j}r_{2j-1}r_{2j}z$ or $v_1\overline{w}^\ell_i\overline{w}^{\ell+1}_iu_{j}r_{2j-1}r_{2j}z$, $\ell\in \{1,3\}$, $i\in [n]$ for all $j\in[m]$ (but never both simultaneously) due to the consistent flow constraints.
	In the former case, clause $C_{j}$ is verified due to the previous assignment $x_i=\textsc{True}$ induced by flow $f^1$ and the fact that $x_i\in C_{j}$ holds. 
	In the latter case, clause $C_{j}$ is verified due to the previous assignment $x_i=\textsc{False}$ induced by flow $f^1$ and the fact that $\overline{x}_i\in C_{j}$ holds. 
	Two extra units are sent along paths $v_1w^4_nr_{4n+1}r_{4n+2}z$ and $v_1\overline{w}^4_nr_{4n+3}z$ which use the last two fixed arcs of path $\widetilde{p}$. 
	The two extra units sent are needed, otherwise there might exist a feasible robust flow whose second scenario flow sends a unit along path $v_1w^3_nw^4_nv_{n+1}r_1r_2z$ or $v_1\overline{w}^3_n\overline{w}^4_nv_{n+1}r_1r_2z$ which in turn allows one unsatisfied clause. 
	Overall, $x_1,\ldots,x_n$ is a satisfying truth assignment for \mbox{instance $\mathcal{I}$}.
\end{proof}
	Before concluding this section, we note the following. 
	In our previous work~\cite{buesing2020robust}, we use a similar construction to prove the strong $\mathcal{NP}$-completeness for the \FirstProblem{} problem on acyclic digraphs. 
	However, the result is not transferable without adjustments, as the arc capacities are indispensable for the construction of the reduction.  
	In the \FirstProblem{} problem, we control that one flow unit is sent via every clause vertex by means of arc capacities.
	In the \ProblemName{} problem, we guarantee this by means of successive fixed arcs included in an additional integrated path $\widetilde{p}$ in combination with a scenario in which only one unit is sent. 
\section{Complexity for SP digraphs}	
\label{Sec:SPDef}
In this section, we analyze the complexity of the \ProblemName{} problem on SP digraphs.
In Section~\ref{Subsec:SP-MultipleSS}, we prove that the problem is in general weakly $\mathcal{NP}$-com\-plete.
In Section~\ref{Subsubsec:UniqueSourceUnqiueSink}, we provide a polynomial-time algorithm for the special case of networks with a unique source and a unique sink.
In Section~\ref{Subsec:UniqueSourceParallelSinks}, we provide a polynomial-time algorithm for the special case of networks with a unique source and parallel sinks or parallel sources and a unique sink.
In Section~\ref{Subsubsec:PearlDigraphs}, we present a polynomial-time algorithm for the special case of pearl digraphs.

Based on the edge SP multi-graphs definition of Valdes et al.~\cite{valdes1982recognition}, we define SP digraphs as follows.
\begin{definition}[SP digraph]\label{Def:SeriesParallelGraphs}
An \textit{SP digraph} is recursively defined as follows.
\begin{itemize}
\item[1.] An arc $(\sourceSPGraph,\sinkSPGraph)$ is an SP digraph with \textit{origin} $\sourceSPGraph$ and \textit{target} $\sinkSPGraph$.
\item[2.] Let $G_1$ with origin $\sourceSPGraph_1$ and target $\sinkSPGraph_1$ and $G_2$ with origin $\sourceSPGraph_2$ and target $\sinkSPGraph_2$ be SP digraphs. 
The digraph that is constructed by one of the following two compositions of SP digraphs $G_1$ and $G_2$ is itself an SP digraph.
\begin{itemize}
\item[a)] The \textit{series composition} $G$ of two SP digraphs $G_1$ and $G_2$ is the digraph obtained by contracting target $\sinkSPGraph_1$ and origin $\sourceSPGraph_2$. The origin of digraph $G$ is then $\sourceSPGraph_1$ (becoming $\sourceSPGraph$) and the target is $\sinkSPGraph_2$ (becoming $\sinkSPGraph$).
\item[b)] The \textit{parallel composition} $G$ of two SP digraphs $G_1$ and $G_2$ is the digraph obtained by contracting origins $\sourceSPGraph_1$ and $\sourceSPGraph_2$ (becoming $\sourceSPGraph$) and contracting targets $\sinkSPGraph_1$ and $\sinkSPGraph_2$ (becoming $\sinkSPGraph$). The origin of digraph $G$ is $\sourceSPGraph$, and the target is $\sinkSPGraph$.
\end{itemize}
\end{itemize}
\end{definition}
The series and parallel compositions are illustrated in Figure~\ref{fig:Spgraphs}. 
In general, SP digraphs are multi-digraphs with one definite origin and one definite target. 
In the following, we denote the origin and target of an SP digraph $G$ by $o_G$ and $q_G$, respectively. 
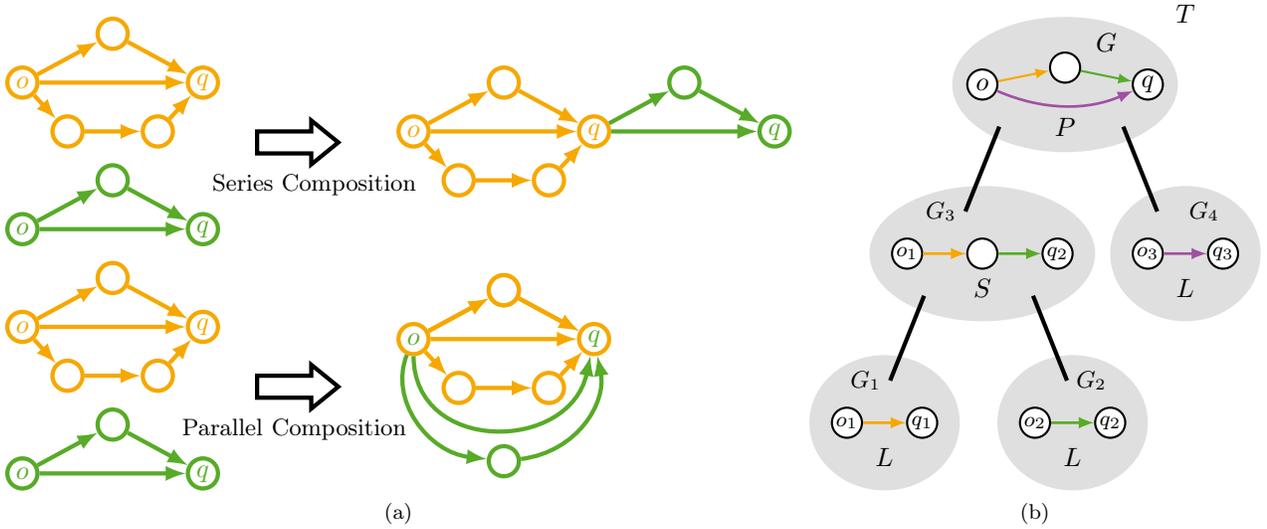
\begin{figure*}
	\begin{adjustbox}{max width=1\textwidth, 
		}
		\subfloat[]{\label{fig:Spgraphs}	\begin{tikzpicture}[xscale=0.8, yscale=0.65, bmarkKnoten/.style={knoten, draw=rwth-blue}, smarkKnoten/.style={knoten, draw=black},markKnoten/.style={knoten, draw=rwth-lred},knoten/.style={shape=circle, ultra thick, fill = white, minimum size = 0.5cm, draw},geKante/.style={ultra thick, -latex}]

\begin{scope}[shift={((0,-10))} ]	
\rotKnoten{(p11)}{(0,0)}{\textcolor{rwth-orange}{$\sourceSPGraph$}}{minimum size=0.4cm, draw=rwth-orange}
\rotKnoten{(p12)}{(0.75,-1)}{}{minimum size=0.4cm, draw=rwth-orange}
\rotKnoten{(p13)}{(2.25,-1)}{}{minimum size=0.4cm, draw=rwth-orange}
\rotKnoten{(p14)}{(1.5,1)}{}{minimum size=0.4cm, draw=rwth-orange}
\rotKnoten{(p15)}{(3,0)}{\textcolor{rwth-orange}{$\sinkSPGraph$}}{minimum size=0.4cm, draw=rwth-orange}

\begin{scope}[shift={((0,-3))} ]
	\blauKnoten{(p21)}{(0,0)}{\textcolor{rwth-green}{$\sourceSPGraph$}}{minimum size=0.4cm, draw=rwth-green}
	\blauKnoten{(p22)}{(1.5,1)}{}{minimum size=0.4cm, draw=rwth-green}
	\blauKnoten{(p23)}{(3,0)}{\textcolor{rwth-green}{$\sinkSPGraph$}}{minimum size=0.4cm, draw=rwth-green}
\end{scope}

% Kanten im roten Graph
\draw[geKante, rwth-orange] (p11) to (p12);
\draw[geKante, rwth-orange] (p11) to (p14);
\draw[geKante, rwth-orange] (p12) to (p13);
\draw[geKante, rwth-orange] (p13) to (p15);
\draw[geKante, rwth-orange] (p14) to (p15);
\draw[geKante, rwth-orange] (p11) to (p15);

% Kanten im blauen Graph
\draw[geKante, rwth-green] (p21) to (p22);
\draw[geKante, rwth-green] (p21) to (p23);
\draw[geKante, rwth-green] (p22) to (p23);

\Pfeil{(3,1.25)}{0.9}{}
\node[right] at (2.5,-2.1){\small Parallel Composition} ;

% Parallelgraph
\begin{scope}[shift={(6.5,-0.25)}]
	\rotKnoten{(p31)}{(0,0)}{\textcolor{rwth-green}{$\sourceSPGraph$}}{minimum size=0.4cm, draw=rwth-orange}
	\rotKnoten{(p32)}{(0.75,-1)}{}{minimum size=0.4cm, draw=rwth-orange}
	\rotKnoten{(p33)}{(2.25,-1)}{}{minimum size=0.4cm, draw=rwth-orange}
	\rotKnoten{(p34)}{(1.5,1)}{}{minimum size=0.4cm, draw=rwth-orange}
	\rotKnoten{(p35)}{(3,0)}{\textcolor{rwth-green}{$\sinkSPGraph$}}{minimum size=0.4cm, draw=rwth-orange}
	\rotKnoten{(p36)}{(1.5,1)}{}{minimum size=0.4cm, draw=rwth-orange}
	\blauKnoten{(p37)}{(1.5,-2.5)}{}{minimum size=0.4cm, draw=rwth-green}
\end{scope}

\draw[geKante, rwth-orange] (p31) to (p32);
\draw[geKante, rwth-orange] (p31) to (p34);
\draw[geKante, rwth-orange] (p32) to (p33);
\draw[geKante, rwth-orange] (p33) to (p35);
\draw[geKante, rwth-orange] (p34) to (p35);
\draw[geKante, rwth-orange] (p31) --(p35);
\draw[geKante, rwth-green] (p31) to [bend right, in=260, out=-90, looseness=1.75](p35);

\draw[geKante, rwth-orange] (p31) to (p36);
\draw[geKante, rwth-orange] (p36) to (p35);

\draw[geKante, rwth-green] (p31) to [bend right, in=230, out=-50] (p37);
\draw[geKante, rwth-green] (p37) to [bend right, in=220, out=-50]  (p35);
\end{scope}
%%%%%%%%%%%%

\begin{scope}[shift={(0,-5)}]
\rotKnoten{(p11)}{(0,0)}{\textcolor{rwth-orange}{$\sourceSPGraph$}}{minimum size=0.4cm, draw=rwth-orange}
\rotKnoten{(p12)}{(0.75,-1)}{}{minimum size=0.4cm, draw=rwth-orange}
\rotKnoten{(p13)}{(2.25,-1)}{}{minimum size=0.4cm, draw=rwth-orange}
\rotKnoten{(p14)}{(1.5,1)}{}{minimum size=0.4cm, draw=rwth-orange}
\rotKnoten{(p15)}{(3,0)}{\textcolor{rwth-orange}{$\sinkSPGraph$}}{minimum size=0.4cm, draw=rwth-orange}

\begin{scope}[shift={((0,-3))} ]
	\blauKnoten{(p21)}{(0,0)}{\textcolor{rwth-green}{$\sourceSPGraph$}}{minimum size=0.4cm, draw=rwth-green}
	\blauKnoten{(p22)}{(1.5,1)}{}{minimum size=0.4cm, draw=rwth-green}
	\blauKnoten{(p23)}{(3,0)}{\textcolor{rwth-green}{$\sinkSPGraph$}}{minimum size=0.4cm, draw=rwth-green}
\end{scope}

% Kanten im roten Graph
\draw[geKante, rwth-orange] (p11) to (p12);
\draw[geKante, rwth-orange] (p11) to (p14);
\draw[geKante, rwth-orange] (p12) to (p13);
\draw[geKante, rwth-orange] (p13) to (p15);
\draw[geKante, rwth-orange] (p14) to (p15);
\draw[geKante, rwth-orange] (p11) to (p15);
%
% Kanten im blauen Graph
\draw[geKante, rwth-green] (p21) to (p22);
\draw[geKante, rwth-green] (p21) to (p23);
\draw[geKante, rwth-green] (p22) to (p23);

\begin{scope}[shift={((6.5,-1))} ]
\rotKnoten{(s31)}{(0,0)}{\textcolor{rwth-orange}{$\sourceSPGraph$}}{minimum size=0.4cm, draw=rwth-orange}
\rotKnoten{(s32)}{(0.75,-1)}{}{minimum size=0.4cm, draw=rwth-orange}
\rotKnoten{(s33)}{(2.25,-1)}{}{minimum size=0.4cm, draw=rwth-orange}
\rotKnoten{(s34)}{(1.5,1)}{}{minimum size=0.4cm, draw=rwth-orange}
\rotKnoten{(s35)}{(3,0)}{\textcolor{rwth-orange}{$\sinkSPGraph$}}{minimum size=0.4cm, draw=rwth-orange}

\blauKnoten{(s36)}{(4.5,1)}{}{minimum size=0.4cm, draw=rwth-green}
\blauKnoten{(s37)}{(6,0)}{\textcolor{rwth-green}{$\sinkSPGraph$}}{minimum size=0.4cm, draw=rwth-green}
\end{scope}

\Pfeil{(3,1.25)}{0.9}{}
\node[right] at (3,-2.1){\small Series Composition} ; 
\end{scope}
\draw[geKante, rwth-orange] (s31) to (s32);
\draw[geKante, rwth-orange] (s31) to (s34);
\draw[geKante, rwth-orange] (s32) to (s33);
\draw[geKante, rwth-orange] (s33) to (s35);
\draw[geKante, rwth-orange] (s34) to (s35);
\draw[geKante, rwth-orange] (s31) to (s35);

\draw[geKante, rwth-green] (s35) to (s36);
\draw[geKante, rwth-green] (s36) to (s37);
\draw[geKante, rwth-green] (s35) to (s37);
\end{tikzpicture}}
		\subfloat[]{\label{fig:Spgraphs2}	\begin{tikzpicture}[xscale=1, yscale=0.75, bmarkKnoten/.style={knoten, draw=rwth-blue}, smarkKnoten/.style={knoten, draw=black},markKnoten/.style={knoten, draw=rwth-red},knoten/.style={shape=circle, thick, fill = white,minimum size=1, draw},geKante/.style={-latex, thick} ]

\begin{scope}[shift={(15,-7.5)}]
% SP tree
%\node[normal] at (0,0) (k) {\tiny s};
\fill[lightgray!50] (0.5,0) ellipse (1cm and 1.2cm);

\schwarzKnoten{(k)}{(0,0)}{\footnotesize $\sourceSPGraph_1$}{minimum size=0.4cm}
\schwarzKnoten{(l)}{(1,0)}{\footnotesize $\sinkSPGraph_1$}{minimum size=0.4cm}
\node at(0.25,0.75){\small $G_1$};

\fill[lightgray!50] (3,0) ellipse (1cm and 1.2cm);

\schwarzKnoten{(m)}{(2.5,0)}{\footnotesize $\sourceSPGraph_2$}{minimum size=0.4cm}
\schwarzKnoten{(n)}{(3.5,0)}{\footnotesize $\sinkSPGraph_2$}{minimum size=0.4cm}
\node at(3.25,0.75){\small $G_2$};

\fill[lightgray!50] (1.8,3) ellipse (1.5cm and 1.2cm);

\schwarzKnoten{(k1)}{(0.8,3)}{\footnotesize $\sourceSPGraph_1$}{minimum size=0.4cm}
\schwarzKnoten{(l1)}{(1.8,3)}{}{minimum size=0.4cm}
\schwarzKnoten{(n1)}{(2.8,3)}{\footnotesize $\sinkSPGraph_2$}{minimum size=0.4cm}

\fill[lightgray!50] (4.5,3) ellipse (1cm and 1.2cm);

\schwarzKnoten{(o)}{(4,3)}{\footnotesize $\sourceSPGraph_3$}{minimum size=0.4cm}
\schwarzKnoten{(p)}{(5,3)}{\footnotesize $\sinkSPGraph_3$}{minimum size=0.4cm}
\node at(4.75,3.75){\small $G_4$};

\fill[lightgray!50] (2.9,6) ellipse (1.5cm and 1.2cm);

\schwarzKnoten{(k2)}{(1.8,6)}{$\sourceSPGraph$}{minimum size=0.4cm}
\schwarzKnoten{(n2)}{(4,6)}{$\sinkSPGraph$}{minimum size=0.4cm}
\schwarzKnoten{(l2)}{(2.9,6.3)}{}{minimum size=0.4cm}
\node at(1.25,3.75){\small $G_3$};

\draw[geKante,lightgray!50] (k) -- (l) node[below, pos=.5, yshift=-6,black] {$L$};
\draw[geKante, rwth-orange, line width=1pt] (k) -- (l);
\draw[geKante,lightgray!50] (m) -- (n) node[below, pos=.5, yshift=-6,black] {$L$};
\draw[geKante,rwth-green, line width=1pt] (m) -- (n) ;

\draw[geKante,lightgray!50] (k1) -- (l1) node[below, yshift=-6,black] {$S$};
\draw[geKante,rwth-orange, line width=1pt] (k1) -- (l1);

\draw[geKante, rwth-green, line width=1pt] (l1) -- (n1);

\draw[geKante,lightgray!50] (o) -- (p) node[below, pos=.5, yshift=-6,black] {$L$};
\draw[geKante, Brombeer, line width=1pt] (o) -- (p);
%\draw[geKante] (k2) to [bend left=20] (l2);
%\draw[geKante] (l2) to [bend left = 15] (n2) ;
\draw[geKante, rwth-orange] (k2) to (l2);
\draw[geKante, rwth-green] (l2) to (n2);
\draw[geKante,lightgray!50] (k2) to [bend right=30] node[below, pos=.5, yshift=-1,black] {$P$} (n2);
\draw[geKante, Brombeer, line width=1pt] (k2) to [bend right=30] (n2);

% tree arcs
\draw[unKante,ultra thick, shorten <= 0.2cm,shorten >= 0.2cm] (0.5,0.5) -- (1.1,2.5) ;
\draw[unKante,ultra thick,shorten <= 0.2cm,shorten >= 0.2cm] (3,0.5) -- (2.4,2.5);
%\path (1.1,2.5) to node {S} (1.9,2.5);

\draw[unKante,ultra thick,shorten <= 0.2cm,shorten >= 0.2cm] (1.5,3.5) -- (2.1,5.5);
\draw[unKante,ultra thick,shorten <= 0.2cm,shorten >= 0.2cm] (4.2,3.5) -- (3.6,5.5);
%\path (2.1,5.5) to node {P} (2.9,5.5);

%\draw[thick,dashed] (2.5, 6.5) -- (2.8,7.5);

\node at (4.5,7.25){$T$}; 
\node at (3.45,6.75){$G$}; 
\end{scope}

\end{tikzpicture}}
	\end{adjustbox}
\caption{(a) Example of an SP digraph defined by a series or parallel composition (b) Example of the representation of an SP digraph $G$ by its SP tree $T$}
\end{figure*}
An SP digraph can be represented in the form of a rooted binary decomposition tree, a so-called \textit{SP tree}, as shown in Figure~\ref{fig:Spgraphs2}. 
The SP tree indicates the composition of the SP digraph by three different vertices, namely $L$-vertices, $S$-vertices, and $P$-vertices.
Each arc of the SP digraph is represented by an individual leaf of the SP tree, an $L$-vertex. 
The $S$- and $P$-vertices are the SP tree's inner vertices whose associated subgraphs are obtained by a series or parallel composition, respectively, of the subgraphs associated with their two child vertices. 
The SP tree can be constructed in polynomial time~\citep{valdes1982recognition}.

%%%%%%%%%%%%%%%%----------------------------------------------------------------------------------
\subsection{Multiple sources \& multiple sinks networks}
\label{Subsec:SP-MultipleSS}
Before discussing the general case of multiple sources and multiple sinks, we analyze the complexity of the \ProblemName{} problem for the special case of networks based on SP digraphs with a unique source and single sinks.
In Section~\ref{Subsec:MSInstance}, we present a specific instance of the \ProblemName{} problem. 
In Section~\ref{Subsec:UniqueSourceSingleSink}, we perform, based on this specific instance, a reduction from a weakly $\mathcal{NP}$-complete special case of the \textsc{Partition} problem~\citep{johnson1979computers}, termed as the \RFPairPartition{} problem. 
For the \RFPairPartition{} problem, positive integer pairs are to be separated by a partition of equal weight. 
The problem is formally defined as follows. 
\begin{definition}[\RFPairPartition{} problem]
	Let $s_1,\ldots,s_{2n}$ be $2n$ positive integers partitioned in sets $S_{i}=\{s_{2i-1},s_{2i}\}$ for $i\in [n]$, referred to as pairs, where the integers sum up to $2w$, i.e., $\sum^{2n}_{j=1}s_j=2w$.
	The \textsc{One-out-of-a-pair-partition} problem asks whether there exists a disjoint partition $S^1$, $S^2$ of the integers $s_1,\ldots,s_{2n}$ such that
	\begin{itemize}
		\item[(i)] the sum of all integers is equal in both subsets, i.e.,  
		\begin{align*}
			\sum_{s_i\in S^1} s_i = \sum_{s_i\in S^2} s_i = w,
		\end{align*} 
		\item[(ii)] each pair is separated, i.e., 
		\begin{align*}
			|S^1 \cap S_{i} |= |S^2 \cap S_{i} | = 1 \text{ for all }i\in [n].
		\end{align*}
	\end{itemize}
\end{definition}
We note that $\sum_{i=1}^{n}s_{2i-1}<2w$ holds true as we consider positive integers. 
The weak $\mathcal{NP}$-completeness of the \RFPairPartition{} problem is shown in Theorem~\ref{RF:NPCompletePairPartition} in Appendix~\ref{RF:AppendixA}.

\subsubsection{Maximum split instance}
In this section, we consider a specific instance of the \ProblemName{} problem which we refer to as \textit{maximum split instance}. 
First, we construct the maximum split instance. 
Second, we present properties of an optimal robust flow of this instance. 
\paragraph{Construction of the maximum split instance}
\label{Subsec:MSInstance}
For $n\in \mathbb{Z}_{}$ and $n\geq 2$, let $\mathcal{I}$ be a \RFPairPartition{} instance with positive integers $s_1,\ldots,s_{2n}$ partitioned in pairs $S_{i}=\{s_{2i-1},s_{2i}\}$, $i\in [n]$ such that $\sum_{j=1}^{2n} s_j = 2w$ holds. 
Without loss of generality, we assume the pairs to be given such that $s_{2i-1} \geq s_{2i}$ holds for $i\in [n]$. 
Considering a set of two scenarios $\Lambda=\{1,2\}$, we construct the corresponding maximum split instance $\mathcal{I}_{n}= (G_n, c, \boldsymbol{b})$ as visualized in Figure~\ref{RFfig:killerProof}. 
\begin{figure}[H]
	\centering
	\begin{adjustbox}{max width=1.5\textwidth}
		\centering
		\begin{scaletikzpicturetowidth}{1.075\textwidth}
\begin{tikzpicture}[scale=\tikzscale]
\Knoten{(s)}{(0,0)}{$s$}{thick, minimum size=0.5cm}
\Knoten{(1)}{(3.5,0)}{$v_1$}{thick, minimum size=0.5cm}
\Knoten{(2)}{(7,0)}{$v_2$}{thick, minimum size=0.5cm}
\Knoten{(3)}{(10.5,0)}{$v_3$}{thick, minimum size=0.5cm}
\Knoten{(4)}{(14,0)}{$v_4$}{thick, minimum size=0.5cm}
\Knoten{(5)}{(18,0)}{\scriptsize$v_{2n-4}$}{thick, minimum size=0.5cm}
\Knoten{(6)}{(21.5,0)}{\scriptsize$v_{2n-3}$}{thick, minimum size=0.5cm}
\Knoten{(7)}{(25,0)}{\scriptsize$v_{2n-2}$}{thick, minimum size=0.5cm}
\Knoten{(8)}{(28.5,0)}{\scriptsize$v_{2n-1}$}{thick, minimum size=0.5cm}
\Knoten{(t1)}{(32,0)}{$t_1$}{thick, minimum size=0.5cm}
\Knoten{(t2)}{(35.5,0)}{$t_2$}{thick, minimum size=0.5cm}

\draw[unKante,line width=5,  opacity=0.5,\gruen] (s) --(1);
\draw[thick,->, -latex, black] (s)  to [out=80, in=110, looseness=1.25]node[above]{\scriptsize$s_1-s_2$}(1); 
\draw[thick,->, -latex, black] (s)  --node[below]{\scriptsize$0$}(1); 
\draw[thick,->, -latex, black] (1)  to [out=80, in=110, looseness=1.25]node[above]{\scriptsize$s_2-s_3+2^{n-1}n w$}(2); 
\draw[unKante,line width=5,  opacity=0.5,\gruen] (1) --(2);
\draw[thick,->, -latex, black] (1)  --node[below]{\scriptsize$2^{n-2}n w$}(2); 
\draw[thick,->, -latex, black] (2)  to [out=80, in=110, looseness=1.25]node[above]{\scriptsize$s_3-s_4$}(3); 
\draw[unKante,line width=5,  opacity=0.5,\gruen] (2) --(3);
\draw[thick,->, -latex, black] (2)  --node[below]{\scriptsize$0$}(3); 
\draw[thick,->, -latex, black] (3)  to [out=80, in=110, looseness=1.25]node[above]{\scriptsize$s_4-s_5+2^{n-2} nw $}(4); 
\draw[unKante,line width=5,  opacity=0.5,\gruen] (3) --(4);
\draw[thick,->, -latex, black] (3)  --node[below]{\scriptsize$2^{n-3}nw$}(4); 
%\draw[thick,->, -latex, black] (4)  to [bend left=50]node[above]{$a_5-a_6$}(5); 
%\draw[unKante,line width=5,  opacity=0.5,\gruen] (4) --(5);
%\draw[thick,->, -latex, black] (4)  --node[below]{$0$}(5); 
\node(dots) at (16,0){\Large $\ldots$};

\draw[thick,->, -latex, black] (5)  to [out=80, in=110, looseness=1.25]node[above, xshift=-0]{\scriptsize$s_{ 2n-3}-s_{2n-2}$}(6); 
\draw[unKante,line width=5,  opacity=0.5,\gruen] (5) --(6);
\draw[thick,->, -latex, black] (5)  --node[below]{\scriptsize$0$}(6); 
%\draw[thick,->, -latex, black] (6)  to [out=80, in=110, looseness=1.25]node[above, align=center]{\scriptsize $a_{2n-2}-a_{2n-1}+$\\\scriptsize $+c_{n-1} B$}(7); 
\draw[thick,->, -latex, black] (6)  to [out=80, in=110, looseness=1.25]node[above, xshift=0.1cm]{\scriptsize $s_{2n-2}-s_{2n-1}$}(7); 
\draw[thick,->, -latex, black] (6)  to [out=80, in=110, looseness=1.25]node[below, align=center]{\scriptsize $+2^1 nw$}(7); 
\draw[unKante,line width=5,  opacity=0.5,\gruen] (6) --(7);
\draw[thick,->, -latex, black] (6)  --node[below]{\scriptsize$2^0 nw$}(7);

\draw[thick,->, -latex, black] (7)  to [out=80, in=110, looseness=1.25]node[above, xshift=0]{\scriptsize $s_{2n-1}-s_{2n}$}(8); 
\draw[unKante,line width=5,  opacity=0.5,\gruen] (7) --(8);
\draw[thick,->, -latex, black] (7)  --node[below]{\scriptsize$0$}(8); 
\draw[thick,->, -latex, black] (8)  to [out=80, in=110, looseness=1.25]node[above]{\scriptsize $s_{2n}+2^0 nw$}(t1);

%\draw[thick,->, -latex, black] (5)  to [bend left=50]node[above]{$a_6+c_3$}(t1); 
\draw[thick,->, -latex, black] (t1)  --node[below]{\scriptsize$2^{n+1}n^2 w$}(t2); 
\draw[thick,->, -latex, black] (s)  to [out=270, in=290]node[below]{\scriptsize$s_2+F  w$}(t2); 
\draw[thick,->, -latex, black] (1)  to [out=270, in=285]node[below]{\scriptsize$s_1+F  w$}(t2); 
\draw[thick,->, -latex, black] (2)  to [out=270, in=280]node[below]{\scriptsize$s_4+F  w$}(t2); 
\draw[thick,->, -latex, black] (3)  to [out=280, in=275]node[below]{\scriptsize$s_3+F  w$}(t2); 
\draw[thick,->, -latex, black] (4)  to [out=280, in=270]node[below]{\scriptsize$s_6+F  w$}(t2); 
\draw[thick,->, -latex, black] (5)  to [out=280, in=265]node[below]{\scriptsize$s_{2n-2}+F  w$}(t2); 
\draw[thick,->, -latex, black] (6)  to [out=280, in=265]node[below]{\scriptsize$s_{2n-3}+F  w$}(t2); 
\draw[thick,->, -latex, black] (7)  to [out=280, in=265]node[below]{\scriptsize$s_{2n}+F  w$}(t2); 
\draw[thick,->, -latex, black] (8)  to [out=280, in=265]node[below]{\scriptsize$s_{2n-1}+F  w$}(t2); 

%Balancen
\node[orange, left] at ($(s)+(-0.5,0.35)$){$n$}; 
\node[blue, left] at ($(s)-(0.5,0.35)$){$n$}; 
\node[orange] at ($(t1)+(0.3,0.95)$){$-n$}; 
\node[blue] at ($(t2)+(0.3,0.95)$){$-n$};

% Legende
\begin{scope}[shift={(25,-9.75)}]
	\begin{scope}[shift={(-8,-3.5)}]
		\Job{(0,0)}{(0.6,0.6)}{\gruen, \gruen}{}{}
		\node[right] at (0.6,0.3){$A^{\text{fix}}$}; 
	\end{scope}
	\begin{scope}[shift={(-4,-3.5)}]
		\Job{(0,0)}{(0.6,0.6)}{black, black}{}{}
		\node[right] at (0.6,0.3){$A^{\text{free}}$}; 
	\end{scope}
	\begin{scope}[shift={(0,-3.5)}]
		\Job{(0,0)}{(0.6,0.6)}{orange, orange}{}{}
		\node[right] at (0.6,0.3){$\text{Balances } b^1$}; 
	\end{scope}
	\begin{scope}[shift={(6.75,-3.5)}]
		\Job{(0,0)}{(0.6,0.6)}{blue, blue}{}{}
		\node[right] at (0.6,0.3){$\text{Balances } b^2$}; 
	\end{scope}
\end{scope}
%
%\node(dots) at (2,-10.85){$F=2^{n+1}-n-2$};
\node(F) at (12,-12.95){$F=2^{n+1}-n-2$};
\end{tikzpicture}
\end{scaletikzpicturetowidth}
	\end{adjustbox}
	\caption{Construction of the maximum split instance $\mathcal{I}_{n}$}
	\label{RFfig:killerProof}
\end{figure}
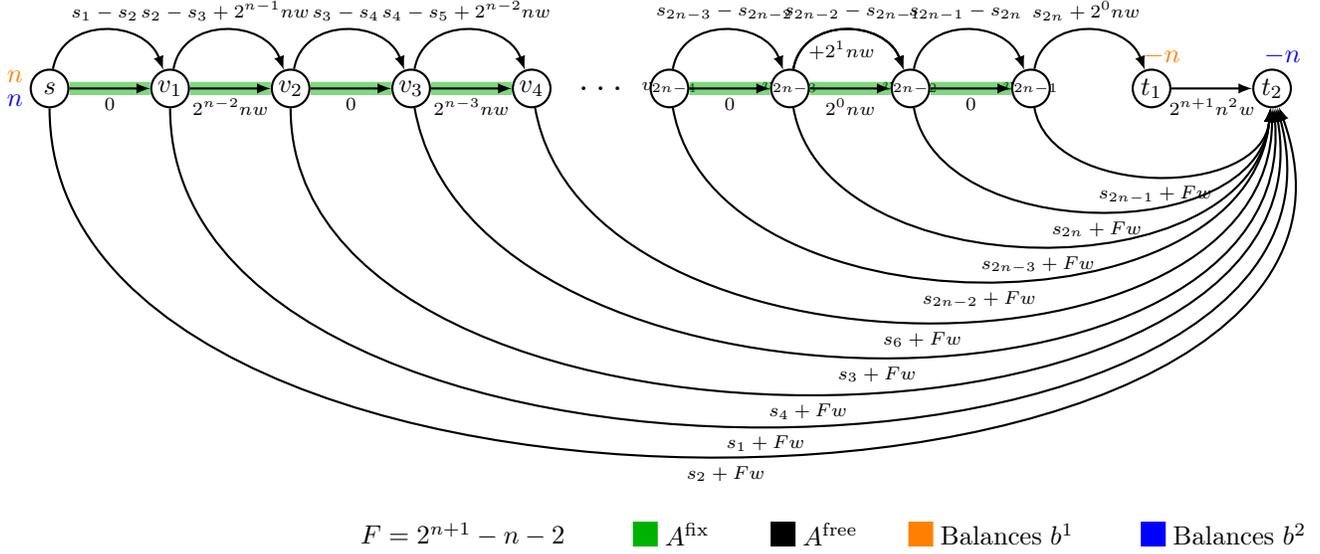
Let $G_n=(V_n,A_n)$ be an SP digraph with vertex set $V_n$ and arc set $A_n$. 
The vertex set $V_n$ contains auxiliary vertices $v_1,\ldots,v_{2n-1}$ and three specified vertices $s$,$t_1$,$t_2$. 
The arc set $A_n$ contains four types of arcs, termed as detour, cross, shortcut, and blocked arcs.  
The \textit{detour arcs} are $a^{d}_{1} = (s,v_1)$, $a^{d}_{i+1} = (v_i,v_{i+1})$ for $i\in [2n-2]$, and $a^{d}_{2n} = (v_{2n-1},t_1)$. 
They create a path from vertex $s$ via auxiliary vertices $v_1,\ldots,v_{2n-1}$ to vertex $t_1$.
The \textit{cross arcs} are $a^{c}_{1} = (s,v_1)$ and $a^{c}_{i+1} = (v_i,v_{i+1})$ for $i\in [2n-2]$. 
They create a path from vertex $s$ via auxiliary vertices $v_1,\ldots,v_{2n-2}$ to vertex $v_{2n-1}$.  
The \textit{shortcut arcs} are $a^{s}_1=(s,t_2)$ and $a^{s}_{i+1}=(v_i,t_2)$ for $i\in [2n-1]$. 
They create a direct connection between vertices $s,v_1,\ldots,v_{2n-1}$ and vertex $t_2$. 
The single \textit{blocked arc} $a^b=(t_1,t_2)$ connects vertices $t_1$ and $t_2$. 
We identify the cross arcs as fixed arcs contained in set $\fix_n$ and all other arcs as free arcs contained in set $\free_n$ such that $A_n=\fix_n \cup \free_n$ holds.
For all arcs $a\in A_n$, we define the cost $c$ as follows. 
The cost of detour arcs $a^d_{i}$ for $i\in [2n]$ is 
\begin{align*}
	c(a^d_{i}) = \begin{cases}
		s_i-s_{i+1} &\text{ if $i\in \{1,3,\ldots,2n-1\},$}\\
		s_i-s_{i+1}+2^{n-\frac{i}{2}}nw &\text{ if $i\in \{2,4,\ldots,2n-2\},$}\\
		s_{2n}+2^{0}nw &\text{ if $i=2n$.}
	\end{cases}
\end{align*}
The cost of cross arcs $a^c_{i}$ for $i\in [2n-1]$ is 
\begin{align*}
	c(a^c_{i}) = 
	\begin{cases}
		0 &\text{ if $i\in \{1,3,\ldots,2n-1\},$}\\
		2^{n-\frac{i}{2}-1}nw &\text{ if $i\in \{2,4,\ldots,2n-2\}.$}
	\end{cases}
\end{align*}
The cost of shortcut arcs $a^s_i$ for $i\in [2n]$ is 
\begin{align*}
	c(a^s_i)= 
	\begin{cases}
		s_{i+1}+Fw &\text{ if $i\in \{1,3,\ldots,2n-1\}$,}\\
		s_{i-1}+Fw &\text{ if $i\in \{2,4,\ldots,2n\}$,}
	\end{cases}
	\ \text{ with $F=2^{n+1}-n-2$.}
\end{align*} 
The cost of the blocked arc is $c(a^b)= 2^{n+1}n^2w$.
We note that the cost is non-negative for all arcs as $s_{2i-1}\geq s_{2i}$ holds for $i\in [n]$ and $s_{2i+1}\leq 2w \leq nw \leq 2^{n-i}nw$ holds for $i\in [n-1]$.
For all vertices $v\in V_n$, we define the balances $\boldsymbol{b}=(b^1,b^2)$ by
\begin{align*}
	b^1(v)
	= \left\{ \begin{array}{ll}
		n 			& \mbox{if $v=s$}, \\ 
		-n 			& \mbox{if $v=t_1$}, \\ 
		0 			& \mbox{otherwise},  \end{array} \right.
	& \hspace*{0.5cm} 
	b^2(v)= \left\{ \begin{array}{ll}
		n 			& \mbox{if $v=s$}, \\ 
		-n 			& \mbox{if $v=t_2$}, \\
		0 			& \mbox{otherwise}.  \end{array} \right.
\end{align*}
Vertex $s$ specifies the unique source and vertices $t_1$ and $t_2$ specify the sinks of the first and second scenario, respectively. 
Overall, we obtain a feasible \ProblemName{} instance $\mathcal{I}_n= (G_n, c, \boldsymbol{b})$. 

\paragraph{Properties of an optimal robust flow of the maximum split instance}
In this paragraph, we present some properties of an optimal robust $\boldsymbol{b}$-flow $\boldsymbol{f}=(f^1,f^2)$ of the maximum split instance $\mathcal{I}_n=(G_n,c,\fett{b})$.
As a result, we establish a unique structure of an optimal robust flow which is independent of the \RFPairPartition{} instance on which the maximum split instance is based. 
The first property is that the flow on the fixed arcs decreases the shorter the distance to vertex $v_{2n-1}$ is, as shown in the following Lemma. 
A proof is found in Appendix~\ref{RF:AppendixC}.
\begin{restatable}[]{lemma}{auxiliarylemmadecreasingflowonfixedarcs}\label{RFlem:BalanceConstraintsMaxSplitInstance}
	For an optimal robust $\fett{b}$-flow $\fett{f}=(f^1,f^2)$ of the maximum split instance $\mathcal{I}_n=(G_n,c,\fett{b})$, it holds  
	\begin{align*}
		f^1(a^c_i) = f^2(a^c_i)\geq f^1(a^c_{i+1})= f^2(a^c_{i+1}) \ \text{ for all }i\in [2n-2].
	\end{align*}
\end{restatable}
The next property is along which paths the scenario flows $f^1$, $f^2$ of an optimal robust $\fett{b}$-flow $\fett{f}$ might send units and what cost is incurred. 
In the first scenario, flow $f^1$ may only use cross and detour arcs to reach the sink. 
Once flow $f^1$ sends a unit along a detour arc $a^d_i$ for $i\in [2n-1]$, the unit is forwarded along the successive detour arcs $a^d_{i+1},\ldots, a^d_{2n}$.
If the unit was forwarded along one of the successive cross arcs $a^c_{i+1}, \ldots, a^c_{2n-1}$, Lemma~\ref{RFlem:BalanceConstraintsMaxSplitInstance} would not be satisfied. 
Consequently, in the first scenario units are sent along paths of the form $p^1_i=sa^c_1v_1a^c_2\ldots a^c_{i-1}v_{i-1}a^d_{i}v_ia^d_{i+1}\ldots a^d_{2n-1}v_{2n-1}a^d_{2n}t_1$ for $i\in [2n]$ as visualized in Figure~\ref{RFfig:ScenarioPaths}.
The costs of paths $p^1_{2i-1}$ and $p^1_{2i}$ for $i\in [n]$ are  
\begin{align*}
	c(p^1_{2i-1})
	&=\sum_{j=1}^{2i-2} c(a^c_{j})  
	+ \sum_{j=2i-1}^{2n} c(a^d_j)\\
	&= \sum_{j=1}^{i-1} 2^{n-j-1}nw 
	+ 
	\sum_{j= i}^{n} (s_{2j-1} - s_{2j} )
	+
	\sum_{j=i}^{n-1} (s_{2j} - s_{2j+1} + 2^{n-j}nw)
	+ s_{2n} + 2^0nw \\
	&= nw (2^{i}-2)2^{n-i-1} + s_{2i-1} +nw(2^{n-i+1}-1)\\
	&= nw \left(2^{n-i} + 2^{n-1}-1 \right) + s_{2i-1}
\end{align*} 
and
\begin{align*}
	c(p^1_{2i})
	&=\sum_{j=1}^{2i-1} c(a^c_{j})  
	+ \sum_{j=2i}^{2n} c(a^d_j)\\
	&= \sum_{j=1}^{i-1} 2^{n-j-1}nw 
	+ 
	\sum_{j= i+1}^{n} (s_{2j-1} - s_{2j} )
	+
	\sum_{j=i}^{n-1} (s_{2j} - s_{2j+1} + 2^{n-j}nw)
	+ s_{2n} + 2^0nw \\
	&= nw (2^{i}-2)2^{n-i-1} + s_{2i} + nw(2^{n-i+1}-1)\\
	&= nw \left(2^{n-i} + 2^{n-1}-1 \right) + s_{2i}, 
\end{align*} 
respectively. Overall, the cost of the first scenario's paths $p^1_i$, $i\in [2n]$ is $c(p^1_i)= nw \left(2^{n-\lceil \frac{i}{2}\rceil} + 2^{n-1}-1 \right) + s_{i}$. 

In the second scenario, flow $f^2$ only uses cross and shortcut arcs to reach the sink. 
The blocked arc $a^b$ is not used due to its high cost compared to the (sum of) other arcs.  
The detour arcs $a^d_i$, $i\in [2n-1]$ are not used, because they cost about twice as much as the cross arcs (plus $s_i$ minus $s_{i+1}$). 
Moreover, the use of cross instead of detour arcs reduces not only the cost in the second but also in the first scenario.  
The detour arc $a^d_{2n}$ is not used, otherwise the blocked arc $a^b$ would subsequently have to be used  to reach the sink. 
Consequently, in the second scenario units are sent along paths of the form $p^2_i=sa^c_1v_1a^c_2\ldots a^c_{i-1}v_{i-1}a^s_{i}t_2$ for $i\in [2n]$ as visualized in Figure~\ref{RFfig:ScenarioPaths}.
The costs of paths $p^2_{2i-1}$ and $p^2_{2i}$ for $i\in [n]$ are  
\begin{align*}
	c(p^2_{2i-1})
	=\sum_{j=1}^{2i-2} c(a^c_{j})  
	+ c(a^s_{2i-1})
	= \sum_{j=1}^{i-1} 2^{n-j-1}nw 
	+ s_{2i} + Fw 
	= nw (2^{i}-2)2^{n-i-1}  + s_{2i} + Fw
\end{align*} 
and 
\begin{align*}
	c(p^2_{2i})
	=\sum_{j=1}^{2i-1} c(a^c_{j})  
	+ c(a^s_{2i})
	= \sum_{j=1}^{i-1} 2^{n-j-1}nw 
	+ s_{2i-1} + Fw 
	= nw (2^{i}-2)2^{n-i-1}  + s_{2i-1} + Fw,
\end{align*} 
respectively. Overall, the cost of the second scenario's paths $p^2_i$, $i\in [2n]$ is 
\begin{align*}
	c(p^2_i)=
	\begin{cases}
		nw (2^{i}-2)2^{n-i-1} + s_{i+1} + Fw & \text{with $i\in \{1,3,\ldots,2n-1\},$}\\
		nw (2^{i}-2)2^{n-i-1}  + s_{i-1} + Fw & \text{with $i\in \{2,4,\ldots,2n\}$}.
	\end{cases}
\end{align*} 
We note that paths $p^1_i$ and $p^2_i$ include the same fixed arcs for all $i\in [2n]$. 
Thus, the choice of the paths in the first and second scenario depends on each other due to the consistent flow constraints. 
More precisely, if a unit is sent along path $p^1_i$, $i\in [2n]$ in the first scenario, a unit is also sent along path $p^2_i$ in the second scenario (and vice versa), as visualized in Figure~\ref{RFfig:ScenarioPaths}. 
If the unit was sent along a different path $p\neq p^2_i$ in the second scenario, the flow would be either infeasible as the consistent flow constraints are not satisfied (for $p$ with $\{a^c_1,\ldots,a^c_{i-1}\}\not \subseteq A(p)$ or $a^c_j \in A(p)$ for at least one $j\in \{i,\ldots,2n-1\}$) and/or the unit could be sent cheaper (for $p$ with $a^d_j\in A(p)$ for at least one $j\in [2n-1]$). 
We refer to this property as \textit{equal paths property}. 
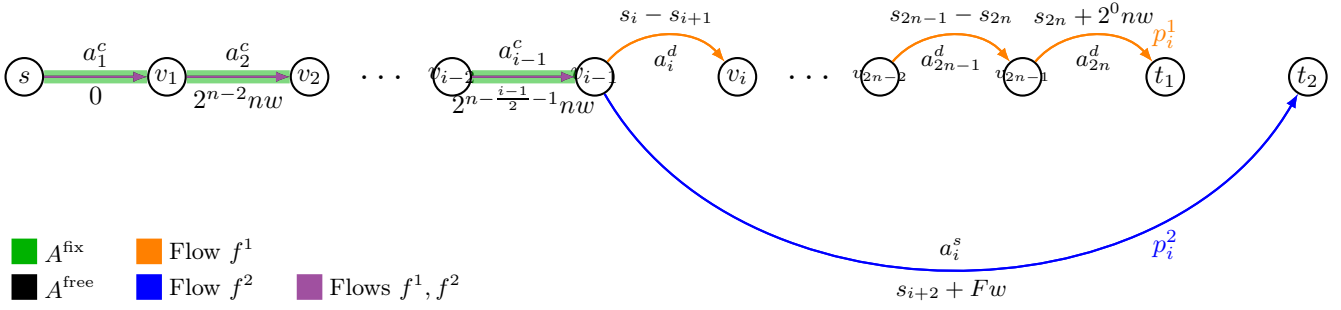
\begin{figure}[H]
	\begin{adjustbox}{max width=1\textwidth}
		\begin{scaletikzpicturetowidth}{\textwidth}
\begin{tikzpicture}[scale=\tikzscale]
\Knoten{(s)}{(0,0)}{$s$}{thick, minimum size=0.5cm}
\Knoten{(1)}{(2,0)}{$v_1$}{thick, minimum size=0.5cm}
\Knoten{(2)}{(4,0)}{$v_2$}{thick, minimum size=0.5cm}
\node(dots) at (5,0){\Large $\ldots$};
\Knoten{(3)}{(6,0)}{$v_{i-2}$}{thick, minimum size=0.5cm}
\Knoten{(4)}{(8,0)}{$v_{i-1}$}{thick, minimum size=0.5cm}
\Knoten{(5)}{(10,0)}{$v_{i}$}{thick, minimum size=0.5cm}
\node(dots) at (11,0){\Large $\ldots$};
\Knoten{(6)}{(12,0)}{\scriptsize $v_{2n-2}$}{thick, minimum size=0.5cm}
\Knoten{(7)}{(14,0)}{\scriptsize$v_{2n-1}$}{thick, minimum size=0.5cm}
\Knoten{(t1)}{(16,0)}{$t_1$}{thick, minimum size=0.5cm}
\Knoten{(t2)}{(18,0)}{$t_2$}{thick, minimum size=0.5cm}

\draw[unKante,line width=5,  opacity=0.5,\gruen] (s) --(1);
\draw[thick,->, -latex] (s)  --node[above, black]{$a^c_1$}(1); 
\draw[thick,->, -latex, Brombeer] (s)  --node[below, black]{$0$}(1); 
%\draw[thick,->, -latex, orange, dash pattern=on 8pt off 8pt] (s)  --(1); 
\draw[unKante,line width=5,  opacity=0.5,\gruen] (1) --(2);
\draw[thick,->, -latex] (1)  --node[above, black]{$a^c_{2}$}(2); 
\draw[thick,->, -latex, Brombeer] (1)  --node[below, black]{$2^{n-2}nw$}(2); 
%\draw[thick,->, -latex, orange, dash pattern=on 8pt off 8pt] (1)  --(2); 
\draw[unKante,line width=5,  opacity=0.5,\gruen] (3) --(4);
\draw[thick,->, -latex, ] (3)  --node[above, black]{$a^c_{i-1}$}(4); 
\draw[thick,->, -latex, Brombeer] (3)  --node[below, black]{$2^{n-\frac{i-1}{2}-1}nw$}(4); 
%\draw[thick,->, -latex, orange, dash pattern=on 8pt off 8pt] (3)  --(4); 

\draw[thick,->, -latex, orange] (4)  to [bend left=50]node[below, black]{\small $a^d_i$}(5); 
\draw[thick,->, -latex, orange] (4)  to [bend left=50]node[above, black]{\small $s_{i}-s_{i+1}$}(5); 
\draw[thick,->, -latex, orange] (6)  to [bend left=50]node[below, black]{\small $a^d_{2n-1}$}(7); 
\draw[thick,->, -latex, orange] (6)  to [bend left=50]node[above, black]{\small $s_{2n-1}-s_{2n}$}(7); 
\draw[thick,->, -latex, orange] (7)  to [bend left=50]node[below, black]{\small $a^d_{2n}$}(t1); 
\draw[thick,->, -latex, orange] (7)  to [bend left=50]node[above, black]{\small $s_{2n}+2^0 nw$}(t1); 

\draw[thick,->, -latex, blue] (4)  to [out=300, in=240]node[above, black]{\small $a^s_i$}(t2); 
\draw[thick,->, -latex, blue] (4)  to [out=300, in=240]node[below, black]{\small $s_{i+2} + F w $}(t2); 

%Beschriftung
\node[orange] at (16,0.6){$p^1_i$};
\node[blue] at (16,-2.35){$p^2_i$};

% Legende
\begin{scope}[shift={(2.35,0.4)}]
\begin{scope}[shift={(-2.5,-3)}]
\Job{(0,0)}{(0.3,0.3)}{\gruen, \gruen}{}{}
\node[right] at (0.3,0.15){\small$A^{\text{fix}}$}; 
\end{scope}
\begin{scope}[shift={(-2.5,-3.5)}]
\Job{(0,0)}{(0.3,0.3)}{black, black}{}{}
\node[right] at (0.3,0.15){\small$A^{\text{free}}$}; 
\end{scope}
\begin{scope}[shift={(-0.75,-3)}]
\Job{(0,0)}{(0.3,0.3)}{orange, orange}{}{}
\node[right] at (0.3,0.15){\small$\text{Flow } f^1$}; 
\end{scope}
\begin{scope}[shift={(-0.75,-3.5)}]
\Job{(0,0)}{(0.3,0.3)}{blue, blue}{}{}
\node[right] at (0.3,0.15){\small$\text{Flow } f^2$}; 
\end{scope}
\begin{scope}[shift={(1.5,-3.5)}]
\Job{(0,0)}{(0.3,0.3)}{Brombeer, Brombeer}{}{}
\node[right] at (0.3,0.15){\small $\text{Flows } f^1, f^2$}; 
\end{scope}
\end{scope}
\end{tikzpicture}
\end{scaletikzpicturetowidth}
	\end{adjustbox}
	\caption{Example of paths $p^1_i$ and $p^2_i$ for $i\in \{1,3,\ldots,2n-1\}$ and their dependence on each other}
	\label{RFfig:ScenarioPaths}
\end{figure}
In the next step, we note the following about the cost of the scenario flows. 
In terms of minimizing the cost in the first scenario, the aim is a flow which uses as few detour arcs as possible or, in other words, as many cross arcs as possible. 
In terms of minimizing the cost in the second scenario, the aim is a flow which uses as few cross arcs as possible. % (and no detour arcs at all). 
The best case for the first scenario's cost and the worst case for the second scenario's cost occur if a robust flow sends all units along path $p^1_{2n}$ in the first and along path $p^2_{2n}$ in the second scenario. 
Conversely, the worst case for the first scenario's cost and the best case for the second scenario's cost occur if a robust flow sends all units along path $p^1_1$ in the first and along path $p^2_1$ in the second scenario. 
As the cost of a robust flow is the maximum of its scenario flows' cost, we aim at a trade-off between the best and worst case of the first and second scenario's cost. 
The fixed (cross) arcs are used in both scenarios and contributes significantly to the cost incurred. 
Therefore, in the following lemma, we consider the cost of an optimal robust flow incurred by the fixed arcs. 
A proof is found in Appendix~\ref{RF:AppendixC}. 
\begin{restatable}[]{lemma}{AuxiliaryLemmaCostOnFixedArcs}
	\label{RFlem:HilfslemmaCostOfFixedArcsForOptimalSol}
	For an optimal robust $\fett{b}$-flow $\fett{f}=(f^1,f^2)$ of the maximum split instance $\mathcal{I}_{n}= (G_n=(V_n,A_n), c, \boldsymbol{b})$ with $A_n=A^{\freehelp}_{n}\cup A^{\fixhelp}_{n}$, the total cost incurred on the fixed arcs is $c^{\fixhelp}:=nw(2^{n-1}n-2^n+1)$, i.e., 
	\begin{align*}
		\sum_{a\in A^{\fixhelp}_{n}} c(a) \cdot f^1(a)  = \sum_{a\in A^{\fixhelp}_{n}} c(a) \cdot f^2(a)  = c^{\fixhelp}. 
	\end{align*}
\end{restatable}
As a result, the flow on the fixed arcs is determined for an optimal robust flow as shown in the following lemma. 
A proof is found in Appendix~\ref{RF:AppendixC}.
\begin{restatable}[]{lemma}{AuxiliaryLemmaFlowValueOnFixedArcs}\label{RFlem:HilfslemmaFlowValueOnFixedArcsOfOptimalSol}
	For an optimal robust $\fett{b}$-flow  $\fett{f}=(f^1,f^2)$ of the maximum split instance $\mathcal{I}_{n}= (G_n=(V_n,A_n), c, \boldsymbol{b})$ with $A_n=A^{\freehelp}_{n}\cup A^{\fixhelp}_{n}$, the flow on the fixed arcs is given by 
	$f^1(a^c_{2i-1})  = f^2(a^c_{2i-1})\in  \{n-i, n-i+1\}$ for $a^c_{2i-1}\in \fix_n$ and 
	$f^1(a^c_{2i})= f^2(a^c_{2i}) = n-i$ for $a^c_{2i}\in \fix_n.$
\end{restatable}

Using the properties about the paths and their cost as well as the insights of Lemmas~\ref{RFlem:BalanceConstraintsMaxSplitInstance}--\ref{RFlem:HilfslemmaFlowValueOnFixedArcsOfOptimalSol}, we present the structure of an optimal robust $\fett{b}$-flow $\fett{f}=(f^1,f^2)$ of the maximum split instance $\mathcal{I}_n$, visualized in Figure~\ref{RFfig:FlowBalanceConstraintsOfNotSetValues} (where the shortcut arcs are only hinted). 
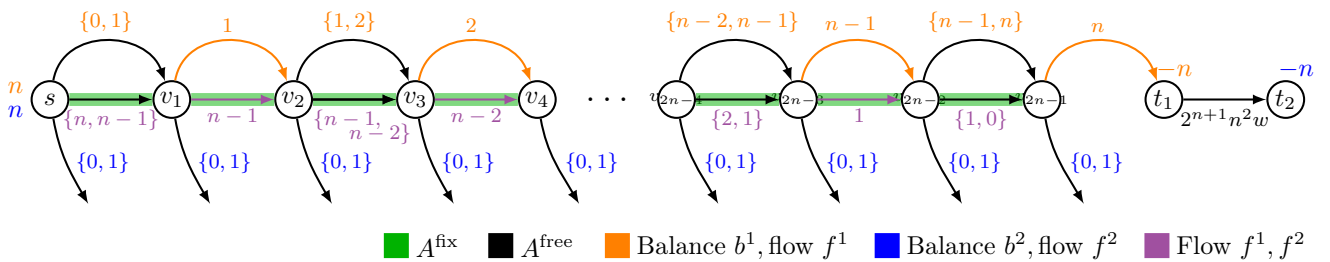
\begin{figure}[H]
	\begin{adjustbox}{max width=1\textwidth}
		\begin{scaletikzpicturetowidth}{\textwidth}
\begin{tikzpicture}[scale=\tikzscale]
\Knoten{(s)}{(0,0)}{$s$}{thick, minimum size=0.5cm}
\Knoten{(1)}{(3.5,0)}{$v_1$}{thick, minimum size=0.5cm}
\Knoten{(2)}{(7,0)}{$v_2$}{thick, minimum size=0.5cm}
\Knoten{(3)}{(10.5,0)}{$v_3$}{thick, minimum size=0.5cm}
\Knoten{(4)}{(14,0)}{$v_4$}{thick, minimum size=0.5cm}
\Knoten{(5)}{(18,0)}{\scriptsize$v_{2n-4}$}{thick, minimum size=0.5cm}
\Knoten{(6)}{(21.5,0)}{\scriptsize$v_{2n-3}$}{thick, minimum size=0.5cm}
\Knoten{(7)}{(25,0)}{\scriptsize$v_{2n-2}$}{thick, minimum size=0.5cm}
\Knoten{(8)}{(28.5,0)}{\scriptsize$v_{2n-1}$}{thick, minimum size=0.5cm}
\Knoten{(t1)}{(32,0)}{$t_1$}{thick, minimum size=0.5cm}
\Knoten{(t2)}{(35.5,0)}{$t_2$}{thick, minimum size=0.5cm}

\draw[unKante,line width=5,  opacity=0.5,\gruen] (s) --(1);
\draw[thick,->, -latex, black] (s)  to [out=80, in=110, looseness=1.25]node[above, orange]{\footnotesize$\{0,1\}$}(1); 
\draw[thick,->, -latex, black] (s)  --node[below, Brombeer]{\footnotesize$\{n,n-1\}$}(1); 
%\draw[thick,->, -latex, orange] (1)  to [out=80, in=110, looseness=1.25]node[above, black]{\scriptsize$2^{n-1}n$}(2); 
\draw[thick,->, -latex, orange] (1)  to [out=80, in=110, looseness=1.25]node[above, orange]{\footnotesize$1$}(2); 
\draw[unKante,line width=5,  opacity=0.5,\gruen] (1) --(2);
\draw[thick,->, -latex, Brombeer] (1)  --node[below, Brombeer]{\footnotesize$n-1$}(2); 
\draw[thick,->, -latex, black] (2)  to [out=80, in=110, looseness=1.25]node[above, orange]{\footnotesize$\{1,2\}$}(3); 
\draw[unKante,line width=5,  opacity=0.5,\gruen] (2) --(3);
\draw[thick,->, -latex, ] (2)  --node[below, xshift=-0.1cm, Brombeer]{\footnotesize $\{n-1,$}(3); 
\draw[thick,->, -latex, ] (2)  --node[below=0.2cm, xshift=0.4cm, Brombeer, align=center]{\footnotesize $n-2\}$}(3); 
%\draw[thick,->, -latex, orange] (3)  to [out=80, in=110, looseness=1.25]node[above]{\scriptsize$2^{n-2}n $}(4); 
\draw[thick,->, -latex, orange] (3)  to [out=80, in=110, looseness=1.25]node[above]{\footnotesize$2$}(4); 
\draw[unKante,line width=5,  opacity=0.5,\gruen] (3) --(4);
\draw[thick,->, -latex, Brombeer] (3)  --node[below, Brombeer]{\footnotesize$n-2$}(4);

\node(dots) at (16,0){\Large $\ldots$};

\draw[thick,->, -latex, black] (5)  to [out=80, in=110, looseness=1.25]node[above, xshift=-0, orange]{\footnotesize$\{n-2,n-1\}$}(6); 
\draw[unKante,line width=5,  opacity=0.5,\gruen] (5) --(6);
\draw[thick,->, -latex, black] (5)  --node[below, Brombeer]{\footnotesize$\{2,1\}$}(6); 
%\draw[thick,->, -latex, black] (6)  to [out=80, in=110, looseness=1.25]node[above, align=center]{\scriptsize $a_{2n-2}-a_{2n-1}+$\\\scriptsize $+c_{n-1} B$}(7); 
%\draw[thick,->, -latex, orange] (6)  to [out=80, in=110, looseness=1.25]node[above, align=center, black]{\scriptsize $2^1 n $}(7); 
\draw[thick,->, -latex, orange] (6)  to [out=80, in=110, looseness=1.25]node[above, align=center, orange]{\footnotesize$n-1$}(7); 
\draw[unKante,line width=5,  opacity=0.5,\gruen] (6) --(7);
\draw[thick,->, -latex, Brombeer] (6)  --node[below, Brombeer]{\footnotesize$1$}(7);

\draw[thick,->, -latex, black] (7)  to [out=80, in=110, looseness=1.25]node[above, xshift=0, orange]{\footnotesize$\{n-1,n\}$}(8); 
\draw[unKante,line width=5,  opacity=0.5,\gruen] (7) --(8);
\draw[thick,->, -latex, black] (7)  --node[below, Brombeer]{\footnotesize $\{1,0\}$}(8); 
%\draw[thick,->, -latex, orange] (8)  to [out=80, in=110, looseness=1.25]node[above, black]{\scriptsize $2^0n $}(t1); 
\draw[thick,->, -latex, orange] (8)  to [out=80, in=110, looseness=1.25]node[above, orange]{\footnotesize$n$}(t1);

%\draw[thick,->, -latex, black] (5)  to [bend left=50]node[above]{$a_6+c_3$}(t1); 
\draw[thick,->, -latex, black] (t1)  --node[below]{\footnotesize $2^{n+1}n^2w$}(t2); 
%\draw[thick,->, -latex, black] (s)  to [out=270, in=290]node[below]{\scriptsize}(t2); 
\draw[thick,->, -latex, ] (s)  to [out=280, in=120]node[right, blue]{\footnotesize$\{0,1\}$}($(s)+(1.1,-3)$); 
%\draw[thick,->, -latex, black] (1)  to [out=270, in=285]node[below]{\scriptsize$  $}(t2); 
\draw[thick,->, -latex, ] (1)  to [out=280, in=120]node[right, blue]{\footnotesize$\{0,1\}$}($(1)+(1.1,-3)$); 
%\draw[thick,->, -latex, black] (2)  to [out=270, in=280]node[below]{\scriptsize$  $}(t2); 
\draw[thick,->, -latex, ] (2)  to [out=280, in=120]node[right, blue]{\footnotesize$\{0,1\}$}($(2)+(1.1,-3)$); 
%\draw[thick,->, -latex, black] (3)  to [out=280, in=275]node[below]{\scriptsize$  $}(t2); 
\draw[thick,->, -latex, ] (3)  to [out=280, in=120]node[right, blue]{\footnotesize$\{0,1\}$}($(3)+(1.1,-3)$); 
%\draw[thick,->, -latex, black] (4)  to [out=280, in=270]node[below]{\scriptsize$  $}(t2); 
\draw[thick,->, -latex, ] (4)  to [out=280, in=120]node[right, blue]{\footnotesize$\{0,1\}$}($(4)+(1.2,-3)$); 
%\draw[thick,->, -latex, black] (5)  to [out=280, in=265]node[below]{\scriptsize$  $}(t2);
\draw[thick,->, -latex, ] (5)  to [out=280, in=120]node[right, blue]{\footnotesize$\{0,1\}$}($(5)+(1.2,-3)$);  
%\draw[thick,->, -latex, black] (6)  to [out=280, in=265]node[below]{\scriptsize$  $}(t2);
\draw[thick,->, -latex, ] (6)  to [out=280, in=120]node[right, blue]{\footnotesize$\{0,1\}$}($(6)+(1.2,-3)$);  
%\draw[thick,->, -latex, black] (7)  to [out=280, in=265]node[below]{\scriptsize$  $}(t2); 
\draw[thick,->, -latex, ] (7)  to [out=280, in=120]node[right, blue]{\footnotesize$\{0,1\}$}($(7)+(1.2,-3)$); 
%\draw[thick,->, -latex, black] (8)  to [out=280, in=265]node[below]{\scriptsize$  $}(t2); 
\draw[thick,->, -latex, ] (8)  to [out=280, in=120]node[right, blue]{\footnotesize$\{0,1\}$}($(8)+(1.2,-3)$); 

%Balancen
\node[orange, left] at ($(s)+(-0.5,0.35)$){$n$}; 
\node[blue, left] at ($(s)-(0.5,0.35)$){$n$}; 
\node[orange] at ($(t1)+(0.3,0.85)$){$-n$}; 
\node[blue] at ($(t2)+(0.3,0.85)$){$-n$};

% Legende
\begin{scope}[shift={(22,-1)}]
\begin{scope}[shift={(-12.35,-3.5)}]
\Job{(0,0)}{(0.6,0.6)}{\gruen, \gruen}{}{}
\node[right] at (0.6,0.3){$A^{\text{fix}}$}; 
\end{scope}
\begin{scope}[shift={(-9.35,-3.5)}]
\Job{(0,0)}{(0.6,0.6)}{black, black}{}{}
\node[right] at (0.6,0.3){$A^{\text{free}}$}; 
\end{scope}
\begin{scope}[shift={(-6,-3.5)}]
\Job{(0,0)}{(0.6,0.6)}{orange, orange}{}{}
\node[right] at (0.6,0.3){$\text{Balance } b^1, \text{flow } f^1$}; 
\end{scope}
\begin{scope}[shift={(1.75,-3.5)}]
\Job{(0,0)}{(0.6,0.6)}{blue, blue}{}{}
\node[right] at (0.6,0.3){$\text{Balance } b^2,\text{flow } f^2$}; 
\end{scope}
\begin{scope}[shift={(9.5,-3.5)}]
\Job{(0,0)}{(0.6,0.6)}{Brombeer, Brombeer}{}{}
\node[right] at (0.6,0.3){$\text{Flow } f^1, f^2$}; 
\end{scope}
\end{scope}
\end{tikzpicture}
\end{scaletikzpicturetowidth}
	\end{adjustbox}
	\caption{The structure of an on optimal robust $\fett{b}$-flow of the maximum split instance $\mathcal{I}_n$}\label{RFfig:FlowBalanceConstraintsOfNotSetValues}
\end{figure}
As $f^1(a^c_{2i}) = n-i$ holds for $i\in [n-1]$, we obtain $f^1(a^d_{2i}) = i$ for $i\in [n-1]$. 
Due to the flow balance constraints, it holds $f^1(a^d_{2i-1}) \in \{i-1,i\}$ for $i\in [n]$. 
As $f^2(a^c_{2i}) = n-i$ holds for $i\in [n-1]$, we obtain $f^2(a) = 1$ for either shortcut arc $a=a^s_{2i-1}$ or $a=a^s_{2i}$ with $i\in [n]$. 
We note that each shortcut arc is included in only one path of the second scenario.
Thus, flow $f^2$ sends one unit along either path $p^2_{2i-1}$ or $p^2_{2i}$ for all $i\in [n]$. 
Because of the equal path property, flow $f^1$ also sends one unit along either path $p^1_{2i-1}$ or $p^1_{2i}$ for all $i\in [n]$. 
We refer to this property as \textit{one path of a pair property}.
Overall, we obtain the unique structure of an optimal robust flow where one unit is sent along one path of each pair in both scenarios. 

Finally, we conclude with remarks about the design of the arc cost of the maximum split instance. 
By choosing powers of two as arc cost, we achieve that a differently sent unit---compared to the unique structure---immediately incurs cost (in the first or second scenario) as high as it cannot be compensated afterwards. 
Furthermore, we choose the arc cost (excluding integers $s_i$, $i\in [2n]$) sufficiently large, a multiple greater than the sum of all integers, i.e., $2w=\sum_{i=1}^{2n}s_i\leq nw \leq Fw$. 
As a result, the unique structure of an optimal robust flow is independent of the \RFPairPartition{} instance.
We note that the unique structure of an optimal robust flow meets the desired trade-off between the best and worst case of the first and second scenario's cost. 
If the integers included in the arc cost are neglected, the significantly higher remaining cost is equal in both scenarios of an optimal robust flow as shown in the following.
By Lemma~\ref{RFlem:HilfslemmaCostOfFixedArcsForOptimalSol}, the cost incurred on the fixed arcs is $c^{\fixhelp}$ in both scenarios.  
If the integers included in the cost of the free (detour and shortcut) arcs are neglected, the remaining cost incurred on the free arcs is $c^{\freehelp}_{\setminus \text{int}}:=nw(2^{n+1}-n-2)$ in both scenarios.
The $n$-times use of shortcut arcs in the second scenario (as $f^2(a) = 1$ for either $a=a^s_{2i-1}$ or $a=a^s_{2i}$ with $i\in [n]$) causes the same cost than the use of detour arc $a^d_2$ once, plus detour arc $a^d_4$ twice, and so on, plus the $n$-times use of detour arc $a^d_{2n}$ in the first scenario (as $f^1(a^d_{2i}) = i$ for $i\in [n]$), i.e., 
$$n\cdot Fw=nw(2^{n+1}-n-2) =1 \cdot 2^{n-1}nw + 2\cdot 2^{n-2}nw + \ldots  + n \cdot 2^0 nw.$$
This relation explains the value of parameter $F= 2^{n+1}-n-2$. 
We note that the multiplication of the powers of two by $n$ ensures the integrality of parameter $F$ and thus the integrality of the arc cost.
Overall, we obtain a lower bound on the cost of an optimal robust flow by $c(\fett{f})\geq c(f^\lambda)\geq c^{\fixhelp} + c^{\freehelp}_{\setminus \text{int}}$, $\lambda\in \Lambda$.

\subsubsection{Special case of unique source and single sinks networks}
\label{Subsec:UniqueSourceSingleSink}
In this section, we analyze the complexity of the \ProblemName{} problem for the special case of networks based on SP digraphs with a unique source and single sinks. 
Before performing a reduction from the weakly $\mathcal{NP}$-complete \RFPairPartition{} problem, we define the decision version of the \ProblemName{} problem as follows.
\begin{definition}[Decision Version \ProblemName{} problem]
The decision version of the \ProblemName{} problem asks whether a robust flow exists with cost at most $\beta \in \mathbb{Z}_{\geq 0}$.
\end{definition}
\begin{theorem}
The decision version of the \ProblemName{} problem is weakly $\mathcal{NP}$-complete for networks based on SP digraphs with a unique source and single sinks, even if only two scenarios are considered. 
\end{theorem}
\begin{proof}
Let $\mathcal{I}$ be a \RFPairPartition{} instance with positive integers $s_1,\ldots,s_{2n}$ partitioned in pairs $S_{i}=\{s_{2i-1},s_{2i}\}$ with $s_{2i-1} \geq s_{2i}$ for $i\in [n]$ such that $\sum_{j=1}^{2n} s_j = 2w$ holds.
We consider the corresponding maximum split instance $\mathcal{I}_n= (G_n=(V_n,A_n), c, \boldsymbol{b})$ as visualized in Figure~\ref{RFfig:killerProof}. 
In the following, we prove that $\mathcal{I}$ is a Yes-instance if and only if there exists a robust $\boldsymbol{b}$-flow for \ProblemName{} instance $\mathcal{I}_{n}$ with cost of at most
$\beta :=w+ c^{\fixhelp} + c^{\freehelp}_{\setminus \text{int}}= w + (2^{n-1}n-2^n+1+F)nw $.

Let $S^1$, $S^2$ be a feasible partition for instance $\mathcal{I}$. 
We determine a robust $\fett{b}$-flow $\fett{f}=(f^1,f^2)$ as follows.
For every integer pair $S_{i}=\{s_{2i-1}, s_{2i}\}$, $i\in [n]$ separated such that $s_{2i}\in S^1$ and $s_{2i-1}\in S^2$ hold, we define for all arcs $a\in A_n$ the flows 
\begin{align*}
f^1_{s_{2i}}(a)
= \left\{ \begin{array}{ll}
1		& \mbox{for $a\in A(p^1_{2i})$}, \\
0 		& \mbox{otherwise,}
\end{array} \right.
\hspace*{1cm}
f^2_{s_{2i}}(a)
= \left\{ 
\begin{array}{ll}
1 		& \mbox{for $a\in A(p^2_{2i})$},\\ 
0 		& \mbox{otherwise}.
\end{array} 
\right.
\end{align*}
For every integer pair $S_{i}=\{s_{2i-1}, s_{2i}\}$, $i\in [n]$ separated such that $s_{2i-1}\in S^1$ and $s_{2i}\in S^2$ hold, we define for all arcs $a\in A_n$ the flows 
\begin{align*}
f^1_{s_{2i-1}}(a)
= \left\{ \begin{array}{ll}
1		& \mbox{for $a\in A(p^1_{2i-1})$}, \\
0 		& \mbox{otherwise,}
\end{array} \right.
\hspace*{1cm}
f^2_{s_{2i-1}}(a)
= \left\{ \begin{array}{ll}
1 		& \mbox{for $a\in A(p^2_{2i-1})$}, \\ 
0 		& \mbox{otherwise}.
\end{array} \right.
\end{align*}
Using these flows, we construct the first and second scenario flow by $f^1=\sum_{s_i\in S^1}f^1_{s_i}$ and $f^2=\sum_{s_i\in S^1}f^2_{s_i}=\sum_{s_{2i-1}\in S^2}f^2_{s_{2i}} \linebreak+ \sum_{s_{2i}\in S^2}f^2_{s_{2i-1}}$, respectively.
As each integer pair $S_{i}=\{s_{2i-1},s_{2i}\}$, $i\in [n]$ is separated in partition $S^1$, $S^2$, flow $f^\lambda$, $\lambda\in \Lambda$ sends one unit along either path $p^\lambda_{2i}$ or $p^\lambda_{2i-1}$ for all $i\in [n]$. 
Overall, flows $f^1$ and $f^2$ send $n$ units each from the source to their sinks and their values are equal on the fixed arcs. 
Consequently, robust flow $\fett{f}=(f^1,f^2)$ is feasible as the flow balance and consistent flow constraints are satisfied. 
For computing the cost of flows $f^1$ and $f^2$, we use that set $S^1$ as well as set $S^2$ contains one integer of each pair $S_{i}=\{s_{2i-1},s_{2i}\}$, $i\in [n]$.
We obtain
\begin{align*}
	c(f^1) 
	&  = \sum_{s_i\in S^1}c(f^1_{s_i})
	   	= \sum_{s_i\in S^1} \left( nw \left(2^{n-\lceil \frac{i}{2}\rceil} + 2^{n-1}-1 \right) + s_{i} \right)
		= \sum_{s_i\in S^1}  s_{i} +  nw\sum_{i=1}^{n}  \left(2^{n-i} + 2^{n-1}-1 \right)\\
	&= w + (2^{n-1}n-n+2^n-1)nw 
	  = w + (2^{n-1}n-n+2^{n+1}-2^n-2+1)nw 
	  = w + (2^{n-1}n-2^n+1+F)nw\\
	 &= w + c^{\fixhelp} + c^{\freehelp}_{\setminus \text{int}}
\end{align*}
and
\begin{align*}
	c(f^2) 
	& = \sum_{s_{2i-1}\in S^2}c(f^2_{s_{2i}}) + \sum_{s_{2i}\in S^2}c(f^2_{s_{2i-1}})\\
	&=  \sum_{s_{2i-1}\in S^2} \left( nw (2^{i}-2)2^{n-i-1}  + s_{2i-1} + Fw \right) + \sum_{s_{2i}\in S^2} \left( nw (2^{i}-2)2^{n-i-1}  + s_{2i} + Fw \right)\\
	&=  \sum_{s_{i}\in S^2} s_{i} 
	+ nw \sum_{i=1}^{n}  \left( (2^{i}-2)2^{n-i-1} + Fw \right) \\
	&= w +(2^{n-1}n-2^n+1+F)nw\\
	&= w +  c^{\fixhelp}  +c^{\freehelp}_{\setminus \text{int}}. 
\end{align*}
We have constructed a robust $\boldsymbol{b}$-flow $\boldsymbol{f}= (f^1 , f^2 )$ for \ProblemName{} \mbox{instance $\mathcal{I}_n$} with cost $c(\boldsymbol{f}) = w +  c^{\fixhelp}  +c^{\freehelp}_{\setminus \text{int}} = \beta$.

Conversely, let $\boldsymbol{f}= (f^1 , f^2 )$ be a robust $\boldsymbol{b}$-flow for \ProblemName{} \mbox{instance $\mathcal{I}_n$} with cost $c(\boldsymbol{f})=\max\{c(f^1), c(f^2)\} \leq \beta$. 
We may assume that flow $\fett{f}$ satisfies the unique structure of an optimal robust $\fett{b}$-flow of the \mbox{maximum split instance $\mathcal{I}_n$}, where a lower bound on the cost is given by 
$c(\fett{f})\geq c^{\fixhelp} + c^{\freehelp}_{\setminus \text{int}}$. 
If flow $\fett{f}$ sent one unit differently (i.e., inevitably differently on the fixed arcs), the cost $c(\fett{f})$ would exceed the value $\beta = w + c^{\fixhelp} + c^{\freehelp}_{\setminus \text{int}} $. 
We note that the gap between lower bound $c^{\fixhelp} + c^{\freehelp}_{\setminus \text{int}}$ and upper bound $\beta$ on the cost of flow $\fett{f}$ is $w$.
In the following, we determine the cost of the scenario flows $f^1$ and $f^2$.  
Let $f^\lambda(p^\lambda_i)$ denote the value which flow $f^\lambda$, $\lambda\in \Lambda$ sends along path $p^\lambda_i$, $i\in [2n]$.
Using Lemma~\ref{RFlem:HilfslemmaFlowValueOnFixedArcsOfOptimalSol}, we obtain for the cost of the first scenario
\begin{align*}
	 c(f^1) 
	&= \sum_{a\in A_n} c(a) \cdot f^1(a) \\
	&= \sum_{i=1}^{2n-1} c(a^c_i) \cdot f^1(a^c_i) + \sum_{i=1}^{2n}c(a^d_i)\cdot f^1(a^d_i) \\
	& =
	\sum^{n-1}_{i=1} 2^{n-i-1}nw \cdot f^1(a^c_{2i}) + \sum_{i=1}^{n} (s_{2i-1} - s_{2i}) \cdot f^1(a^d_{2i-1}) +  
	\sum_{i=1}^{n-1} (s_{2i} - s_{2i+1} + 2^{n-i}nw) \cdot f^1(a^d_{2i}) \\  
	& \phantom{=} + (s_{2n} + 2^{0}nw) \cdot f^1(a^d_{2n})  \\
	& = \sum^{n-1}_{i=1} 2^{n-i-1}nw\cdot (n-i) + \sum_{i=1}^{n} (s_{2i-1} - s_{2i}) \cdot f^1(a^d_{2i-1}) + \sum_{i=1}^{n-1} (s_{2i} - s_{2i+1} + 2^{n-i}nw) \cdot i + (s_{2n} + 2^{0}nw) \cdot n\\
	& = \sum^{n-1}_{i=1} 2^{n-i-1}nw\cdot (n-i) + \sum_{i=1}^{n}  2^{n-i}nw \cdot i + \sum_{i=1}^{n} (s_{2i-1} - s_{2i}) \cdot f^1(a^d_{2i-1}) + \sum_{i=1}^{n-1} (s_{2i} - s_{2i+1} ) \cdot i + s_{2n}  \cdot n\\
	&= c^{\fixhelp} + c^{\freehelp}_{\setminus \text{int}}  + \sum_{i=1}^{n} (s_{2i-1} - s_{2i}) \cdot f^1(a^d_{2i-1}) + \sum_{i=1}^{n-1} (s_{2i} - s_{2i+1} ) \cdot i + s_{2n}  \cdot n.
\end{align*} 
Considering the last part of the expression above, we define $\alpha :=\sum_{i=1}^{n} (s_{2i-1} - s_{2i}) \cdot f^1(a^d_{2i-1}) + \sum_{i=1}^{n-1} (s_{2i} - s_{2i+1} ) \cdot i + s_{2n}  \cdot n$. 
Before simplifying term $\alpha$, we note the following. 
Due to the one path of a pair property, it holds $f^1(p^1_{2i-1}) = 1- f^1(p^1_{2i})\in \{0,1\}$ for all $i\in[n]$.
Furthermore, it holds $f^1(a^d_{2i-1})\in\{i-1,i\}$, where $f^1(a^d_{2i-1})=i-1$ if $f^1(p^1_{2i-1})=0$ and $f^1(a^d_{2i-1})=i$ if $f^1(p^1_{2i-1})=1$. 
Overall, we obtain $f^1(a^d_{2i-1})=i-1+ f^1(p^1_{2i-1})$ for $i\in [n]$.
Using this as well as Lemma~\ref{RFlem:HilfslemmaFlowValueOnFixedArcsOfOptimalSol} and the one path of a pair property, we perform the following transformation 
\begin{align*}
\alpha =&\sum_{i=1}^{n} (s_{2i-1} - s_{2i}) \cdot f^1(a^d_{2i-1}) + \sum_{i=1}^{n-1} (s_{2i} - s_{2i+1} ) \cdot i + s_{2n}  \cdot n \\
= &\sum_{i=1}^{n} (s_{2i-1} - s_{2i}) \cdot (i-1+ f^1(p^1_{2i-1}))+ \sum_{i=1}^{n-1} (s_{2i} - s_{2i+1} ) \cdot i + s_{2n}  \cdot n \\
= &\sum_{i=1}^{n} \left(s_{2i-1} \cdot f^1(p^1_{2i-1})  +  s_{2i} \cdot (1-f^1(p^1_{2i-1})) \right) \\
= &\sum_{i=1}^{n} \left(s_{2i-1} \cdot f^1(p^1_{2i-1})  +  s_{2i} \cdot f^1(p^1_{2i}) \right) \\
= &\sum_{i=1}^{2n} s_{i} \cdot f^1(p^1_{i}).
\end{align*}
Overall, we obtain
\begin{align*}
	c(f^1) = c^{\fixhelp} + c^{\freehelp}_{\setminus \text{int}} + \sum_{i=1}^{2n} s_{i} \cdot f^1(p^1_{i}).
\end{align*} 
Using Lemma~\ref{RFlem:HilfslemmaFlowValueOnFixedArcsOfOptimalSol}, the one path of a pair property, and the fact that $f^2(a^s_i) = f^2(p^2_i)$ holds for all $i\in [2n]$, we obtain for the cost of the second scenario
\begin{align*}
	 c(f^2) 
	&= \sum_{a\in A_n} c(a) \cdot f^2(a) \\
	&= \sum_{i=1}^{2n-1} c(a^c_i) \cdot f^2(a^c_i) + \sum_{i=1}^{2n}c(a^s_i)\cdot f^2(a^s_i) \\
	& = 
	\sum^{n-1}_{i=1} 2^{n-i-1}nw \cdot f^2(a^c_{2i}) + \sum_{i=1}^{n} \left( (s_{2i-1} + Fw) \cdot f^2(a^s_{2i})  +  (s_{2i}+Fw) \cdot f^2(a^s_{2i-1})\right)\\
	& = 
	\sum^{n-1}_{i=1} 2^{n-i-1}nw \cdot (n-i)+ n\cdot Fw + \sum_{i=1}^{n} \left( s_{2i-1}  \cdot f^2(p^2_{2i})  +  s_{2i} \cdot f^2(p^2_{2i-1})\right)\\
	&=
		c^{\fixhelp} + c^{\freehelp}_{\setminus \text{int}} + \sum_{i=1}^{n} \left( s_{2i-1}  \cdot f^2(p^2_{2i})  +  s_{2i} \cdot f^2(p^2_{2i-1})\right).
\end{align*} 
As $c(f^\lambda)\leq\beta=c^{\fixhelp} + c^{\freehelp}_{\setminus \text{int}}+w$ holds for $\lambda\in \Lambda$, we obtain
\begin{align}\label{boundOnIntegersFirstScenario}
	\sum_{i=1}^{2n} s_{i} \cdot f^1(p^1_{i})     \leq w
\end{align}
and
\begin{align}\label{boundOnIntegersSecondScenario}
\sum_{i=1}^{n}  \left( s_{2i-1}  \cdot f^2(p^2_{2i})  +  s_{2i} \cdot f^2(p^2_{2i-1})\right) 
\leq w. 
\end{align}
In the next step, we define a partition $S^1$, $S^2$ by
\begin{align*}
S^1 &= \{s_{2i-1}  \mid f^1(p^1_{2i-1}) = 1,\ i\in [n]\} \cup \{ s_{2i} \mid f^1(p^1_{2i})  = 1,\ i\in [n]\}, \\
S^2 &= \{s_{2i}  \mid  f^2(p^2_{2i-1}) = 1,\  i\in [n]\} \cup \{ s_{2i-1} \mid f^2(p^2_{2i}) = 1,\ i\in [n]\}.
\end{align*}
Due to the equal paths and the one path of a pair property, either $f^1(p^1_{2i-1})= f^2(p^2_{2i-1})=1$ or $f^1(p^1_{2i})= f^2(p^2_{2i})=1$ holds for all $i\in [n]$. 
For this reason, neither subset $S^1$ nor $S^2$ includes both integers of an integer pair. 
Instead, by definition of the sets, one integer of every integer pair $S_{i}=\{s_{2i-1},s_{2i}\}$, $i\in [n]$ is included in set $S^1$ and the other in set $S^2$. 
Consequently, partition $S^1$, $S^2$ contains the positive integers $s_1,\ldots,s_{2n}$, it is disjoint, and it separates all integer pairs.  
Expressions~\eqref{boundOnIntegersFirstScenario} and~\eqref{boundOnIntegersSecondScenario} are equivalent to expressions 
\begin{align*}
	\sum_{s_i\in S^1}s_i \leq w
	\text{\hspace*{0.25cm} and \hspace*{0.25cm}}
	\sum_{s_i\in S^2}s_i \leq w,
\end{align*} 
respectively. 
Thus, it holds
\begin{align*}
	\sum_{s_i\in S^1}s_i = \sum_{s_i\in S^2}s_i = w. 
\end{align*}
Finally, we have constructed a feasible partition $S^1$, $S^2$ for the \RFPairPartition{} instance $\mathcal{I}$. 
\end{proof}
As a result, we obtain the following corollary.
\begin{corollary}
	The decision version of the \ProblemName{} problem is weakly $\mathcal{NP}$-complete for networks based on SP digraphs with multiple sources and multiple sinks, even if only two scenarios are considered. \label{Cor:SPComplexity}
\end{corollary}
In the special case of a constant number of scenarios, we can solve the \ProblemName{} problem for networks based on SP digraphs with multiple sources and multiple sinks by the pseudo-polynomial-time algorithm presented in our previous work~\cite{buesing2020robust}. 
The algorithm is based on dynamic programming. 
The runtime of the algorithm depends on the arc capacities required as input. 
For the \ProblemName{} problem, we may set the capacity of every arc to the maximum total supply among all scenarios.

\subsection{Special case of unique source \& unique sink networks}
\label{Subsubsec:UniqueSourceUnqiueSink}
In this section, we analyze the complexity of the \ProblemName{} problem for the special case of networks based on SP digraphs with a unique source and a unique sink. 
We provide a polynomial-time algorithm based on the Algorithm 5.1 of our previous work~\cite{buesing2020robust}. %and can be used by choice of sufficient large arc capacities.
Let $(G,c,\boldsymbol{b})$ be a \ProblemName{} instance on SP digraph $G=(V,A=\fix\cup\free)$ with unique source $s$ and unique sink $t$. 
Without loss of generality, we assume that the unique source and unique sink comply with the origin and target of SP digraph $G$, respectively.
Further, we assume that the scenarios are non-decreasingly ordered with respect to the supply of the unique source, i.e., $b^1(s)\leq \ldots \leq b^{|\Lambda|}(s)$. 

Before presenting the algorithm, we note the following on the basis of our previous work~\cite{buesing2020robust}. 
%As we consider a unique source and unique sink network based on an SP digraph, the cost of a robust $\fett{b}$-flow $\fett{f}=(f^1,\ldots,f^{|\Lambda|})$ is determined by the scenario flow with the maximum supply, i.e., $c(\fett{f})= c(f^{|\Lambda|})$. 
%Furthermore, 
There exists a robust $\fett{b}$-flow that sends in each scenario $\lambda\in\Lambda$ the so-called excess supply $b^\lambda(s)-\min_{\lambda^{\prime}\in \Lambda}b^{\lambda^{\prime}}(s)= b^\lambda(s)-b^{1}(s)$ (or the number of units which the arc capacities allow) along a shortest path (with respect to arc cost $c$) in digraph $G-\fix$. 
As we consider an uncapacitated network, the total excess supply of each scenario can be sent along a shortest path in digraph $G-\fix$.  
Thus, Algorithm 5.1 of our previous work~\cite{buesing2020robust} reduces to the computation of two shortest paths---one in SP digraph $G$ and one in digraph $G-\fix$.
We obtain the following simplified algorithm. 
\begin{algorithm}[H]
	\caption{} \label{RFAlg:2ShortestPathForUniqueSourceUniqueSink}
	\begin{lyxlist}{Initialization:}
		\item [{Input:}] \ProblemName{} instance $(G,c,\boldsymbol{b})$ with unique source $s$, unique sink $t$ where $G=(V,A=\fix\cup\free)$ is an SP digraph
		\item [{Output:}] Robust minimum cost $\boldsymbol{b}$-flow $\boldsymbol{f}$
		\item [{Method:}]~ 	\end{lyxlist}
	\begin{algorithmic}[1]
		\State{Compute shortest $(s,t)$-path $p^{\freehelp}$ in digraph $G-\fix$} 
		\State{Compute shortest $(s,t)$-path $p^{\fixhelp}$ in SP digraph $G$} 
		\If{$c(p^{\fixhelp})< c(p^{\freehelp})$}
		\State{For every scenario $\lambda\in \Lambda$ determine the flow
			\begin{align}\label{RFequ:RobFlowOptUniqueSourceUniqueSink}
				f^\lambda(a) = \begin{cases}
					b^{1}(s)
					& \text{for arcs $a\in A(p^{\fixhelp})\setminus A(p^{\freehelp})$,} \\
					b^\lambda(s)-  b^{1}(s) 					
					& \text{for arcs $a\in A(p^{\freehelp})\setminus A(p^{\fixhelp})$},\\
					b^\lambda(s)				
					& \text{for arcs $a\in A(p^{\fixhelp}) \cap A(p^{\freehelp})$},\\
					0 & \text{otherwise}
				\end{cases}
			\end{align}}
		\Else 
		\State{For every scenario $\lambda\in \Lambda$ determine the flow
			\begin{align*}%\label{RFequ:RobFlowOptUniqueSourceUniqueSink2}
				f^\lambda(a) = \begin{cases}
					 b^{\lambda}(s)
					& \text{for arcs $a\in A(p^{\freehelp})$,} \\
					0 & \text{otherwise}
				\end{cases}
			\end{align*}}
		\EndIf
		\Return Robust $\boldsymbol{b}$-flow $\boldsymbol{f}=(f^1,\ldots,f^{|\Lambda|})$
	\end{algorithmic}
\end{algorithm}
We obtain the polynomial-time solvability of the \ProblemName{} problem for networks based on SP digraphs with a unique source and a unique sink as shown in the following theorem. 
\begin{theorem}\label{RFthm:PolyAlgUniqueSourceUniqueSink}
	Let $(G,c,\boldsymbol{b})$ be a \ProblemName{} instance on SP digraph $G=(V,A=\fix\cup\free)$ with unique source $s$ and unique sink $t$. 
	Algorithm~\ref{RFAlg:2ShortestPathForUniqueSourceUniqueSink} computes an optimal robust $\fett{b}$-flow in $\mathcal{O}(|V|+|A|)$ time.
\end{theorem}
\begin{proof}
The correctness of Algorithm~\ref{RFAlg:2ShortestPathForUniqueSourceUniqueSink} follows from the correctness of the algorithm of our previous work~\cite{buesing2020robust} with arc capacities set to the maximum supply among all scenarios, i.e., $u\equiv b^{|\Lambda|}(s)$. 
As computing shortest paths in acyclic digraphs can be done in $\mathcal{O}(|V|+|A|)$~\cite{cormen2022introduction}, we obtain the polynomial runtime of Algorithm~\ref{RFAlg:2ShortestPathForUniqueSourceUniqueSink}.
\end{proof}
For an alternative algorithm, we refer to our preliminary version of this paper~\cite{busing2022complexity}. 
The alternative algorithm provides further insights into the \ProblemName{} problem for networks based on SP digraphs with a unique source and a unique sink.
It uses the SP structure to shrink the SP digraph to only one multi-arc on which the \ProblemName{} problem is solved.
Before concluding this section, we present a new result about the structure of an existing optimal robust flow, which we use in the next section. 
\begin{lemma}\label{ClaimPropertyCheapestFlowUniqueSourceAndSink}
	Let $\mathcal{I}=(G,c,\boldsymbol{b})$ be a \ProblemName{} instance on SP digraph $G=(V,A=\fix\cup\free)$ with unique source $s$ and unique sink $t$. 	
	There exists an optimal robust $\boldsymbol{b}$-flow $\fett{\tilde{f}}=(\tilde{f}^1, \ldots, \tilde{f}^{|\Lambda|})$ such that for all feasible robust $\boldsymbol{b}$-flows $\boldsymbol{f}=(f^1,\ldots, f^{|\Lambda|})$ it holds
	$c(\tilde{f}^\lambda)\leq c(f^\lambda)$ for all $\lambda\in \Lambda$.
\end{lemma}
\begin{proof}
	Let $p^{\fixhelp}$ and $p^{\freehelp}$ be shortest $(s,t)$-paths in digraphs $G$ and $G-\fix$, respectively. 
	Without loss of generality, we assume  that $c(p^{\fixhelp})< c(p^{\freehelp})$ holds. 
	Let an optimal robust $\fett{b}$-flow $\fett{\tilde{f}}=(\tilde{f}^1,\ldots,\tilde{f}^{|\Lambda|})$ be given as determined in expression~\eqref{RFequ:RobFlowOptUniqueSourceUniqueSink}. 
	We claim that for all feasible robust $\boldsymbol{b}$-flows $\fett{f}=(f^1, \ldots, f^{|\Lambda|})$ it holds $c(\tilde{f}^\lambda)\leq c(f^\lambda)$ for all $\lambda\in \Lambda$.
	Assume this does not hold true, i.e., there exists a feasible robust $\boldsymbol{b}$-flow $\fett{\hat{f}}=(\hat{f}^1, \ldots, \hat{f}^{|\Lambda|})\neq \fett{\tilde{f}}$ such that $c(\hat{f}^{\lambda^{\prime}})< c(\tilde{f}^{\lambda^{\prime}})$ holds for at least one scenario $\lambda^{\prime}\in \Lambda$.
	Let $\mathcal{P}=(p_1,\ldots,p_k)$ be an $(s,t)$-path decomposition of flow $\hat{f}^{\lambda^{\prime}}$, i.e., it holds $\hat{f}^{\lambda^{\prime}} = \sum_{i=1}^{k}\hat{f}^{\lambda^{\prime}}_{|p_i}$ where $\hat{f}^{\lambda^{\prime}}_{|p_i}$ indicates the part of flow $\hat{f}$ restricted to path $p_i$. 
	We assume that the paths are ordered such that the first $r$ paths may contain fixed and free arcs and the last $k-r$ paths contain free arcs only, i.e., $p_1,\ldots,p_r \subseteq G$ and $p_{r+1},\ldots,p_{k} \subseteq G-\fix$. 
	We note that $c(p^{\fixhelp})\leq c(p_i)$ holds for all $i\in [r]$ and $c(p^{\freehelp})\leq c(p_j)$ holds for all $j\in \{r+1,\ldots,k\}$. 
	Every scenario flow $f^\lambda$, $\lambda\in \Lambda$ of a feasible robust $\fett{b}$-flow $\fett{f}$ sends at least $b^\lambda(s)-b^1(s)$ units in digraph $G-\fix$ due to the consistent flow constraints.  
	This means that the value $\sum_{j=r+1}^{k} \hat{f}^{\lambda^{\prime}}(p_j)$ of the flow $\sum_{j=r+1}^{k} \hat{f}^{\lambda^{\prime}}_{|p_j}$ is at least $b^{\lambda^{\prime}}(s)-b^1(s)$. 
	If the value of the flow $\sum_{j=r+1}^{k} \hat{f}^{\lambda^{\prime}}_{|p_j}$ is greater than or equal to $b^{\lambda^{\prime}}(s)-b^1(s)$, the value of flow $\sum_{i=1}^{r} \hat{f}^{\lambda^{\prime}}_{|p_i}$ is less than or equal to $b^1(s)$, respectively. 
	Using the $(s,t)$-path decomposition, we obtain a lower bound on the cost of flow $\hat{f}^{\lambda^{\prime}}$ as follows
	\begin{align*}
		c(\hat{f}^{\lambda^{\prime}}) 
		&= \sum_{p\in \mathcal{P}} c(p)\cdot \hat{f}^{\lambda^{\prime}}(p)\\
		&= \sum_{i=1}^{r} c(p_i)\cdot \hat{f}^{\lambda^{\prime}} (p_i) + \sum_{j=r+1}^{k} c(p_j)\cdot \hat{f}^{\lambda^{\prime}}(p_j) \\
		&\stackrel{(1)}{\geq} c(p^{\fixhelp}) \cdot \sum_{i=1}^{r} \hat{f}^{\lambda^{\prime}}(p_i) + c(p^{\freehelp}) \cdot \sum_{j=r+1}^{k} \hat{f}^{\lambda^{\prime}}(p_j)\\
		&\stackrel{(2)}{\geq} c(p^{\fixhelp})\cdot b^1(s) + c(p^{\freehelp}) \cdot (b^{\lambda^{\prime}}(s)-b^1(s))\\
		&=c(\tilde{f}^{\lambda^{\prime}}). 
	\end{align*}
	As we consider a robust $\fett{b}$-flow $\fett{\hat{f}}\neq \fett{\tilde{f}}$, at least one of the inequalities $(1)$, $(2)$ is strict.
	We obtain a contradiction to the assumption. 
\end{proof}
\subsection{Special cases of unique source and parallel sinks \& parallel sources and unique sink networks}
\label{Subsec:UniqueSourceParallelSinks}
In this section, we analyze the complexity of the \ProblemName{} problem for the special case of networks based on SP digraphs with a unique source and parallel (multiple) sinks. 
First, we introduce terms and notations.  
Subsequently, we present structural results for a feasible robust flow arising from the considered special case.  
Finally, based on these structural results, we provide a polynomial-time algorithm. 
By redirecting the arcs and identifying the source as sink and the sinks as sources, the results can be analogously applied for the special case of networks based on SP digraphs with parallel (multiple) sources and a unique sink. 

Let $(G,c,\boldsymbol{b})$ be a \ProblemName{} instance on SP digraph $G$ with a unique source and parallel sinks. 
Without loss of generality, we assume that the unique source, denoted by $s$, complies with the origin $o_G$ of the SP digraph $G$. 
Let $\gamma_t(\lambda)$ indicate the total number of sinks in scenario $\lambda\in \Lambda$. 
The $j$-th sink of scenario $\lambda\in \Lambda$ is specified by $t^\lambda_{j}$ with $j\in [\gamma_t(\lambda)]$.
We say a subgraph $G_{vw}\subseteq G$ is \textit{spanned} by two vertices $v,w\in V(G)$ (for which a $(v,w)$-path exists) if subgraph $G_{vw}$ is induced by vertices reachable from vertex $v$ and from which vertex $w$ is reachable, i.e., 
$V(G_{vw})= \{ x \in V(G)\mid \text{there exist a } (v,x)\text{- and a } \linebreak[1](x,w)\text{-path}  \}.$ 
A spanned subgraph $G_{vw}$ of an SP digraph is an SP digraph itself with origin $v$ and target $w$ as shown in Lemma~\ref{RFlem:SubgraphOfSPGraphsDefinedByTwoVertices} in Appendix~\ref{RF:AppendixD}.

In the next step, we define \textit{vertex labels} $\pi:V(G) \rightarrow 2^{V(G)}$ that indicate at every vertex $v\in V(G)$ the reachable sinks, i.e., $\pi (v) = \{ t \in V(G)\mid \text{$t$ is a sink and there exists a $(v,t)$-path} \}$.
As all vertices are reachable from the origin in an SP digraph, the label of the unique source is given by the set of all sinks, i.e., $\pi(s)= \{t^{\lambda}_{i} \mid i\in [\gamma_t(\lambda)], \lambda\in \Lambda \}$.
We note that the labels form inclusion chains from the unique source $s$ to the sinks $t^\lambda_i$, $i\in [\gamma_t(\lambda)]$, $\lambda\in \Lambda$. 
We say two labels $\pi(v),\pi(w)$ of vertices $v,w\in V(G)$ are \textit{$(v,w)$-consecutive} if $\pi(w)\subsetneq \pi(v)$ holds and there does not exist a vertex $z\in V(G)$ with label $\pi(z)$ such that $\pi(w)\subsetneq\pi(z)\subsetneq \pi(v)$ holds. 
Furthermore, we say a vertex $v\in V(G)$ is a \textit{last vertex} (of label $\pi(v)$) if 
\begin{itemize}
	\item[(i)] its label $\pi(v)$ is nonempty,  
	\item[(ii)] for each vertex $w\in V(G)$ for which an arc $a=(v,w)\in A(G)$ exists the labels $\pi(v)$, $\pi (w)$ are $(v,w)$-consecutive. 
\end{itemize} 
As the labels are inclusion chains, a last vertex $v\in V(G)$ with label $\pi(v)$ implies that for each vertex $w\in V(G)$ for which a $(v,w)$-path exists it holds $\pi(w)\subsetneq \pi(v)$. 
A last vertex of a label is uniquely determined as shown in the following lemma. 
\begin{lemma}\label{RFlem:OneLastVertexOfEveryLabel}
	Let $(G,c,\boldsymbol{b})$ be a \ProblemName{} instance on SP digraph $G$ with a unique source and parallel sinks. For each label, there is a uniquely determined last vertex.
\end{lemma}
\begin{proof}
%	Firstly, we note that by definition each sink $t^\lambda_i$, $i\in [\gamma_t(\Lambda)]$, $\lambda\in\Lambda$ is a last vertex whose label contains only the sink itself.
	Firstly, we note that the label of each sink contains only the sink itself.
	Clearly, for the label of a sink, the uniquely determined last vertex is the sink itself. 
	As the labels of the sinks are the only one element labels, the statement is true for all labels with cardinality one.  
	Assume now the statement is false for a label of cardinality at least two.
	There exist at least two last vertices $v,w\in V(G)$ of the same label $\pi(v)=\pi(w)$ with $|\pi(v)|=|\pi(w)|\geq2$. 
	The vertices are parallel, otherwise we immediately obtain a contradiction to the definition of a last vertex.
	If we consider an SP tree $T$ of SP digraph $G$, we see that there exist $P$-vertices by which the sinks $t^\lambda_i$, $i\in [\gamma_t(\Lambda)]$, $\lambda\in\Lambda$ are composed in parallel. 
	We note that if two subgraphs, each containing one sink, are composed in parallel, the two sinks may not be the origin or target of the composition.
	Otherwise, there exists a path between two sinks which contradicts the definition of parallel sinks. 	
	This means that two arcs, where each arc is adjacent to a sink, cannot be straightly composed in parallel but that a prior series composition is required. 
	There exists a $P$-vertex $x\in V(T)$ whose associated subgraph $G_x\subseteq G$ is the first subgraph that contains the parallel sinks of set $\pi(v)= \pi(w)$. 
	More precisely, for $x\in V(T)$ it holds $\pi(v)=\pi(w)\subseteq V(G_x)$ but for all descendants $y\in V(T)$ of vertex $x$ it holds $\pi(v)=\pi(w) \not \subseteq V(G_y)$. 
	Furthermore, we note that none of the sinks complies with the origin and target of subgraph $G_x$, i.e., $ t^{\lambda}_i \neq o_{G_x}$ and $t^{\lambda}_i \neq q_{G_x}$, $i\in [\gamma_t(\lambda)]$, $\lambda\in \Lambda$. 
	An example is visualized in Figure~\ref{fig:LastVertex1}.
	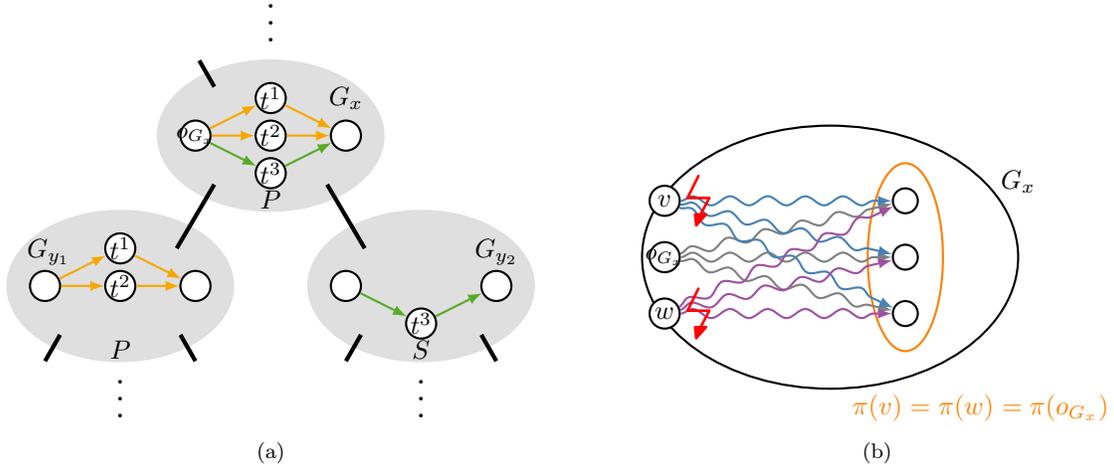
\begin{figure}
		\centering
		\begin{adjustbox}{max width=\textwidth}
			\subfloat[]{\label{fig:LastVertex1}	\centering\scalebox{1}{\begin{tikzpicture}
%	
%	\Shape{(-0.5,0.1) (1.75,2.75) (4.2,2.75) (4.2,-2.75) (1.75,-2.15)  }{opacity=0.3, gray!60}{none,fill=gray!60};
%	\node[gray!60] at (0.5,2){$\widetilde{G}$}; 
%	

	\Ellipse{1,0}{1.5}{1}{thick, fill=lightgray!50, draw=lightgray!50}

	\Knoten{(1)}{(0,0)}{\scalebox{0.9}{$o_{G_x}$}}{thick, minimum size=0.4cm}
	\Knoten{(2)}{(2,0)}{}{thick, minimum size=0.4cm}
	\Knoten{(3)}{(1,0.5)}{$t^1$}{thick, minimum size=0.4cm}
	\Knoten{(4)}{(1,-0)}{$t^2$}{thick, minimum size=0.4cm}
	\Knoten{(5)}{(1,-0.5)}{$t^3$}{thick, minimum size=0.4cm}

	\draw[thick, -latex, rwth-orange] (1) -- (3); 
	\draw[thick, -latex, rwth-orange] (3) -- (2); 
	\draw[thick, -latex, rwth-orange] (1) -- (4); 
	\draw[thick, -latex, rwth-orange] (4) -- (2); 
	\draw[thick, -latex, rwth-green] (1) -- (5); 
	\draw[thick, -latex, rwth-green] (5) -- (2);

	%Kind links
	\begin{scope}[shift={(-2,-2)}]
		\Ellipse{1,0}{1.5}{1}{thick, fill=lightgray!50, draw=lightgray!50}
				
		\Knoten{(1)}{(0,0)}{}{thick, minimum size=0.4cm}
		\Knoten{(2)}{(2,0)}{}{thick, minimum size=0.4cm}
		\Knoten{(3)}{(1,0.5)}{$t^1$}{thick, minimum size=0.4cm}
		\Knoten{(4)}{(1,-0)}{$t^2$}{thick, minimum size=0.4cm}
		
		\draw[thick, -latex, rwth-orange] (1) -- (3); 
		\draw[thick, -latex, rwth-orange] (3) -- (2); 
		\draw[thick, -latex, rwth-orange] (1) -- (4); 
		\draw[thick, -latex, rwth-orange] (4) -- (2); 
		\node at (1,-0.85){$P$}; 
	\end{scope}

	%Kind rechts
	\begin{scope}[shift={(2,-2)}]
		\Ellipse{1,0}{1.5}{1}{thick, fill=lightgray!50, draw=lightgray!50}
				
		\Knoten{(1)}{(0,0)}{}{thick, minimum size=0.4cm}
		\Knoten{(2)}{(2,0)}{}{thick, minimum size=0.4cm}
		\Knoten{(5)}{(1,-0.5)}{$t^3$}{thick, minimum size=0.4cm}
		
		\draw[thick, -latex, rwth-green] (1) -- (5); 
		\draw[thick, -latex, rwth-green] (5) -- (2); 
		\node at (1,-0.85){$S$}; 
	\end{scope}
	
	%Beschriftung
	\node at (1,-0.85){$P$}; 
	\node at (2,0.5){$G_x$}; 
	\node at (-1.95,-1.55){$G_{y_1}$}; 
	\node at (4,-1.55){$G_{y_2}$}; 
	
	%Baumkanten
	\draw[ultra thick] (0.25,-0.65) -- (-0.25,-1.5); 	
	\draw[ultra thick] (1.75,-0.65) -- (2.25,-1.5); 	
	
	\draw[ultra thick] (-1.8,-2.65) -- (-2,-3); 
	\draw[ultra thick] (-0.2,-2.65) -- (0,-3); 	
	
	\begin{scope}[shift={(4,0)}]
		\draw[ultra thick] (-1.8,-2.65) -- (-2,-3); 
		\draw[ultra thick] (-0.2,-2.65) -- (0,-3); 	
	\end{scope}

	\begin{scope}[shift={(0.25,3.65)}]
		\draw[ultra thick] (-0.2,-2.65) -- (0,-3); 	
	\end{scope}

	\node[rotate=-90] at (-1,-3.5){\Large$\ldots$}; 
	\node[rotate=-90] at (3,-3.5){\Large$\ldots$}; 
	\node[rotate=-90] at (1,1.5){\Large$\ldots$}; 
		
\end{tikzpicture}}}
			\hspace*{1cm}
			\subfloat[]{\label{fig:LastVertex2}	\centering\scalebox{1}{\begin{tikzpicture}
	%	
	%	\Shape{(-0.5,0.1) (1.75,2.75) (4.2,2.75) (4.2,-2.75) (1.75,-2.15)  }{opacity=0.3, gray!60}{none,fill=gray!60};
	%	\node[gray!60] at (0.5,2){$\widetilde{G}$}; 
	%	
	
%	\Knoten{(1)}{(0,0)}{}{thick, minimum size=0.4cm}
%	\Knoten{(2)}{(2,0)}{}{thick, minimum size=0.4cm}
%	\Knoten{(3)}{(1,0.5)}{$t^1$}{thick, minimum size=0.4cm}
%	\Knoten{(4)}{(1,-0)}{$t^2$}{thick, minimum size=0.4cm}
%	\Knoten{(5)}{(1,-0.5)}{$t^3$}{thick, minimum size=0.4cm}
%	
%	
%	\draw[thick, -latex] (1) -- (3); 
%	\draw[thick, -latex] (3) -- (2); 
%	\draw[thick, -latex] (1) -- (4); 
%	\draw[thick, -latex] (4) -- (2); 
%	\draw[thick, -latex] (1) -- (5); 
%	\draw[thick, -latex] (5) -- (2); 
	
	\Ellipse{0,0}{2.5}{1.75}{thick}
	\Ellipse{1,0}{0.5}{1.25}{thick, draw=orange}
	
	\node[orange] at (2,-2){$\pi(v)=\pi(w)=\pi(o_{G_x})$}; 
	
	\Knoten{(1)}{(1,0)}{}{thick, minimum size=0.15cm}
	\Knoten{(2)}{(1,0.75)}{}{thick, minimum size=0.15cm}
	\Knoten{(3)}{(1,-0.75)}{}{thick, minimum size=0.15cm}
	\Knoten{(4)}{(-2.2,0)}{\scalebox{0.9}{\textcolor{black}{$o_{G_x}$}}}{thick, minimum size=0.4cm, draw=black}
	
	\draw[thick, -latex, decorate,decoration={snake,amplitude=.6mm,segment length=5mm,post length=2mm}, gray] (4) -- (1);
	\draw[thick, -latex, decorate,decoration={snake,amplitude=.6mm,segment length=5mm,post length=2mm}, in=120, out=45, gray ] (4) -- (2); 
	\draw[thick, -latex, decorate,decoration={snake,amplitude=.6mm,segment length=5mm,post length=2mm}, in=210, out=-45, gray] (4) -- (3);  

	\Knoten{(5)}{(-2.2,0.75)}{$v$}{thick, minimum size=0.4cm}
	\Knoten{(6)}{(-2.2,-0.75)}{$w$}{thick, minimum size=0.4cm}

	\draw[thick, -latex, decorate,decoration={snake,amplitude=.6mm,segment length=5mm,post length=2mm}, SteelBlue] (5) -- (1);
	\draw[thick, -latex, decorate,decoration={snake,amplitude=.6mm,segment length=5mm,post length=2mm}, SteelBlue] (5) -- (2);
	\draw[thick, -latex, decorate,decoration={snake,amplitude=.6mm,segment length=5mm,post length=2mm}, SteelBlue] (5) -- (3);
	\draw[thick, -latex, decorate,decoration={snake,amplitude=.6mm,segment length=5mm,post length=2mm}, Brombeer] (6) -- (1);
	\draw[thick, -latex, decorate,decoration={snake,amplitude=.6mm,segment length=5mm,post length=2mm}, Brombeer] (6) -- (2);
	\draw[thick, -latex, decorate,decoration={snake,amplitude=.6mm,segment length=5mm,post length=2mm}, Brombeer] (6) -- (3);	
	
%	\draw[thick, -latex, decorate,decoration={snake,amplitude=.6mm,segment length=5mm,post length=2mm},] (5) -- (4);
%	\draw[thick, -latex, decorate,decoration={snake,amplitude=.6mm,segment length=5mm,post length=2mm}] (6) -- (4);
	
	\node[red] at (-1.75,0.75){\Huge\lightning}; 
	\node[red] at (-1.75,-0.75){\Huge\lightning}; 
	
	\node[] at (2.5,1){$G_x$}; 

\end{tikzpicture}}}
		\end{adjustbox}
		\caption{(a) Subgraph $G_x$ for $x\in V(T)$ of SP tree $T$ is the first subgraph that contains sinks $t^1,t^2,t^3$ and its origin $o_{G_x}$ is a last vertex with label $\pi(o_{G_x})=\{t^1,t^2,t^3\}$ (b) Subgraph $G_x$ for $x\in V(T)$ of SP tree $T$ is the first subgraph that contains sinks $\pi(v)=\pi(w)$ and its origin $o_{G_x}$ is a last vertex with label $\pi(o_{G_x})=\pi(v)=\pi(w)$}
	\end{figure}
	If we further follow the SP tree's instructions to compose SP digraph $G$, we continue composing subgraphs serially at the origin or target or in parallel at the origin and target of subgraph $G_x$. 
	Consequently, the access from the unique source $s$ to the sinks of set $\pi(v)=\pi(w)$ remains given only by the origin $o_{G_x}$ of subgraph $G_x$. 
	By choice of subgraph $G_x$, origin $o_{G_x}$ is a last vertex of the same label as the last vertices $v,w$, i.e., $\pi(o_{G_x})=\pi(v)=\pi(w)$. 
	Thus, vertices $v,w$ need to be parallel to origin $o_{G_x}$, as visualized in Figure~\ref{fig:LastVertex2}, or one of them complies with origin $o_{G_x}$. %they are predecessors of origin $o_{G_x}$
	In both cases, we obtain a contradiction as the access of the sinks is only possible via origin $o_{G_x}$ as discussed above. 
\end{proof}
We save the label structure by means of a directed tree referred to as \textit{label tree}. 
The label tree $T^\pi$ of SP digraph $G$ is defined by the the following vertex and arc set
 \begin{align*}
 V(T^\pi) &= \{s\} \cup  \{v\in V(G)\mid \text{$v$ is a last vertex}\}, \\
 A(T^\pi)&=\{(s,w) \mid  \text{$w\in V(T^\pi)$ and $\pi(s)= \pi(w)$}\} \cup \{(v,w) \mid  \text{$v,w\in V(T^\pi)$ and vertices $v,w$ are $(v,w)$-consecutive}\}. 
 \end{align*}
The label tree is a tree spanned by the unique source and all last vertices. 
The root is defined by the unique source $s$ and the leaves are defined by the sinks $t^\lambda_i$, $i\in [\gamma_t(\lambda)]$, $\lambda\in \Lambda$. 
An example is visualized in Figure~\ref{RFfig:labeltree}. 
\begin{figure}
	\hspace*{-0.25cm}
	\scalebox{1}{\input{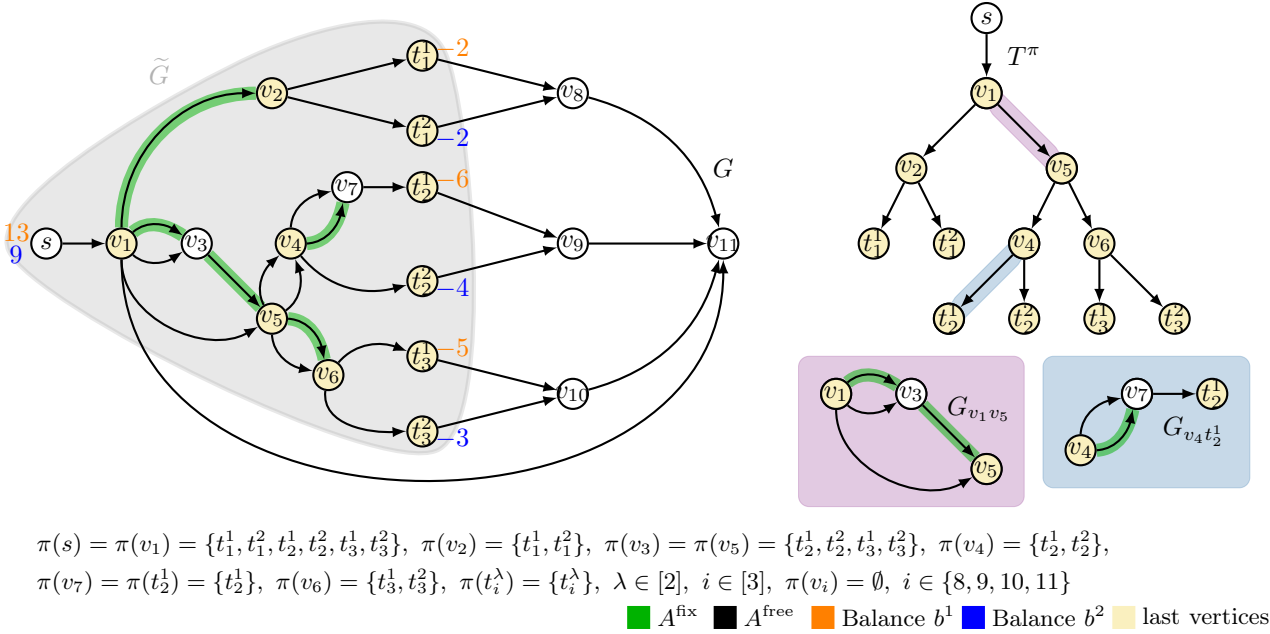}}
	\caption{Example of an SP digraph $G$, vertex labels $\pi(v)$, $v\in V(G)$, label tree $T^{\pi}$, subgraphs spanned by vertices incident to arcs $(v_1,v_5), (v_4,t^1_2)\in A(T^{\pi})$, and the composed union of the subgraphs $\widetilde{G} = \cup_{(v,w)\in A(T^\pi)} G_{vw}$}
	\label{RFfig:labeltree}
\end{figure}
We note that the label tree can be constructed by breadth-first search~\cite{korte2012combinatorial} at every vertex in polynomial time. 

Using the label tree, in the following lemma, we present two properties satisfied by every feasible robust flow.
\begin{lemma}\label{RFlemma:FeasibleFlowPropertiesUniqueSourceParallelSinks}
	Let $\mathcal{I}=(G,c,\boldsymbol{b})$ be a \ProblemName{} instance on SP digraph $G$ with unique source $s$ and parallel sinks $t^\lambda_i$, $i\in [\gamma_t(\lambda)]$, $\lambda\in \Lambda$. 	
	Further, let $T^\pi$ be the label tree of SP digraph $G$, $G_{vw}$ be the subgraph spanned by the vertices incident to arc $(v,w)\in A(T^\pi)$, and $\widetilde{G} := \bigcup_{(v,w)\in A(T^\pi)}G_{vw}$ be the composed union. %with $\widetilde{G}\subseteq G$ b
	For a feasible robust $\fett{b}$-flow $\fett{f}=(f^1,\ldots,f^{|\Lambda|})$ the following holds true
	\begin{itemize}
		\item[(i)] the flow is zero beyond digraph $\widetilde{G}$, i.e., $f^\lambda(a)=0$ for all arcs $a\in A(G)\setminus A(\widetilde{G})$ in every scenario $\lambda\in \Lambda$,
		\item[(ii)] the value that flows within subgraph $G_{vw}$, $(v,w)\in A(T^\pi)$ is determined by the sum of the demand of the sinks reachable from vertex $w$, i.e., 
		\begin{align*}
			\sum_{a=(v,z)\in A(G_{vw})} f^\lambda(a) \overset{(E1)}{=}    \sum_{a=(z,w)\in A(G_{vw})} f^\lambda(a)   	\overset{(E2)}{=}		 \sum_{t\in \pi(w)\cap \{t^\lambda_1,\ldots,t^\lambda_{\gamma_t(\lambda)}\} } |b^\lambda(t)|, \ \ \lambda\in \Lambda.
		\end{align*}
	\end{itemize}
\end{lemma}
For simplicity's sake, we denote the sum of the demand of sinks in scenario $\lambda \in \Lambda$ reachable from vertex $v\in V(G)$ by parameter $\beta^\lambda(v)$, i.e., $\beta^\lambda(v) = \sum_{t\in \pi(v)\cap \{t^\lambda_1,\ldots,t^\lambda_{\gamma_t(\lambda)}\} } |b^\lambda(t)|$.
\begin{proof}	
For the proof of the first statement, we note that $\widetilde{G}\subseteq G$ holds true. 
Furthermore, we note that $\pi(z)\neq \emptyset$ holds for all vertices $z\in V(\widetilde{G})$ and $\pi(z)= \emptyset$ holds for all vertices $z\in V(G) \setminus V(\widetilde{G})$.  
For an example, we refer to Figure~\ref{RFfig:labeltree}. 
Assume now the first statement is false.
There exists a feasible robust $\fett{b}$-flow $\fett{f}=(f^1,\ldots,f^{|\Lambda|})$ with $f^\lambda(a)>0$ for at least one arc $a=(x,y)\in A(G)\setminus A(\widetilde{G})$ in one scenario $\lambda\in \Lambda$. 
For the labels of the incident vertices of arc $a=(x,y)$ it holds either $\pi(x)=\pi (y)=\emptyset$ or $\pi (x)\neq \emptyset$ and $\pi (y)=\emptyset$ due to the inclusion chains formed by the labels.
Thus, in both cases none of the sinks $t^\lambda_i$, $i\in [\gamma_t(\lambda)]$, $\lambda\in\Lambda$ is reachable via arc $a$. 
We obtain $f^\lambda(a)=0$ for all scenarios $\lambda\in \Lambda$ because of the flow balance constraints, which contradicts the assumption. 

For the proof of the second statement, let $\lambda\in \Lambda$ be an arbitrary but fixed scenario. 
Equality~$(E1)$ holds true as the flow which enters SP digraph $G_{vw}$ via vertex $v$ needs to reach vertex $w$. 
There does not exist a sink in subgraph $G_{vw}$ (unless vertex $w$ is a sink itself) due to the following reason. 
As $w\in V(T^\pi)$ holds, we obtain $\pi(w)\neq \emptyset$ by definition. 
Thus, either vertex $w$ is a sink itself or a number of sinks is reachable from vertex $w$. 
In both cases, none of the vertices $z\in V(G_{vw})\setminus \{w\}$ is a sink as it would not be parallel to the sinks of set $\pi(w)$. 
Equality~$(E2)$ holds true as the value of the ingoing flow at vertex $w$ in SP digraph $G_{vw}$ must be equal to the sum of the demand of the sinks reachable from vertex $w$, i.e., $ \beta^\lambda(w)$. 
If the value was greater, the flow balance constraints would not be satisfied due to oversupply. 
If the value was lower, the demand of the sinks of set $\pi(w)$ could not be met as the only access is via vertex $w$. 
We note that no additional units are sent beyond subgraph $G_{vw}$ via vertex $w$ to sinks of set $\pi(w)$ as there does not exist an arc $(u,w)\in A(G)\setminus A(G_{vw})$ due to the following reason. 
We note that there does not exist a $(v,u)$-path with $u\in V(G)\setminus V(G_{vw})$.  
If there existed a $(v,u)$-path, it would be included with arc $(u,w)$ in subgraph $G_{vw}$.
As $G$ is an SP digraph, there does not exist an $(u,v)$-path either due to $(v,w)$-path and the assumption of the existence of arc $(u,w)$. 
Thus, vertices $u$ and $v$ are parallel as visualized in Figure~\ref{RFfig:proofparallelvertices}. 
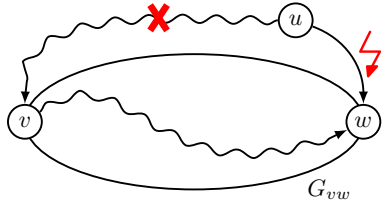
\begin{figure}
	\centering
	\begin{adjustbox}{max width=1\textwidth}
			\centering\scalebox{0.9}{\begin{tikzpicture}
	
	\Ellipse{2.5,0}{2.5}{1}{thick}
	\Knoten{(1)}{(0,0)}{$v$}{thick}
	\Knoten{(2)}{(5,0)}{$w$}{thick}
	\node at (4.5,-1){$G_{vw}$}; 
	
	\Knoten{(3)}{(4,1.5)}{$u$}{thick}
	
	\draw[thick, -latex] (3) to [in=90, out=-20] (2); 
	
	\draw[thick, -latex, decorate,decoration={snake,amplitude=.6mm,segment length=5mm,post length=2mm}] (1) .. controls (1.5,1) and (2,-1.5)  .. (2); 
	
	\draw[thick, -latex, decorate,decoration={snake,amplitude=.6mm,segment length=5mm,post length=1mm}, ]  (3)   to [in=90, out=180]  (1); 
	
	\Kreuz{0.1cm}{red}{(1.85,1.35)}{0.25}

	\node[red] at (5.1,1){\Huge \lightning}; 
\end{tikzpicture}}
	\end{adjustbox}
	\caption{If vertex $u\in V(G)\setminus V(G_{vw})$ incident to arc $(u,w)$ existed, it would be parallel to vertex $v$}\label{RFfig:proofparallelvertices}
\end{figure}
Let $T$ be an SP tree of SP digraph $G$. 
As $v$ is a last vertex, there must exist a $P$-vertex $x\in V(T)$ whose associated subgraph $G_x$ is the first subgraph that contains all sinks of set $\pi(v)$ and has vertex $v$ as origin. 
Let $G_1$ and $G_2$ denote the subgraphs by which the subgraph $G_x$ is composed in parallel. 
As $w$ is a last vertex with $\pi(w)\subsetneq \pi(v)$, it is contained in either subgraph $G_1$ or $G_2$ but it is not the target.
If vertex $w$ was the target, the sinks in set $\pi(v)\setminus \pi(w)\subseteq V(G_x)$ would not be parallel to the sinks in set $\pi(w)$. 
For this reason, the access to vertex $w$ is given by vertex $v$. 
As vertices $u$ and $v$ are parallel, there cannot exist an $(u,w)$-path in SP digraph $G$, and in particular, there cannot exist an arc $(u,w) \in A(G)\setminus A(G_{v,w})$. 
\end{proof}
As a result, we can focus on the spanned subgraphs induced by the label tree's arcs to compute a robust flow. 
Furthermore, using Lemma~\ref{RFlemma:FeasibleFlowPropertiesUniqueSourceParallelSinks}, we provide the analogous result to Lemma~\ref{ClaimPropertyCheapestFlowUniqueSourceAndSink} about the structure of an existing optimal robust flow.
\begin{lemma}\label{ClaimPropertyCheapestFlow}
	Let $\mathcal{I}=(G,c,\boldsymbol{b})$ be a \ProblemName{} instance on SP digraph $G$ with a unique source and parallel sinks. 	
	There exists an optimal robust $\boldsymbol{b}$-flow $\boldsymbol{\tilde{f}}=(\tilde{f}^1,\ldots, \tilde{f}^{|\Lambda|})$ such that for all feasible robust $\boldsymbol{b}$-flows $\fett{f}=(f^1, \ldots, f^{|\Lambda|})$ it holds
	$c(\tilde{f}^\lambda)\leq c(f^\lambda)$ for all $\lambda\in \Lambda$.
\end{lemma}
\begin{proof}
	Let $T^\pi$ be the label tree of SP digraph $G$ and let $G_{vw}$ be the subgraph spanned by the vertices incident to arc $(v,w)\in A(T^\pi)$. 
	By Lemma~\ref{RFlemma:FeasibleFlowPropertiesUniqueSourceParallelSinks}, the flow is zero beyond digraph $\widetilde{G}=\cup_{(v,w)\in A(T^\pi)}G_{vw}$ and the values which flow within subgraphs $G_{vw}$, $(v,w)\in A(T^\pi)$ are determined for every feasible robust flow. 
	Consequently, only the flow within a subgraph $G_{vw}$, $(v,w)\in A(T^\pi)$ may differ in feasible and in particular in optimal solutions. 
	Considering the subgraphs $G_{vw}$ for all $(v,w)\in A(T^\pi)$ separately, we split \ProblemName{} instance $\mathcal{I}$ into instances $\mathcal{I}_{vw}=(G_{vw},c,\fett{b_{vw}})$ where the arc cost $c$ on the subgraphs remains equal and the new balances $\fett{b_{vw}}:V(G_{vw})\rightarrow \mathbb{Z}$ are defined as follows
	\begin{align*}
		b^\lambda_{vw}(u) &= 
		\begin{cases}
			0 									& \text{ for all }u\in V(G_{vw})\setminus \{v,w\}, \\
			 \beta^\lambda(w) 			& \text{ for }u=v,\\
			- \beta^\lambda(w) 	 & \text{ for }u=w.
		\end{cases}
	\end{align*}	
	Instance $\mathcal{I}_{vw}$, $(v,w)\in A(T^\pi)$ is a \ProblemName{} instance based on an SP digraph with a unique source and a unique sink. 
	By Lemma~\ref{ClaimPropertyCheapestFlowUniqueSourceAndSink}, there exists an optimal robust $\boldsymbol{b_{vw}}$-flow $\boldsymbol{\tilde{f}_{vw}}=(\tilde{f}^1_{vw},\ldots, \tilde{f}^{|\Lambda|}_{vw})$ such that for all feasible robust $\boldsymbol{b_{vw}}$-flows $\fett{f_{vw}}=(f^1_{vw}, \ldots, f^{|\Lambda|}_{vw})$ it holds $c(\tilde{f}^\lambda_{vw})\leq c(f^\lambda_{vw})$ for all $\lambda\in \Lambda$.
	The composed flow $\boldsymbol{\tilde{f}} = \sum_{(v,w)\in A(T^\pi)}\boldsymbol{\tilde{f}_{vw}}$ is a feasible robust $\fett{b}= \sum_{(v,w)\in A(T^\pi)} \fett{b_{vw}}$-flow for \ProblemName{} instance $\mathcal{I}$ as the flow balance and consistent flow constraints are satisfied.
	Let $\fett{\hat{f}}$ be an arbitrary feasible robust $\fett{b}$-flow for \ProblemName{} instance $\mathcal{I}$. 
	Further, let $\fett{\hat{f}}_{|G_{vw}}$ be the flow which results from restricting flow $\fett{\hat{f}}$ to subgraph $G_{vw}$, $(v,w)\in A(T^\pi)$. 
	Due to Lemma~\ref{RFlemma:FeasibleFlowPropertiesUniqueSourceParallelSinks}, flow $\fett{\hat{f}}_{|G_{vw}}$ is a feasible robust $\fett{b_{vw}}$-flow for instance $\mathcal{I}_{vw}$, $(v,w)\in A(T^\pi)$. 
	Thus, it holds $c(\tilde{f}^\lambda_{vw})\leq c(\hat{f}^\lambda_{|G_{vw}}) $ for all $\lambda\in \Lambda$.
	We obtain
	\begin{align*}
		c(\tilde{f}^\lambda)  = \sum_{(v,w)\in A(T^\pi)} c(\tilde{f}^\lambda_{vw})  \leq  \sum_{(v,w)\in A(T^\pi)} c(\hat{f}^\lambda_{|G_{vw}}) = c(\hat{f}^\lambda)
	\end{align*}
	for all scenarios $\lambda\in \Lambda$, which implies $c(\fett{\tilde{f}})\leq c(\fett{\hat{f}})$.
	Consequently, robust $\fett{b}$-flow $\boldsymbol{\tilde{f}}$ is optimal such that for all feasible robust $\boldsymbol{b}$-flows $\fett{f}=(f^1, \ldots, f^{|\Lambda|})$ it holds $c(\tilde{f}^\lambda)\leq c(f^\lambda)$ for all $\lambda\in \Lambda$.	
\end{proof}
Using Lemmas~\ref{RFlemma:FeasibleFlowPropertiesUniqueSourceParallelSinks} and~\ref{ClaimPropertyCheapestFlow}, we specify the composition of an optimal robust flow. 
\begin{lemma}\label{RFlem:zusammengesetzteLoesungOptimal}
	Let $\mathcal{I}=(G,c,\boldsymbol{b})$ be a \ProblemName{} instance on SP digraph $G$ with a unique source and parallel sinks. 
	Further, let $T^\pi$ be the label tree of SP digraph $G$ and $G_{vw}$ be the subgraph spanned by the vertices incident to arc $(v,w)\in A(T^\pi)$.
	Let $\mathcal{I}_{vw}=(G_{vw},c,\boldsymbol{b_{vw}})$ be the \ProblemName{} instance restricted to subgraph $G_{vw}\subseteq G$ where the arc cost $c$ remains equal and the new balances are defined as follows
	\begin{align*}
		&b^\lambda_{vw}(u) = 
		\begin{cases}
			0  & \text{ for all }u\in V(G_{vw})\setminus \{v,w\}, \\
			\beta^\lambda(w)  & \text{ for }u=v,\\
			-\beta^\lambda(w)  & \text{ for }u=w.
		\end{cases}
	\end{align*}
%	A robust $\boldsymbol{b}$-flow $\boldsymbol{f}=(f^1,\ldots, f^{|\Lambda|})$ for instance $\mathcal{I}$ is optimal if and only if for every $(v,w)\in A(T^\pi)$ the robust flow $\fett{f}_{|G_{vw}}=(f^1_{|G_{vw}},\ldots, f^{|\Lambda|}_{|G_{vw}})$, which is obtained by restricting flow $\fett{f}$ to subgraph $G_{vw}$, is an optimal robust $\fett{b_{vw}}$-flow for instance $\mathcal{I}_{vw}$.
For $(v,w)\in A(T^\pi)$, let $\fett{f_{vw}}=(f^1_{{vw}},\ldots, f^{|\Lambda|}_{{vw}})$ be an optimal robust $\fett{b_{vw}}$-flow for instance $\mathcal{I}_{vw}$ such that for all feasible robust $\fett{b_{vw}}$-flows $\fett{\hat{f}}=(\hat{f}^1, \ldots, \hat{f}^{|\Lambda|})$ it holds $c(f^\lambda_{vw})\leq c(\hat{f}^\lambda)$ for all $\lambda\in \Lambda$.
The composed flow $\fett{f} = \sum_{(v,w)\in A(T^\pi)} \fett{f_{vw}}$ is an optimal robust $\fett{b} = \sum_{(v,w)\in A(T^\pi)} \fett{b_{vw}}$-flow for instance $\mathcal{I}$. 
\end{lemma}
\begin{proof}
	The composed flow $\fett{f}$ is a feasible robust $\fett{b}$-flow for instance $\mathcal{I}$ 
	as the flow balance and consistent flow constraints are satisfied. 
	Assume $\fett{f}$ is not optimal for instance $\mathcal{I}$. 
	There exists a robust $\fett{b}$-flow $\fett{\tilde{f}}$ with less cost, i.e., $c(\fett{\tilde{f}}) = \max_{\lambda\in\Lambda} c(\tilde{f}^\lambda) = c(\tilde{f}^{\lambda_1})< c(f^{\lambda_2}) = \max_{\lambda\in \Lambda} c(f^\lambda)= c(\fett{f})$ with $\lambda_1,\lambda_2\in \Lambda$. 
	Let $\fett{\tilde{f}}_{|G_{vw}}$ denote the flow obtained by restricting flow $\fett{\tilde{f}}$ to subgraph $G_{vw}$, $(v,w)\in A(T^\pi)$. 
	Using Lemma~\ref{RFlemma:FeasibleFlowPropertiesUniqueSourceParallelSinks}, we obtain for the cost
	\begin{align*}
		c(\tilde{f}^{\lambda_1}) &= \sum_{a\in A(G)} c(a)\cdot \tilde{f}^{\lambda_1}(a)  = \sum_{(v,w)\in A(T^\pi)}	\sum_{a\in A(G_{vw})} c(a)\cdot \tilde{f}^{\lambda_1}(a)  = \sum_{(v,w)\in A(T^\pi)}	c(\tilde{f}^{\lambda_1}_{|G_{vw}}), \\
		c(f^{\lambda_2}) &= \sum_{a\in A(G)} c(a)\cdot f^{\lambda_2}(a)  = \sum_{(v,w)\in A(T^\pi)}	\sum_{a\in A(G_{vw})} c(a)\cdot f^{\lambda_2}(a)  = \sum_{(v,w)\in A(T^\pi)}	c(f^{\lambda_2}_{vw}).
	\end{align*}
	If $\lambda_1=\lambda_2$ holds, we obtain 
	\begin{align*}
		\sum_{(v,w)\in A(T^\pi)}	c(\tilde{f}^{\lambda_1}_{|G_{vw}}) = c(\tilde{f}^{\lambda_1}) < c(f^{\lambda_2}) = c(f^{\lambda_1}) = \sum_{(v,w)\in A(T^\pi)}	c(f^{\lambda_1}_{vw}), 
	\end{align*}
	which is a contradiction as $c(f^{\lambda_1}_{vw})\leq c(\tilde{f}^{\lambda_1}_{|G_{vw}})$ holds for all $(v,w)\in A(T^\pi)$. 
	If $\lambda_1\neq \lambda_2$ holds, we obtain 
	\begin{align*}
		\sum_{(v,w)\in A(T^\pi)}	c(\tilde{f}^{\lambda_2}_{|G_{vw}}) = c(\tilde{f}^{\lambda_2})	 \leq  \max_{\lambda\in \Lambda} c(\tilde{f}^\lambda) = c(\tilde{f}^{\lambda_1}) < c(f^{\lambda_2}) = \sum_{(v,w)\in A(T^\pi)}	c(f^{\lambda_2}_{vw}), 
	\end{align*}
	which is a contradiction as $c(f^{\lambda_2}_{vw})\leq c(\tilde{f}^{\lambda_2}_{|G_{vw}})$ holds for all $(v,w)\in A(T^\pi)$. 
\end{proof}
Based on the presented structural results, we provide Algorithm~\ref{RFAlg:OptFlowMultipleSinksUncap} to compute an optimal robust flow. 
\begin{algorithm}[H]
	\caption{} \label{RFAlg:OptFlowMultipleSinksUncap}
	\begin{lyxlist}{Method:}
		\item [{Input:}] \ProblemName{} instance $(G,c,\boldsymbol{b})$ with a unique source and parallel sinks where $G$ is an SP digraph 
		\item [{Output:}] Robust minimum cost $\fett{b}$-flow $\fett{f}$
		\item [{Method:}]~ 	\end{lyxlist}
	\begin{algorithmic}[1]
		\State Construct label tree $T^\pi$
		\For{all arcs $(v,w)\in A(T^\pi)$}
		\State \begin{varwidth}[t]{\linewidth} Consider the unique source, unique sink instance $\mathcal{I}_{vw}=(G_{vw},c,\fett{b_{vw}})$ on subgraph $G_{vw}$ spanned by vertices $v,w$, where the arc cost $c$ remains equal and the new balances are defined for $\lambda\in \Lambda$ as follows \end{varwidth}
		\begin{align*}
			&b^\lambda_{vw}(u) = 
			\begin{cases}
				0 & \text{ for all }u\in V(G_{vw})\setminus \{v,w\}, \\
				\beta^\lambda(w)  & \text{ for }u=v,\\
				-\beta^\lambda(w)  & \text{ for }u=w
			\end{cases}
		\end{align*}
			\State \begin{varwidth}[t]{\linewidth}Compute an optimal robust $\fett{b_{vw}}$-flow $\fett{f_{vw}}$ for instance $\mathcal{I}_{vw}$ such that for all feasible robust $\fett{b_{vw}}$-flows $\fett{\tilde{f}}$ it \linebreak holds $c(f^\lambda_{vw})\leq c(\tilde{f}^\lambda)$ for all $\lambda\in \Lambda$   
			\end{varwidth}
		\EndFor
		\State Set $\fett{f}:=\sum_{(v,w)\in A(T^\pi)}\fett{f_{vw}}$
		\State\Return Robust $\fett{b}$-flow $\fett{f}=(f^1,\ldots,f^{|\Lambda|})$
	\end{algorithmic}
\end{algorithm}
We obtain the polynomial-time solvability of the \ProblemName{} problem for networks based on SP digraphs with a unique source and parallel sinks as shown in the following theorem. 
\begin{theorem}\label{RFlem:}
	Let $(G,c,\boldsymbol{b})$ be a \ProblemName{} instance on SP digraph $G$ with a unique source and parallel sinks. 
	Algorithm~\ref{RFAlg:OptFlowMultipleSinksUncap} computes an optimal robust $\fett{b}$-flow in $\mathcal{O}(|V(G)|^2  +|V(G)| \cdot |A(G)|  )$ time.
\end{theorem}
\begin{proof}
The correctness of the algorithm results from Lemma~\ref{RFlem:zusammengesetzteLoesungOptimal}. 
Considering the runtime, we obtain the following. 
The label tree $T^\pi$ can be constructed by breadth-first search~\cite{korte2012combinatorial} at every vertex in $\mathcal{O}(|V(G)|^2  + |V(G)| \cdot |A(G)| )$ time.
For each instance $\mathcal{I}_{vw}$, $(v,w)\in A(T^\pi)$, an optimal robust $\fett{b_{vw}}$-flow $\fett{f_{vw}}$ with the desired property is constructed in $\mathcal{O}(|V(G)|+|A(G)|)$ time by means of the methods of Section~\ref{Subsubsec:UniqueSourceUnqiueSink}. 
Overall, this can be done in $\mathcal{O}(|V(G)|^2+|V(G)| \cdot |A(G)|)$ time as $A(T^\pi)\leq V(G)$ holds true.
In total, the algorithm runs in $\mathcal{O}(|V(G)|^2  + |V(G)| \cdot |A(G)|  )$ time. 
\end{proof}
Finally, we note that the results of this section can be applied to the \ProblemName{} problem for networks based on SP digraphs with parallel sources and parallel sinks if there exists a path between each source and each sink. 
For this special case, we need to determine two vertices, the first vertex $v_1$ reachable from all sources and the last vertex $v_2$ from which all sinks are reachable. 
Using vertices $v_1$, $v_2$, we divide a \ProblemName{} instance $\mathcal{I}$ into the three instances $\mathcal{I}_1, \mathcal{I}_2,\mathcal{I}_3$---a parallel sources, unique sink, a unique source, parallel sinks, and a unique source, unique sink instance---as visualized in Figure~\ref{RFfig:ParallelSourcesParallelSinks}.
\begin{figure}
	\centering
	\scalebox{1}{\begin{tikzpicture}
	
	\Shape{(-0.3,0)   (2.1,1.3) (2.1,-1.3)   }{opacity=0.3, gray!60}{none,fill=gray!60};
	\node[gray!60] at (2.25,1.7){$\widetilde{G}_{v_2q_G}$};
	\Shape{(-0.4,0)   (2.1,1.4) (4.3,0) (2.1,-1.4)   }{ SteelBlue!90}{none};
	\node[SteelBlue!90] at (3.25,1.25){$G_{v_2q_G}$};

	\Shape{(-5.75,0.1)   (-8.1,1.3) (-8.25,-1.2) (-7.1,-1.35)   }{opacity=0.3, gray!60}{none,fill=gray!60};
	\node[gray!60] at (-8,-1.8){$\widetilde{G}_{o_Gv_1}$};
	\Shape{(-5.75,0.1)  (-6.75,1.4) (-8,1.75) (-9.25,1.5) (-10.25,0.1) (-9.1,-1.35)  (-6.9,-1.4)   }{SteelBlue!90}{none};
	\node[SteelBlue!90] at (-7,-1.75){$G_{o_Gv_1}$};
	
%	\node[gray!60] at (0.5,2.3){$\widetilde{G}$}; 
	
	\Ellipse{-3,0}{3}{1.75}{very thick, fill=lightgray!50, draw=gray!60, opacity=0.3, fill=gray!60}
	\node[gray!60] at (-0.45,1.3){$G_{v_1v_2}$};
	%Senken		
	\Knoten{(1)}{(0,0)}{$v_2$}{thick, minimum size=0.4cm, fill=blond}
	\Knoten{(2)}{(2,0)}{$t^2_1$}{thick, minimum size=0.4cm, fill=blond}
	\Knoten{(3)}{(2,1)}{$t^1_1$}{thick, minimum size=0.4cm, fill=blond}
	\Knoten{(4)}{(2,-1)}{$t_2^2$}{thick, minimum size=0.4cm, fill=blond}
	\Knoten{(5)}{(1,.5)}{}{thick, minimum size=0.4cm, fill=blond}
	\Knoten{(6)}{(1,-.5)}{}{thick, minimum size=0.4cm}
	\Knoten{(7)}{(4,0)}{$q_G$}{thick, minimum size=0.4cm}
	\Knoten{(8)}{(3,0.5)}{}{thick, minimum size=0.4cm}

	%Quellen
	\Knoten{(10)}{(-10,0)}{$o_G$}{thick, minimum size=0.4cm}
	\Knoten{(11)}{(-9,0.5)}{}{thick, minimum size=0.4cm}
	\Knoten{(12)}{(-8,1)}{$s^1_1$}{thick, minimum size=0.4cm, fill=Brombeer!60}
	\Knoten{(13)}{(-8,0)}{$s^2_1$}{thick, minimum size=0.4cm, fill=Brombeer!60}
	\Knoten{(14)}{(-9,-1)}{}{thick, minimum size=0.4cm}
	\Knoten{(15)}{(-8,-1)}{$s^1_2$}{thick, minimum size=0.4cm, fill=Brombeer!60}
	
	\Knoten{(16)}{(-7,0.5)}{}{thick, minimum size=0.4cm, fill=Brombeer!60}
	\Knoten{(17)}{(-6,0)}{$v_1$}{thick, minimum size=0.4cm, fill=Brombeer!60}
	\Knoten{(18)}{(-7,-1)}{}{thick, minimum size=0.4cm}

	\Knoten{(20)}{(-4,-2)}{}{thick, minimum size=0.4cm}
	\Knoten{(21)}{(-2,-2.25)}{}{thick, minimum size=0.4cm}
	\Knoten{(22)}{(-1,-2)}{}{thick, minimum size=0.4cm}
	\Knoten{(23)}{(-1,-2.75)}{}{thick, minimum size=0.4cm}
	\Knoten{(24)}{(0,-2.25)}{}{thick, minimum size=0.4cm}
	\Knoten{(25)}{(2,-2)}{}{thick, minimum size=0.4cm}
	
	\draw[thick, -latex] (17) to [bend right, out=-30, in=-150] (20);
	\draw[thick, -latex] (20) to [bend right, out=-10, in=-170] (21);
	\draw[thick, -latex] (24) to [bend right, out=-10, in=-170] (25);
	\draw[thick, -latex] (21) to [bend left, out=45] (22);  
	\draw[unKante,line width=5,  opacity=0.5,\gruen] (21) to [bend right, out=-45] (22);  	
	\draw[thick, -latex] (21) to [bend right, out=-45] (22);  
	\draw[thick, -latex] (21) to [bend right, out=-45]  (23);  
	\draw[thick, -latex] (22) to [bend left, out=45] (24);  
	\draw[unKante,line width=5,  opacity=0.5,\gruen] (22) to [bend right, out=-45] (24);  	
	\draw[thick, -latex] (22) to [bend right, out=-45] (24);  
	\draw[thick, -latex] (23) to [bend right, out=-45]  (24);  
	\draw[thick, -latex] (25) to [bend right, out=-10, in=-170] (7);
	
	%Quellen
	\draw[thick, -latex] (10) -- (11);
	\draw[thick, -latex] (11) -- (12);
	\draw[thick, -latex] (11) -- (13);
	\draw[thick, -latex] (10) -- (14);	
	\draw[unKante,line width=5,  opacity=0.5,\gruen] (14) to [bend left, out=45] (15);
	\draw[thick, -latex] (14) to [bend left, out=45] (15);
	\draw[thick, -latex] (14) to [bend right, out=-45] (15);

	\draw[unKante,line width=5,  opacity=0.5,\gruen] (12) to [bend right, out=-45] (16);  	
	\draw[thick, -latex] (12) to [bend right, out=-45] (16);  
	\draw[thick, -latex] (12) to [bend left, out=45] (16);  
	\draw[thick, -latex] (13) -- (16);  
	\draw[unKante,line width=5,  opacity=0.5,\gruen] (16) to [bend right, out=-45] (17);
	\draw[thick, -latex] (16) to [bend left, out=45] (17);
	\draw[thick, -latex] (16) to [bend right, out=-45] (17);
	
	\draw[unKante,line width=5,  opacity=0.5,\gruen] (15) to [bend right, out=-45] (18);
	\draw[thick, -latex] (15) to [bend left, out=45] (18);
	\draw[thick, -latex] (15) to [bend right, out=-45] (18);
	
	\draw[thick, -latex] (18) -- (17);

	%Senken
	\draw[thick, -latex] (1) to [bend left, out=45] (5);  
	\draw[unKante,line width=5,  opacity=0.5,\gruen] (1) to [bend right, out=-45] (5);  	
	\draw[thick, -latex] (1) to [bend right, out=-45] (5);  
	\draw[thick, -latex] (5) to [bend left, out=45] (3);  
	\draw[unKante,line width=5,  opacity=0.5,\gruen] (5) to [bend right, out=-45] (3);  	
	\draw[thick, -latex] (5) to [bend right, out=-45] (3);  	
	
	\draw[thick, -latex] (5) -- (2);
	\draw[thick, -latex] (2) -- (8);  
	\draw[thick, -latex] (8) -- (7);  
	\draw[thick, -latex] (4) -- (7);  
	\draw[thick, -latex] (3) -- (8);  
	  
	\draw[thick, -latex] (1) -- (6);  
	\draw[thick, -latex] (6) to [bend left, out=45] (4);  
	\draw[unKante,line width=5,  opacity=0.5,\gruen] (6) to [bend right, out=-45] (4);  	
	\draw[thick, -latex] (6) to [bend right, out=-45] (4);

	\draw[thick, -latex] (17) to [bend right, out=-80, in=-100, looseness=1.15] (7);
	\draw[thick, -latex] (10) to [bend right, out=-80, in=-90] (7);
	\draw[thick, -latex] (10) to [bend left, out=80, in=110, looseness=1.25] (17);

%	\draw[thick, -latex, decorate,decoration={snake,amplitude=.6mm,segment length=5mm,post length=1mm}, ]  (17)   to [in=200, out=-30]  (1); 
%	\draw[unKante,line width=5,  opacity=0.5,\gruen, decorate,decoration={snake,amplitude=.6mm,segment length=5mm,post length=3mm}, ]  (17)   to [in=180, out=40]  (1); 
%	\draw[thick, -latex, decorate,decoration={snake,amplitude=.6mm,segment length=5mm,post length=3mm}, ]  (17)   to [in=180, out=40]  (1); 
%	\node at (-3.5,-1.15){$p^{\freehelp}\subseteq G_{v_1v_2}-\fix$}; 
%	\node at (-1.,0.6){$p^{\fixhelp}\subseteq G_{v_1v_2}$}; 
%	
	\node at (2.85,-3){$G$}; 
	
	\Knoten{(27)}{(-4,0)}{}{thick, minimum size=0.4cm}
	\Knoten{(28)}{(-3,1)}{}{thick, minimum size=0.4cm}
	\Knoten{(29)}{(-2,0)}{}{thick, minimum size=0.4cm}
	\Knoten{(30)}{(-2.5,-1)}{}{thick, minimum size=0.4cm}
	\Knoten{(31)}{(-3.5,-1)}{}{thick, minimum size=0.4cm}
	
	\draw[thick, -latex] (17) to [out=45,in=135] (27);
	\draw[unKante,line width=5,  opacity=0.5,\gruen] (27) to[out=90, in=180] (28);
	\draw[unKante,line width=5,  opacity=0.5,\gruen] (28) -- (29);
	\draw[thick, -latex] (27) to[out=90, in=180] (28);
	\draw[thick, -latex] (28) -- (29);
	\draw[thick, -latex] (27) to[out=-5, in=270] (28);
	\draw[thick, -latex] (28) -- (29);
	\draw[thick, -latex] (29) to [out=45,in=135]  (1);
	
	\draw[thick, -latex] (17) to [out=-45, in=180] (31);
	\draw[thick, -latex] (31) -- (30);
	\draw[thick, -latex] (30) to [out=0, in=225] (1);

	\begin{scope}[shift={(-13,-0.5)}]
	\begin{scope}[shift={(0.1,0)}]
		\begin{scope}[shift={(-.2,0)}]
			\begin{scope}[shift={(10.85,-4.6)}]
				\Job{(0,0)}{(0.25,0.25)}{\gruen, \gruen}{}{}
				\node[right] at (0.25,0.15){\small $A^{\text{fix}}$}; 
			\end{scope}
			\begin{scope}[shift={(12,-4.6)}]
				\Job{(0,0)}{(0.25,0.25)}{black, black}{}{}
				\node[right] at (0.25,0.15){\small $A^{\text{free}}$}; 
			\end{scope}
		\end{scope}
	\end{scope}
	\begin{scope}[shift={(13.2,-4.6)}]
		\Job{(0,0)}{(0.25,0.25)}{Brombeer, Brombeer}{}{}
		\node[right] at (0.25,0.15){\small first vertices}; 
	\end{scope}
	\begin{scope}[shift={(15.4,-4.6)}]
		\Job{(0,0)}{(0.25,0.25)}{blond, blond}{}{}
		\node[right] at (0.25,0.15){\small last vertices}; 
	\end{scope}
\end{scope}

\end{tikzpicture}}
	\caption{
		A \ProblemName{} instance can be divided into three instances $\mathcal{I}_1, \mathcal{I}_2,\mathcal{I}_3$ based on subgraphs $G_{o_Gv_1}$, $G_{v_2q_G}$, $G_{v_1v_2}$, respectively. 
		A robust $\fett{b}$-flow is zero beyond subgraph  $\widetilde{G}_{o_Gv_1} \cup  G_{v_1v_2} \cup \widetilde{G}_{v_2q_G}$. First vertices are reversely defined to last vertices 
	}
	\label{RFfig:ParallelSourcesParallelSinks}
\end{figure}
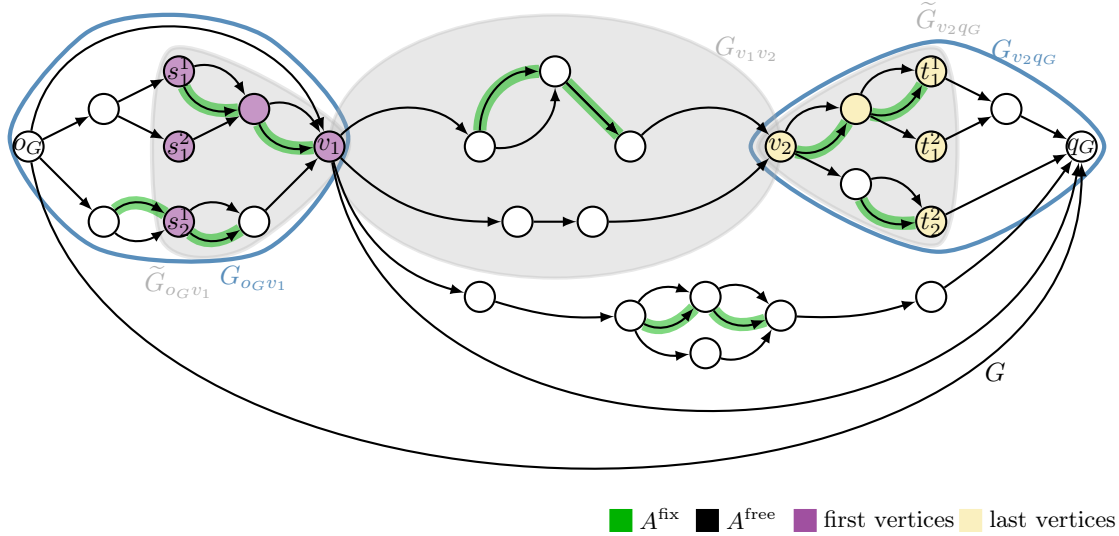
Instances $\mathcal{I}_1=(G_{o_{G}v_1 }, c, \fett{b_{o_{G}v_1 }})$, $\mathcal{I}_2=(G_{v_2q_G}, c, \fett{b_{v_2q_G}})$, $\mathcal{I}_3=(G_{v_1 v_2}, c, \fett{b_{v_1 v_2}})$ are based on the spanned subgraphs 
$G_{o_{G}v_1 }$, $G_{v_2q_G}$, $G_{v_1 v_2}\subseteq G$, respectively. 
The arc cost $c$ remains equal and the balances are defined for $\lambda\in\Lambda$ as follows 
\begin{align*}
	&b^\lambda_{ o_{G}v_1}(u) = 
	\begin{cases}
		b^\lambda(u)  & \text{ for all }u\in V(G_{o_{G}v_1 })\setminus \{v_1\}, \\
		-\beta^\lambda(v_1)  & \text{ for }u=v_1,
	\end{cases}
&&b^\lambda_{v_2q_G}(u) = 
\begin{cases}
	b^\lambda(u)  & \text{ for all }u\in V(G_{v_2q_G})\setminus \{v_2\}, \\
	\beta^\lambda(v_2)  & \text{ for }u=v_2, 
\end{cases}
\end{align*}
\begin{align*}
	&b^\lambda_{v_1 v_2}(u) = 
	\begin{cases}
		b^\lambda(u)  & \text{ for all }u\in V(G_{v_1 v_2})\setminus \{v_1 v_2\}, \\
		\beta^\lambda(v_1)  & \text{ for }u=v_1, \\
		-\beta^\lambda(v_2)  & \text{ for }u=v_2. 
	\end{cases}
\end{align*}
For the \ProblemName{} instances $\mathcal{I}_1, \mathcal{I}_2 $, $\mathcal{I}_3$, we determine optimal robust flows $\fett{f_1}, \fett{f_2}, \fett{f_3}$, respectively, by means of Algorithms~\ref{RFAlg:2ShortestPathForUniqueSourceUniqueSink} and~\ref{RFAlg:OptFlowMultipleSinksUncap}. 
Flows $\fett{f_3}$ and $\fett{f_1}, \fett{f_2}$ satisfy the properties of Lemmas~\ref{ClaimPropertyCheapestFlowUniqueSourceAndSink} and~\ref{ClaimPropertyCheapestFlow}, respectively. 
Therefore, analogous to the proof of Lemma~\ref{RFlem:zusammengesetzteLoesungOptimal}, we can prove that the composed flow $\fett{f}=\fett{f_1}+\fett{f_2}+\fett{f_3}$ is an optimal robust $\fett{b}$-flow for \ProblemName{} instance $\mathcal{I}$.
%%%%%%NEW
\subsection{Complexity for pearl digraphs}
\label{Subsubsec:PearlDigraphs}
In this section, we provide a polynomial-time algorithm for the \ProblemName{} problem on pearl digraphs---SP digraphs with a specific structure.
We note that this special case is $\mathcal{NP}$-complete if arc capacities are given as shown in our previous work~\cite{buesing2020robust}.
Using the representation of SP trees, we define pearl digraphs based on the definition of Ohst~\cite{ohst2016construction}.
\begin{definition}[Pearl digraph]\label{Def:PearlGraphs}
	A \textit{pearl digraph} is an SP digraph where the corresponding SP tree does not have $P$-vertices whose children are $S$-vertices.
\end{definition}
In other words, a pearl digraph is a path consisting of multi-arcs.
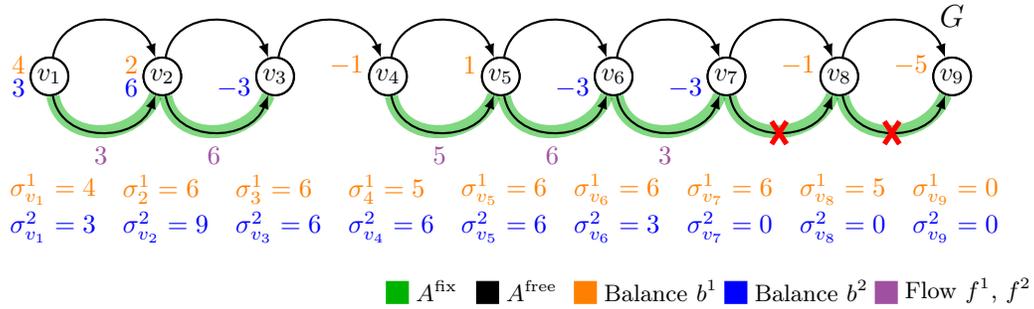
\begin{figure}
	\centering
	\begin{tikzpicture}
\Knoten{(0)}{(-1.5,0)}{}{thick, minimum size=0.5cm}
\Knoten{(1)}{(0,0)}{}{thick, minimum size=0.5cm}
\Knoten{(2)}{(1.5,0)}{}{thick, minimum size=0.5cm}
\Knoten{(3)}{(3,0)}{}{thick, minimum size=0.5cm}
\Knoten{(4)}{(4.5,0)}{}{thick, minimum size=0.5cm}
\Knoten{(5)}{(6,0)}{}{thick, minimum size=0.5cm}
\Knoten{(6)}{(7.5,0)}{}{thick, minimum size=0.5cm}
\Knoten{(7)}{(9,0)}{}{thick, minimum size=0.5cm}
\Knoten{(8)}{(10.5,0)}{}{thick, minimum size=0.5cm}

%\draw[draw=Brombeer, rounded corners, opacity=0.3, fill=Brombeer] ($(2)+(-0,0.9)$) -- node[xshift=2cm,yshift=0.25cm, Brombeer, opacity=1]{\large$\widetilde{G}$}($(6)+(0,0.9)$) -- ($(6)+(0,-0.9)$) -- ($(2)+(-0,-0.9)$) --cycle;

\node at (10.5,0.8){\large$G$};

%\draw[thick,->, -latex, black,] (s) to [out=80, in=110, looseness=1.25] node[midway, below]{}(0);
%\draw[unKante,line width=5,  opacity=0.5,\gruen] (s) to [out=-80, in=-110, looseness=1.25] node[midway, above]{}(0);
%\draw[thick,->, -latex, black,] (s) to [out=-80, in=-110, looseness=1.25] node[midway, above]{}(0);
%\Kreuz{2}{red}{(-2.4,-0.9)}{0.2}

\draw[thick,->, -latex, black,] (0) to [out=80, in=110, looseness=1.25] node[midway, below]{}(1);
\draw[unKante,line width=5,  opacity=0.5,\gruen] (0) to [out=-80, in=-110, looseness=1.25] node[midway, above]{}(1);
\draw[thick,->, -latex, black,] (0) to [out=-80, in=-110, looseness=1.25] node[midway, below=2pt, Brombeer]{$3$}(1);

\draw[unKante,line width=5,  opacity=0.5,\gruen] (1) to [out=-80, in=-110, looseness=1.25] node[midway, below]{}(2);
\draw[thick,->, -latex, black,] (1) to [out=80, in=110, looseness=1.25] node[midway, below]{}(2);
\draw[thick,->, -latex, black,] (1) to [out=-80, in=-110, looseness=1.25] node[midway, below=2pt, Brombeer]{$6$}(2);

\draw[thick,->, -latex, black,] (2) to [out=80, in=110, looseness=1.25] node[midway, below]{}(3);
%\draw[thick,->, -latex, black,] (2) to [out=-80, in=-110, looseness=1.25] node[midway, above]{}(3);

\draw[thick,->, -latex, black,] (3) to [out=80, in=110, looseness=1.25] node[midway, below]{}(4);
\draw[unKante,line width=5,  opacity=0.5,\gruen] (3) to [out=-80, in=-110, looseness=1.25] node[midway, above]{}(4);
\draw[thick,->, -latex, black,] (3) to [out=-80, in=-110, looseness=1.25] node[midway, below=2pt, Brombeer]{$5$}(4);

\draw[thick,->, -latex, black,] (4) to [out=80, in=110, looseness=1.25] node[midway, below]{}(5);
\draw[unKante,line width=5,  opacity=0.5,\gruen] (4) to [out=-80, in=-110, looseness=1.25] node[midway, below]{}(5);
\draw[thick,->, -latex, black,] (4) to [out=-80, in=-110, looseness=1.25] node[midway, below=2pt, Brombeer]{$6$}(5);	

\draw[thick,->, -latex, black,] (5) to [out=80, in=110, looseness=1.25] node[midway, below]{}(6);
\draw[unKante,line width=5,  opacity=0.5,\gruen] (5) to [out=-80, in=-110, looseness=1.25] node[midway, below]{}(6);
\draw[thick,->, -latex, black,] (5) to [out=-80, in=-110, looseness=1.25] node[midway, below=2pt, Brombeer]{$3$}(6);	

\draw[unKante,line width=5,  opacity=0.5,\gruen] (6) to [out=-80, in=-110, looseness=1.25] node[midway, below]{}(7);
\draw[thick,->, -latex, black,] (6) to [out=80, in=110, looseness=1.25] node[midway, below]{}(7);
\draw[thick,->, -latex, black,] (6) to [out=-80, in=-110, looseness=1.25] node[midway, below, Brombeer]{}(7);	
\Kreuz{2}{red}{(8.1,-0.9)}{0.2}

\draw[unKante,line width=5,  opacity=0.5,\gruen] (7) to [out=-80, in=-110, looseness=1.25] node[midway, below]{}(8);
\draw[thick,->, -latex, black,] (7) to [out=80, in=110, looseness=1.25] node[midway, below]{}(8);
\draw[thick,->, -latex, black,] (7) to [out=-80, in=-110, looseness=1.25] node[midway, below, Brombeer]{}(8);
\Kreuz{2}{red}{(9.6,-0.9)}{0.2}

%\Knoten{(s)}{(-3,0)}{}{thick, minimum size=0.5cm}
\Knoten{(0)}{(-1.5,0)}{$v_1$}{thick, minimum size=0.5cm}
\Knoten{(1)}{(0,0)}{$v_2$}{thick, minimum size=0.5cm}
\Knoten{(2)}{(1.5,0)}{$v_3$}{thick, minimum size=0.5cm}
\Knoten{(3)}{(3,0)}{$v_4$}{thick, minimum size=0.5cm}
\Knoten{(4)}{(4.5,0)}{$v_5$}{thick, minimum size=0.5cm}
\Knoten{(5)}{(6,0)}{$v_6$}{thick, minimum size=0.5cm}
\Knoten{(6)}{(7.5,0)}{$v_7$}{thick, minimum size=0.5cm}
\Knoten{(7)}{(9,0)}{$v_8$}{thick, minimum size=0.5cm}
\Knoten{(8)}{(10.5,0)}{$v_9$}{thick, minimum size=0.5cm}

\begin{scope}[shift={(0,0.05)}]
	%last source, first sink 
%	\node[right] at (0,-1.4){$\tilde{s}=s^2_2$}; 
%	\node[right] at (0,-1.9){$\tilde{t}=t^1_1$}; 
%	
%	
%	%Demand 
%	\node[right] at (1.55,-1.4){$d^{\min}=4$}; 
%	\node[right] at (3.3,-1.4){$d^1=5$}; 
%	\node[right] at (3.3,-1.9){$d^2=4$}; 
%	
	%Balancen
	%2. Szenario
	\node[blue, left] at ($(0)+(-0.2,-0.15)$){$3$}; 
	\node[blue, left] at ($(1)+(-0.2,-0.15)$){$6$}; 
	\node[blue, left] at ($(2)+(-0.2,-0.15)$){$-3$}; 
	\node[blue, left] at ($(5)+(-0.2,-0.15)$){$-3$}; 
	\node[blue, left] at ($(6)+(-0.2,-0.15)$){$-3$};

	%1. Szenario
	\node[orange, left] at ($(0)+(-0.2,0.15)$){$4$}; 
	\node[orange, left] at ($(1)+(-0.2,0.15)$){$2$}; 
	\node[orange, left] at ($(3)+(-0.2,0.15)$){$-1$}; 
	\node[orange, left] at ($(4)+(-0.2,0.15)$){$1$}; 
	\node[orange, left] at ($(7)+(-0.2,0.15)$){$-1$}; 
	\node[orange, left] at ($(8)+(-0.2,0.15)$){$-5$}; 
	
%	% Gamma
%	\node[right] at (4.8,-1.4){$\gamma_s(1)=2$}; 
%	\node[right] at (4.8,-1.9){$\gamma_t(1)=2$}; 
%	\node[right] at (6.55,-1.4){$\gamma_s(2)=2$}; 
%	\node[right] at (6.55,-1.9){$\gamma_t(2)=1$}; 

	%demands 
	\node[right, orange] at ($(0)+(-0.65,-1.5)$){$\sigma^1_{v_1}=4$}; 
	\node[right, orange] at ($(1)+(-0.65,-1.5)$){$\sigma^1_{2}=6$}; 
	\node[right, orange] at ($(2)+(-0.65,-1.5)$){$\sigma^1_{3}=6$}; 
	\node[right, orange] at ($(3)+(-0.65,-1.5)$){$\sigma^1_{4}=5$}; 
	\node[right, orange] at ($(4)+(-0.65,-1.5)$){$\sigma^1_{v_5}=6$}; 
	\node[right, orange] at ($(5)+(-0.65,-1.5)$){$\sigma^1_{v_6}=6$}; 
	\node[right, orange] at ($(6)+(-0.65,-1.5)$){$\sigma^1_{v_7}=6$}; 
	\node[right, orange] at ($(7)+(-0.65,-1.5)$){$\sigma^1_{v_8}=5$}; 
	\node[right, orange] at ($(8)+(-0.65,-1.5)$){$\sigma^1_{v_9}=0$}; 
	
	\node[right, blue] at ($(0)+(-0.65,-2)$){$\sigma^2_{v_1}=3$}; 
	\node[right, blue] at ($(1)+(-0.65,-2)$){$\sigma^2_{v_2}=9$}; 
	\node[right, blue] at ($(2)+(-0.65,-2)$){$\sigma^2_{v_3}=6$}; 
	\node[right, blue] at ($(3)+(-0.65,-2)$){$\sigma^2_{v_4}=6$}; 
	\node[right, blue] at ($(4)+(-0.65,-2)$){$\sigma^2_{v_5}=6$}; 
	\node[right, blue] at ($(5)+(-0.65,-2)$){$\sigma^2_{v_6}=3$}; 
	\node[right, blue] at ($(6)+(-0.65,-2)$){$\sigma^2_{v_7}=0$}; 
	\node[right, blue] at ($(7)+(-0.65,-2)$){$\sigma^2_{v_8}=0$}; 
	\node[right, blue] at ($(8)+(-0.65,-2)$){$\sigma^2_{v_9}=0$};

\end{scope}

%Legende Kanten 
%\begin{scope}[shift={(-4.25,0.5)}]
%	\begin{scope}[shift={(10,-3.5)}]
%		\Job{(0,0)}{(0.25,0.25)}{\gruen, \gruen}{}{}
%		\node[right] at (0.25,0.15){$A^{\text{fix}}$}; 
%	\end{scope}
%	\begin{scope}[shift={(10,-4)}]
%		\Job{(0,0)}{(0.25,0.25)}{black, black}{}{}
%		\node[right] at (0.25,0.15){$A^{\text{free}}$}; 
%	\end{scope}
%	\begin{scope}[shift={(11.5,-3.5)}]
%		\Job{(0,0)}{(0.25,0.25)}{orange, orange}{}{}
%		\node[right] at (0.25,0.15){Balance $b^1$}; 
%	\end{scope}
%	\begin{scope}[shift={(11.5,-4)}]
%		\Job{(0,0)}{(0.25,0.25)}{blue, blue}{}{}
%		\node[right] at (0.25,0.15){Balance $b^2$}; 
%	\end{scope}
%	\begin{scope}[shift={(13.75,-3.5)}]
%	\Job{(0,0)}{(0.25,0.25)}{Brombeer, Brombeer}{}{}
%	\node[right] at (0.25,0.15){Flow $f^1$, $f^2$}; 
%\end{scope}
%\end{scope}

\begin{scope}[shift={(-4.5,0.5)}]
	\begin{scope}[shift={(7.5,-3.5)}]
		\Job{(0,0)}{(0.25,0.25)}{\gruen, \gruen}{}{}
		\node[right] at (0.25,0.15){\small $A^{\text{fix}}$}; 
	\end{scope}
	\begin{scope}[shift={(8.7,-3.5)}]
		\Job{(0,0)}{(0.25,0.25)}{black, black}{}{}
		\node[right] at (0.25,0.15){\small $A^{\text{free}}$}; 
	\end{scope}
	\begin{scope}[shift={(10,-3.5)}]
		\Job{(0,0)}{(0.25,0.25)}{orange, orange}{}{}
		\node[right] at (0.25,0.15){\small Balance $b^1$}; 
	\end{scope}
	\begin{scope}[shift={(12,-3.5)}]
		\Job{(0,0)}{(0.25,0.25)}{blue, blue}{}{}
		\node[right] at (0.25,0.15){\small Balance $b^2$}; 
	\end{scope}
	\begin{scope}[shift={(14,-3.5)}]
		\Job{(0,0)}{(0.25,0.25)}{Brombeer, Brombeer}{}{}
		\node[right] at (0.25,0.15){\small Flow $f^1$, $f^2$}; 
	\end{scope}
\end{scope}
\end{tikzpicture}
	\caption{An example of a \ProblemName{} instance based on a pearl digraph}
	\label{Fig:PearlGraph}
\end{figure}
In the following, we use this special structure and shrink every multi-arc of the pearl digraph to a multi-arc consisting of the cheapest fixed and the cheapest free arc (if they exist). 
As we do not consider arc capacities we may assume that an optimal robust flow always uses the cheapest fixed and free arc of a multi-arc. 
If this was not the case, the flow could be shifted and sent at the same or lower cost.
In addition, we can even remove the remaining fixed arc if a remaining free arc of the same multi-arc exists that is equal or less expensive. 
For a proof, we refer to the parallel-shrinking procedure of the preliminary version of this paper~\cite{busing2022complexity}.
As a result, we may assume without loss of generality that a pearl digraph $G=(V,A)$ is given with vertex set $V=\{v_1,\ldots,v_n\}$ and arc set $A = \{a_i=(a^{\fixhelp}_i,a^{\freehelp}_i) \mid a^{\fixhelp}_i=(v_i,v_{i+1}), a^{\freehelp}_i=(v_i,v_{i+1}) \text{ for }i\in [n-1] \text{ (if they exist)}\}$.
The arc set $A$ is divided in sets $\fix$ and $\free$ such that $a_i^{\fixhelp}\in \fix$ and $a_i^{\freehelp}\in \free$ hold for all $i\in [n-1]$. 
Let $(G,c,\boldsymbol{b})$ be a corresponding \ProblemName{} instance. 
If arcs $a^{\fixhelp}_i$ and $a^{\freehelp}_i$ exist for $i\in [n-1]$, we note that $c(a^{\fixhelp}_i)< c(a^{\freehelp}_i)$ holds. 
An example of an \ProblemName{} instance on a pearl digraph is visualized in Figure~\ref{Fig:PearlGraph}. 
Let $\sigma^\lambda_{v_i}$ denote the sum of the balances of vertices $v_1,\ldots,v_i\in V$ in scenario $\lambda\in \Lambda$, i.e., $\sigma^\lambda_{v_i} = \sum_{j=1}^{i} b^\lambda(v_j)$. 
We refer to $\sigma^\lambda_{v_i}$ as the \textit{state of balance}, in short \textit{state}, at vertex $v_i\in V$ in scenario $\lambda\in \Lambda$. 
We note that $\sigma^\lambda_{v_i}$ units need to be sent from vertex $v_i$ to vertex $v_{i+1}$ to satisfy the balance at vertex $v_i$. 
Furthermore, we note that state $\sigma^\lambda_{v_i}$ is non-negative at every vertex $v_i\in V$ for all scenarios $\lambda\in \Lambda$. 
If this was not the case, the total supply of the predecessor vertices $v_1,\ldots,v_{i-1}$ would not meet the total demand of vertices $v_1,\ldots,v_i$. 
By parameter $\sigma^{\min}_{v_i}$, we indicate the minimum state at vertex $v_i\in V$ among all scenarios $\lambda\in \Lambda$, i.e., $\sigma^{\min}_{v_i} = \min_{\lambda\in \Lambda} \sigma^\lambda_{v_i}$. 
Using the states at the vertices, we formulate the following result about the existence of an optimal robust flow.
\begin{lemma}\label{RFlemChainGraph}
Let $\mathcal{I}=(G,c,\boldsymbol{b})$ be a \ProblemName{} instance on pearl digraph $G=(V,A=\fix\cup\free)$. 
There exists an optimal robust $\boldsymbol{b}$-flow $\boldsymbol{f}=(f^1,\ldots,f^{|\Lambda|})$ with scenario flows $f^\lambda$, $\lambda\in \Lambda$ defined as follows 
\begin{align*}
	f^\lambda(a) = \begin{cases}
		\sigma^{\min}_{v_i} & \text{for all $a=a^{\fixhelp}_i\in \fix $ if $a^{\freehelp}_i\in \free$ exists,} \\
		\sigma^{\lambda}_{v_i} - \sigma^{\min}_{v_i} & \text{for all $a=a^{\freehelp}_i\in \free  $ if  $a^{\fixhelp}_i\in \fix$ exists,}\\
		\sigma^{\lambda}_{v_i}  & \text{for all $a= a^{\fixhelp}_i\in \fix$ if there does not exist $a^{\freehelp}_i$,}\\
		\sigma^{\lambda}_{v_i}  & \text{for all $a= a^{\freehelp}_i\in \free $ if there does not exist $a^{\fixhelp}_i$.}
	\end{cases}
\end{align*}
\end{lemma}
\begin{proof}
%If balance $\fett{b}$ allows a feasible robust flow for instance $\mathcal{I}$, then 
Flow $\fett{f}$ defined as above is a feasible robust $\fett{b}$-flow due to the following two reasons. 
Firstly, flow $f^\lambda$, $\lambda\in\Lambda$ satisfies the flow balance constraints at every vertex $v_i\in V$ as $\sigma^\lambda_{v_{i}}$ units are sent along multi-arc $a_i\in A$ for all $i\in [n-1]$, i.e., 
\begin{align*}
	\sum_{a_i=(v_i,v_{i+1})\in A} f^\lambda(a_i) - \sum_{a_{i-1}=(v_{i-1},v_i)\in A} f^\lambda(a_{i-1}) 
	= \sigma^\lambda_{v_i} - \sigma^\lambda_{v_{i-1}} 
	= \sum_{j=1}^{i} b^\lambda(v_j) - \sum_{j=1}^{i-1} b^\lambda(v_j) = b^\lambda(v_i).
\end{align*}
If an (multi-)arc consists of a fixed arc only, i.e., $a=a^{\fixhelp}_i\in A$, we note that $\sigma^{\lambda}_{v_i}=\sigma^{\min}_{v_i}$ must hold true for all scenarios $\lambda\in\Lambda$.
Otherwise, there does not exist a feasible robust $\fett{b}$-flow for instance $\mathcal{I}$. 
Secondly, flows $f^1,\ldots,f^{|\Lambda|}$ satisfy the consistent flow constraints as $\sigma^{\min}_{v_i}$ units are sent along fixed arc $a^{\fixhelp}_i\in \fix$ in all scenarios $\lambda\in \Lambda$.
Assume flow $\fett{f}$ is not optimal. 
There exists an optimal robust $\fett{b}$-flow $\fett{\tilde{f}}=(\tilde{f}^1,\ldots,\tilde{f}^{|\Lambda|})$ with less cost, i.e., $c(\fett{\tilde{f}})<c(\fett{f})$. 
Due to the consistent flow constraints, it holds $\tilde{f}^\lambda(a^{\fixhelp}_i) \leq \sigma^{\min}_{v_i}$ for all fixed arcs $a^{\fixhelp}_i\in \fix$ and all scenarios $\lambda\in \Lambda$. 
As all multi-arcs are of the form $a_i=(a^{\fixhelp}_i,a^{\freehelp}_i)\in A$ (if arcs $a^{\fixhelp}_i,a^{\freehelp}_i$ exist), there must exist at least one multi-arc arc $a_i\in A$ such that $\tilde{f}^\lambda(a^{\fixhelp}_i) < \sigma^{\min}_{v_i}$ and $\tilde{f}^\lambda(a^{\freehelp}_i)  > \sigma^{\lambda}_{v_i} -\sigma^{\min}_{v_i}$ hold for $\lambda\in\Lambda$ (in order that $\fett{\tilde{f}}$ is feasible and $\fett{\tilde{f}}\neq \fett{f}$). 
We shift as many units of flow $\tilde{f}^\lambda$, $\lambda\in \Lambda$ from the free arc $a^{\freehelp}_i$ to the fixed arc $a^{\fixhelp}_i$ until $\sigma^{\min}_{v_i}$ units are sent along fixed arc $a^{\fixhelp}_i$. 
As $c(a^{\freehelp}_i) > c(a^{\fixhelp}_i)$ holds, we obtain a feasible robust $\fett{b}$-flow $\fett{\hat{f}}$ with cost $c(\fett{\hat{f}})<c(\fett{\tilde{f}})<c(\fett{f})$, which contradicts the assumption. 
\end{proof}
Using Lemma~\ref{RFlemChainGraph}, we present Algorithm~\ref{RFAlgoChainGraph} to solve the \ProblemName{} problem on pearl digraphs.
\begin{algorithm}[H]
\caption{}\label{RFAlgoChainGraph}
\begin{lyxlist}{Initialization:}
	\item [{Input:}] \ProblemName{} instance $(G,c,\boldsymbol{b})$ where $G=(V,A=\fix\cup\free)$ is a pearl digraph 
	\item [{Output:}] Robust minimum cost $\boldsymbol{b}$-flow $\boldsymbol{f}$
	\item [{Method:}]~ 	\end{lyxlist}
\begin{algorithmic}[1]
	\State Compute state $\sigma^\lambda_{v_i}$ for every vertex $v_i\in V$ and every scenario $\lambda \in \Lambda$
	\State Compute the minimum state $\sigma^{\min}_{v_i}$ for every vertex $v_i\in V$
	\State{For every scenario $\lambda\in \Lambda$ construct the flow 
		\begin{align*}
		f^\lambda(a) =
		 \begin{cases}
		 	\sigma^{\min}_{v_i} & \text{for all $a=a^{\fixhelp}_i\in \fix $ if $a^{\freehelp}_i\in \free$ exists,} \\
		 	\sigma^{\lambda}_{v_i} - \sigma^{\min}_{v_i} & \text{for all $a=a^{\freehelp}_i\in \free  $ if  $a^{\fixhelp}_i\in \fix$ exists,}\\
		 	\sigma^{\lambda}_{v_i}  & \text{for all $a= a^{\fixhelp}_i\in \fix$ if there does not exist $a^{\freehelp}_i$,}\\
		 	\sigma^{\lambda}_{v_i}  & \text{for all $a= a^{\freehelp}_i\in \free $ if there does not exist $a^{\fixhelp}_i$}
		 \end{cases}
		\end{align*}
	}\\
	\Return Robust $\boldsymbol{b}$-flow $\boldsymbol{f}=(f^1,\ldots,f^{|\Lambda|})$
\end{algorithmic}
\end{algorithm}
We conclude this section with the polynomial-time solvability of the \ProblemName{} problem on pearl digraphs as shown in the following theorem. 
\begin{theorem}
Let $(G,c,\fett{b})$ be a \ProblemName{} instance on pearl digraph $G=(V,A=\fix\cup\free)$. Algorithm~\ref{RFAlgoChainGraph} computes an optimal robust $\fett{b}$-flow in $\mathcal{O}(|V|\cdot |\Lambda|)$ time.
\end{theorem}
\begin{proof}
The correctness of the algorithm results from Lemma~\ref{RFlemChainGraph}. 
Considering the runtime, we obtain the following. 	
The states can be computed at all vertices in $\mathcal{O}(|V|\cdot |\Lambda|)$ time. 
The minimum state can be computed at every vertex in $\mathcal{O}(|V|\cdot |\Lambda|)$ time. 
The construction of the flow is done in $\mathcal{O}(1)$ time. 
In total, the algorithm runs in  $\mathcal{O}(|V|\cdot |\Lambda|)$ time. 
\end{proof}
%%%%%%
%%%%--------------------------------------------------------------------------------------------
\section{Conclusion}
\label{Sec:Conclusion}
In this paper, we considered the \ProblemName{} problem, the uncapacitated version of the \FirstProblem{} problem.
As the complexity results of the \FirstProblem{} problem depend on the arc capacities given, we analyzed the complexity of the \ProblemName{} problem.
On acyclic digraphs, we proved that finding a feasible solution to the \ProblemName{} problem is strongly $\mathcal{NP}$-complete, even if only two scenarios are considered that have the same unique source and unique sink.
On SP digraphs, we obtained the following results. 
We proved the weak $\mathcal{NP}$-completeness, even if only two scenarios are considered that have the same unique source and single (but different) sinks.
For the special case of a constant number of scenarios, we observed the pseudo-polynomial-time solvability.
For the special case of a unique source and a unique sink, we presented a polynomial-time algorithm. 
For the special case of a unique source and parallel sinks or parallel sources and a unique sink, we also presented a polynomial-time algorithm. 
Furthermore, we observed that the algorithm is extendable for the special case of parallel sources and parallel sinks if there exists a path between each source and each sink. 
For the special case of pearl digraphs, we presented a polynomial-time algorithm, independent of the number of sources and sinks. 

For the future work, we will consider the \ProblemName{} problem on SP digraphs if the number of scenarios is part of the input. Furthermore, we will study the problem for further graph classes as digraphs with bounded treewidth.

\section*{\normalsize{ACKNOWLEDGMENTS}}
We would like to recognize the invaluable assistance of Anna Margarethe Limbach who contributed to the proof of Theorem~\ref{RFtheorem:Potenzen} in Appendix~\ref{RF:AppendixB}. 

%%--------------------------------------------------------------------------------------------------------------------------------------------------------------------
\bibliographystyle{plain}
\bibliography{Quellen}   % name your BibTeX data base
\renewcommand{\appendixname}{}
\renewcommand*{\thesection}{\appendixname~\Alph{section}}
\appendix
\section{Appendix}
\label{RF:AppendixA}
\begin{theorem}\label{RF:NPCompletePairPartition}
	The \RFPairPartition{} problem is weakly $\mathcal{NP}$-complete. 
\end{theorem}
\begin{proof}
	The \RFPairPartition{} problem is in $\mathcal{NP}$ as we can check in polynomial time if each pair is separated and if the sum of the integers are equal in both subsets. 
	Let $\widetilde{\mathcal{I}}$ be a \textsc{Partition} instance with positive integers $\tilde{s}_1,\ldots,\tilde{s}_n$ whose sum is $2\tilde{w}$. 
	We construct a \RFPairPartition{} instance $\mathcal{I}$.  
	We set the $2n$ positive integers as follows
	\begin{align*}
		s_{2i-1} &= \tilde{s}_i + 1,\\
		s_{2i} &= 1
	\end{align*}
	for $i\in [n]$ and partition them in pairs $S_{i}=\{s_{2i-1},s_{2i}\}$. 
	The integers $s_1,\ldots,s_{2n}$ sum up to $2\tilde{w} + 2n = 2w$ with $w:=\tilde{w} + n$, i.e., $\sum_{i=1}^{2n}s_i = 2\tilde{w} + 2n =2w$.
	We obtain a feasible \RFPairPartition{} instance $\mathcal{I}$, which is constructed in polynomial time. 
	Hence, it remains to show that $\widetilde{\mathcal{I}}$ is a Yes-instance if and only if there exists a feasible partition for the \RFPairPartition{} instance $\mathcal{I}$.
	
	Let $\widetilde{S}^1$, $\widetilde{S}^2$ be a feasible partition for instance $\widetilde{\mathcal{I}}$. 
	We define a partition for the \RFPairPartition{} instance $\mathcal{I}$ by subsets
	\begin{align*}
		S^1 &= \{s_{2i} \mid \tilde{s}_i \in \widetilde{S}^1\} \cup \{s_{2i-1} \mid \tilde{s}_i \in \widetilde{S}^2\}, \\
		S^2 &= \{s_{2i} \mid \tilde{s}_i \in \widetilde{S}^2\} \cup \{s_{2i-1} \mid \tilde{s}_i \in \widetilde{S}^1\}.
	\end{align*}
	As $\widetilde{S}^1$, $\widetilde{S}^2$ is a feasible partition, each pair $S_{i}=\{s_{2i-1},s_{2i}\}$, $i\in [n]$ is separated by partition $S^1$, $S^2$. 
	We further obtain 
	\begin{align*}
		\sum_{s_i\in S^1} s_i 
	 &= \sum_{s_{2i-1}\in S^1} s_{2i-1} + \sum_{s_{2i}\in S^1} s_{2i} 
		= \sum_{\tilde{s}_{i}\in \widetilde{S}^2} s_{2i-1} + \sum_{\tilde{s}_{i}\in \widetilde{S}^1} s_{2i} 
		= \sum_{\tilde{s}_{i}\in \widetilde{S}^2} \left(\tilde{s}_{i}+1\right) + \sum_{\tilde{s}_{i}\in \widetilde{S}^1} 1
		%= \sum_{s_{i}\in \widetilde{S}^2} \tilde{s}_{i}+ 2n
		= \tilde{w}+ n 
		= w \\
		& = \sum_{\tilde{s}_{i}\in \widetilde{S}^1} \left(\tilde{s}_i +1\right) + \sum_{\tilde{s}_{i}\in \widetilde{S}^2} 1 
		= \sum_{\tilde{s}_{i}\in \widetilde{S}^1} s_{2i-1} + \sum_{\tilde{s}_{i}\in \widetilde{S}^2} s_{2i} 
		= \sum_{s_{2i-1}\in S^2} s_{2i-1} + \sum_{s_{2i}\in S^2} s_{2i} 	= \sum_{s_{i}\in S^2} s_{i} .
	\end{align*}
	Consequently, we have constructed a feasible partition for the \RFPairPartition{} instance $\mathcal{I}$. 
	Conversely, let $S^1$, $S^2$ be a feasible partition for instance $\mathcal{I}$. 
	The partition is also feasible for the \mbox{\textsc{Partition} instance $\widetilde{\mathcal{I}}$}. 
\end{proof}
\section{Appendix}
\label{RF:AppendixC}
%\auxiliarylemmadecreasingflowonfixedarcs*
\begin{proof}[Proof of Lemma~\ref{RFlem:BalanceConstraintsMaxSplitInstance}]
	Assume the statement is false. 
	There exists an optimal robust $\fett{b}$-flow $\fett{f}=(f^1,f^2)$ such that for two successive arcs $a^c_{i}$ and $a^c_{i+1}$, $i\in [2n-2]$ the following holds true
	\begin{align}
		f^1(a^c_i) =f^2(a^c_i) < f^1(a^c_{i+1}) = f^2(a^c_{i+1}). \label{RFequ:proofassumptionflowOnfixedarcs}
	\end{align}
	For every $i^\prime\in [2n-2]$ and $\lambda\in \Lambda$ it holds
	\begin{align}
		f^\lambda(a^d_{i^\prime}) + f^\lambda(a^c_{i^\prime}) = f^\lambda(a^d_{i^\prime+1}) + f^\lambda(a^c_{i^\prime+1}) + f^\lambda(a^s_{i^\prime+1}) \label{RFequ:proofflowbalanceconstraints}
	\end{align}
	due to the flow balance constraints.
	Transforming equation~\eqref{RFequ:proofflowbalanceconstraints} and using the assumption~\eqref{RFequ:proofassumptionflowOnfixedarcs}, for $\lambda\in \Lambda$ we obtain 
	$$f^\lambda(a^d_i) > f^\lambda(a^c_{i+1}) - f^\lambda(a^c_{i}) =: \beta > 0. $$
	Let $\fett{\tilde{f}}=(\tilde{f}^1,\tilde{f}^2)$ be the robust $\fett{b}$-flow which results by redirecting $\beta$ units of flow $f^\lambda$, $\lambda\in \Lambda$ from arc $a^d_i$ to arc $a^c_i$, i.e., 
	\begin{align*}
		\tilde{f}^\lambda(a) = 
		\begin{cases}
			f^\lambda(a)+\beta & \text{ for arc }a=a^c_i,\\
			f^\lambda(a)-\beta & \text{ for arc }a=a^d_i,\\
			f^\lambda(a) & \text{ otherwise.}
		\end{cases}
	\end{align*}
	An example is visualized in Figure~\ref{RFFig:BalanceConstraintsMaxSplitInstance}.
	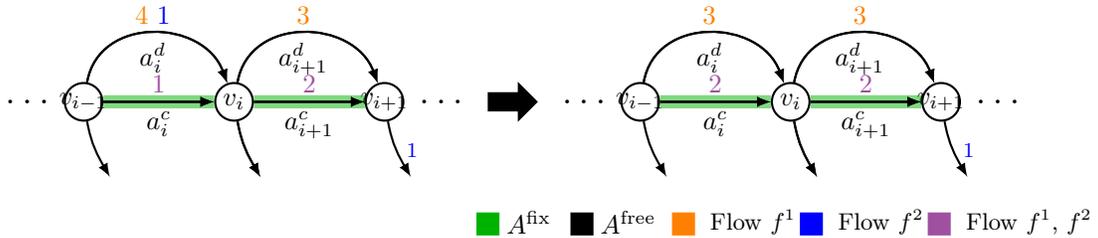
\begin{figure}[H]
		\centering
		\begin{tikzpicture}

\node(dots) at (-0.75,0){\Large $\ldots$};
\Knoten{(s)}{(0,0)}{$v_{i-1}$}{thick, minimum size=0.5cm}
\Knoten{(1)}{(2,0)}{$v_i$}{thick, minimum size=0.5cm}
\Knoten{(2)}{(4,0)}{$v_{i+1}$}{thick, minimum size=0.5cm}
\node(dots) at (4.75,0){\Large $\ldots$};

\draw[unKante,line width=5,  opacity=0.5,\gruen] (s) --(1);
\draw[thick,->, -latex, black] (s)  to [out=80, in=110, looseness=1.25]node[below]{$a^d_i$}(1); 
\draw[thick,->, -latex, black] (s)  to [out=80, in=110, looseness=1.25]node[above]{\textcolor{orange}{$4$}\ \textcolor{blue}{$1$}}(1); 
\draw[thick,->, -latex, black] (s)  --node[below]{$a^c_i$}(1);
\draw[thick,->, -latex, black] (s)  --node[above, Brombeer]{$1$}(1);  
\draw[unKante,line width=5,  opacity=0.5,\gruen] (1) --(2);
\draw[thick,->, -latex, black] (1)  to [out=80, in=110, looseness=1.25]node[below]{$a^d_{i+1}$}(2); 
\draw[thick,->, -latex, black] (1)  to [out=80, in=110, looseness=1.25]node[above, orange]{$3$}(2); 
\draw[thick,->, -latex, black] (1)  --node[below]{$a^c_{i+1}$}(2); 
\draw[thick,->, -latex, black] (1)  --node[above, Brombeer]{$2$}(2); 

\draw[thick,->, -latex, ] (s)  to [out=280, in=120]node[right, blue]{\footnotesize}($(s)+(0.35,-1)$); 
\draw[thick,->, -latex, ] (1)  to [out=280, in=120]node[right, blue]{\footnotesize}($(1)+(0.35,-1)$); 
\draw[thick,->, -latex, ] (1)  to [out=280, in=120]node[right, blue]{}($(1)+(0.35,-1)$); 
\draw[thick,->, -latex, ] (2)  to [out=280, in=120]node[right, blue]{\footnotesize $1$}($(2)+(0.35,-1)$); 

\Pfeil{(5,1.1)}{0.4}{fill=black}

\begin{scope}[shift={(7.4,0)}]
\node(dots) at (-0.75,0){\Large $\ldots$};
\Knoten{(s)}{(0,0)}{$v_{i-1}$}{thick, minimum size=0.5cm}
\Knoten{(1)}{(2,0)}{$v_i$}{thick, minimum size=0.5cm}
\Knoten{(2)}{(4,0)}{$v_{i+1}$}{thick, minimum size=0.5cm}
\node(dots) at (4.75,0){\Large $\ldots$};

\draw[unKante,line width=5,  opacity=0.5,\gruen] (s) --(1);
\draw[thick,->, -latex, black] (s)  to [out=80, in=110, looseness=1.25]node[below]{$a^d_i$}(1); 
\draw[thick,->, -latex, black] (s)  to [out=80, in=110, looseness=1.25]node[above]{\textcolor{orange}{$3$}}(1); 
\draw[thick,->, -latex, black] (s)  --node[below]{$a^c_i$}(1);
\draw[thick,->, -latex, black] (s)  --node[above, Brombeer]{$2$}(1);  
\draw[unKante,line width=5,  opacity=0.5,\gruen] (1) --(2);
\draw[thick,->, -latex, black] (1)  to [out=80, in=110, looseness=1.25]node[below]{$a^d_{i+1}$}(2); 
\draw[thick,->, -latex, black] (1)  to [out=80, in=110, looseness=1.25]node[above, orange]{$3$}(2); 
\draw[thick,->, -latex, black] (1)  --node[below]{$a^c_{i+1}$}(2); 
\draw[thick,->, -latex, black] (1)  --node[above, Brombeer]{$2$}(2); 

\draw[thick,->, -latex, ] (s)  to [out=280, in=120]node[right, blue]{\footnotesize}($(s)+(0.35,-1)$); 
\draw[thick,->, -latex, ] (1)  to [out=280, in=120]node[right, blue]{\footnotesize}($(1)+(0.35,-1)$); 
\draw[thick,->, -latex, ] (1)  to [out=280, in=120]node[right, blue]{}($(1)+(0.35,-1)$); 
\draw[thick,->, -latex, ] (2)  to [out=280, in=120]node[right, blue]{\footnotesize $1$}($(2)+(0.35,-1)$); 
\end{scope}

% Legende
\begin{scope}[shift={(-5.5,1.75)}]
\begin{scope}[shift={(10.75,-3.5)}]
\Job{(0,0)}{(0.25,0.25)}{\gruen, \gruen}{}{}
\node[right] at (0.25,0.15){$A^{\text{fix}}$}; 
\end{scope}
\begin{scope}[shift={(12,-3.5)}]
\Job{(0,0)}{(0.25,0.25)}{black, black}{}{}
\node[right] at (0.25,0.15){$A^{\text{free}}$}; 
\end{scope}
\begin{scope}[shift={(13.35,-3.5)}]
\Job{(0,0)}{(0.25,0.25)}{orange, orange}{}{}
\node[right] at (0.25,0.15){\small{ Flow} $f^1$}; 
\end{scope}
\begin{scope}[shift={(15.05,-3.5)}]
	\Job{(0,0)}{(0.25,0.25)}{blue, blue}{}{}
	\node[right] at (0.25,0.15){\small{ Flow} $f^2$}; 
\end{scope}
\begin{scope}[shift={(16.75,-3.5)}]
	\Job{(0,0)}{(0.25,0.25)}{Brombeer, Brombeer}{}{}
	\node[right] at (0.25,0.15){\small{ Flow} $f^1$, $f^2$}; 
\end{scope}
\end{scope}
\end{tikzpicture}
		\caption{Example for the shift of flow from arc $a^d_i$ to arc $a^c_i$ such that $f^\lambda(a^c_i) =  f^\lambda(a^c_{i+1})$ holds for $\lambda\in \Lambda$ }
		\label{RFFig:BalanceConstraintsMaxSplitInstance}
	\end{figure}
	As $c(a^d_{i^\prime})>c(a^c_{i^\prime})$ holds for all $i^\prime\in [2n-1]$, we obtain a lower bound on the cost of flow $f^\lambda$, $\lambda\in \Lambda$ as follows
	\begin{align*}
		c(f^\lambda) 	&= \sum_{a\in A_n} c(a) \cdot f^\lambda(a) \\
		&= \sum_{a\in A_n\setminus \{a^c_i, a^d_i\}}  c(a) \cdot f^\lambda(a)  + c(a^c_i) \cdot f^\lambda(a^c_i) +  c(a^d_i) \cdot f^\lambda(a^d_i)  \\
		&> \sum_{a\in A_n\setminus \{a^c_i, a^d_i\}}  c(a) \cdot f^\lambda(a)  + c(a^c_i) \cdot \left( f^\lambda(a^c_i)+\beta\right)  +  c(a^d_i) \cdot \left(f^\lambda(a^d_i)-\beta\right)  \\
		&= \sum_{a\in A_n\setminus \{a^c_i, a^d_i\}}  c(a) \cdot \tilde{f}^\lambda(a)  +  c(a^c_i) \cdot \tilde{f}^\lambda(a^c_i) +  c(a^d_i) \cdot \tilde{f}^\lambda(a^d_i)  \\
		&= \sum_{a\in A_n} c(a) \cdot \tilde{f}^\lambda(a)  \\
		&= c(\tilde{f}^\lambda). 
	\end{align*}
	Thus, it holds $c(\fett{f})=\max\{c(f^1), c(f^2)\}>\max \{c(\tilde{f}^1), c(\tilde{f}^2)\}= c(\fett{\tilde{f}})$, which contradicts the assumption. 
\end{proof}

%\AuxiliaryLemmaCostOnFixedArcs
\begin{proof}[Proof of Lemma~\ref{RFlem:HilfslemmaCostOfFixedArcsForOptimalSol}]
	Before assuming the statement is false, we consider two robust $\boldsymbol{b}$-flows $\boldsymbol{f}=(f^1,f^2)$ and $\boldsymbol{\hat{f}}=(\hat{f}^1,\hat{f}^2)$ for instance $\mathcal{I}_n$ given as follows.
	
	For every $i\in [n]$, flows $f^1$ and $f^2$ send one unit along paths $p^1_{2i}$ and $p^2_{2i}$, respectively.
	The cost of flow $\fett{f}$ is $c(\boldsymbol{f}) = \max \{ c(f^1), c(f^2) \} $ with
	\begin{align*}  
		c(f^1)	
		&= \sum_{i=1}^{n} c(p^1_{2i}) 
		= \sum_{i=1}^{n} \left(nw(2^{n-i}+2^{n-1}-1) +s_{2i} \right)
		= nw\sum_{i=1}^{n} 2^{n-i}+ n^2w(2^{n-1}-1) +\sum_{i=1}^{n} s_{2i} \\
		&= nw(2^{n}-1)+ n^2w(2^{n-1}-1) +\sum_{i=1}^{n} s_{2i} 
	\end{align*}
	and 
	\begin{align*}  
		c(f^2)	
		&= \sum_{i=1}^{n} c(p^2_{2i}) 
		= \sum_{i=1}^{n} \left(nw (2^{i}-2)2^{n-i-1}  + s_{2i-1} + Fw \right)
		= nw \sum_{i=1}^{n} \left((2^{i}-2)2^{n-i-1} \right) +  \sum_{i=1}^{n}  s_{2i-1} + nw(2^{n+1}-n-2)\\
		&= n^2w2^{n-1}-nw(2^n-1) + \sum_{i=1}^{n}  s_{2i-1} + nw(2\cdot 2^{n}-n-2)
		= nw(2^{n}-1)+ n^2w(2^{n-1}-1) +\sum_{i=1}^{n} s_{2i-1}.
	\end{align*}
	As $s_{2i-1}\geq s_{2i}$ for all $i\in [n]$, it holds $c(f^1)\leq c(f^2)$. 
	
	For every $i\in [n]$, flows $\hat{f}^1$ and $\hat{f}^2$ send one unit along paths $p^1_{2i-1}$ and $p^2_{2i-1}$, respectively. 
	The cost of flow $\fett{\hat{f}}$ is $c(\boldsymbol{\hat{f}}) = \max \{ c(\hat{f}^1), c(\hat{f}^2) \} $ with
	\begin{align*}  
		c(\hat{f}^1)	
		& =  \sum_{i=1}^{n} c(p^1_{2i-1}) 
		= \sum_{i=1}^{n} \left(nw(2^{n-i}+2^{n-1}-1) +s_{2i-1} \right)
		= nw\sum_{i=1}^{n} 2^{n-i}+ n^2w(2^{n-1}-1) +\sum_{i=1}^{n} s_{2i-1} \\
		&= nw(2^{n}-1)+ n^2w(2^{n-1}-1) +\sum_{i=1}^{n} s_{2i-1} 
	\end{align*}
	and 
	\begin{align*}  
		c(\hat{f}^2)	
		&= \sum_{i=1}^{n} c(p^2_{2i-1}) 
		= \sum_{i=1}^{n} \left(nw (2^{i}-2)2^{n-i-1}  + s_{2i} + Fw \right)
		= nw \sum_{i=1}^{n} \left((2^{i}-2)2^{n-i-1} \right) +  \sum_{i=1}^{n}  s_{2i} + nw(2^{n+1}-n-2)\\
		&= n^2w2^{n-1}-nw(2^n-1) + \sum_{i=1}^{n}  s_{2i} + nw(2\cdot 2^{n}-n-2)
		= nw(2^{n}-1)+ n^2w(2^{n-1}-1) +\sum_{i=1}^{n} s_{2i}.
	\end{align*}
	As $s_{2i-1}\geq s_{2i}$ for all $i\in [n]$, it holds $c(\hat{f}^1)\geq c(\hat{f}^2)$.
	
	On the fixed arcs, flows $f^\lambda$ and $\hat{f}^{\lambda}$ with $\lambda\in \Lambda$ cause cost amounting to
	\begin{align*}
		\sum_{a\in A^{\fixhelp}_n}  c(a) \cdot f^\lambda(a) 	
		&= \sum^{n}_{i=1} c(a^c_{2i-1}) \cdot (n-i+1)   + 	\sum^{n-1}_{i=1} c(a^c_{2i}) \cdot (n-i) 			 						
		= 0 + nw(2^{n-1}n-2^n+1) = c^{\fixhelp}\\
		&= \sum^{n}_{i=1} c(a^c_{2i-1})\cdot (n-i) + \sum^{n-1}_{i=1} c(a^c_{2i})\cdot (n-i)  			
		= \sum_{a\in A^{\fixhelp}_{n}} c(a) \cdot \hat{f}^\lambda(a).
	\end{align*}
	
	Now assume there exists an optimal robust $\boldsymbol{b}$-flow $\boldsymbol{\tilde{f}}=(\tilde{f}^1, \tilde{f}^2)$ whose cost incurred on the fixed arcs, denoted by $c^{\fixhelp}_{\fett{\tilde{f}}}$, does not amount to $c^{\fixhelp}$, i.e., 
	$c^{\fixhelp}_{\fett{\tilde{f}}}= \sum_{i=1}^{2n-1}c(a^c_{i})\cdot \tilde{f}^1(a^c_{i}) = \sum_{i=1}^{n-1}c(a^c_{2i})\cdot \tilde{f}^1(a^c_{2i})\neq c^{\fixhelp}$.
	Firstly, assume that $\fett{\tilde{f}}=(\tilde{f}^1,\tilde{f}^2)$ is given such that the cost incurred on the fixed arcs $c^{\fixhelp}_{\fett{\tilde{f}}}$ is less than $c^{\fixhelp}$, i.e., $c^{\fixhelp}_{\fett{\tilde{f}}} < c^{\fixhelp}$. 
	As the costs of the fixed arcs are $c(a^c_{2i-1})=0$ for $i\in [n]$ and $c(a^c_{2i})=2^{n-i-1}nw$ for $i\in [n-1]$, it holds $c^{\fixhelp} - c^{\fixhelp}_{\fett{\tilde{f}}}\geq  2^0nw= nw\geq 2w$. 
	In the following, we focus on the first scenario flow $\tilde{f}^1$.
	Let $I$ denote the set of indices that indicate along which paths $p^1_i$ flow $\tilde{f}^1$ sends units, i.e., $I:=\{i\in [2n]\mid \tilde{f}^1(p^1_i)>0\}$. 
	If a unit is sent along path $p^1_i$, $i\in I$, integer $s_i$ is added to the total cost of flow $\tilde{f}^1$.
	The integers of all other arc cost add up to zero because of the telescope sum.  
	Accordingly, as $n$ units are sent from the source along paths $p^1_i$, $i\in I$ to the sink, exactly $n$ (not necessarily different) integers are added to the total cost of flow $\tilde{f}^1$. 
	For the first scenario, the minimum cost path is given by $p^1_{2n}$ as $c(a^c_i)\leq c(a^d_i)$ holds for all $i\in [2n-1]$.
	If flow $\tilde{f}^1$ sends $n$ units along the minimum cost path $p^1_{2n}=sa^c_1v_1a^c_2\ldots a^c_{2n-1}v_{2n-1}a^d_{2n}t_1$, the cost would be
	\begin{align}\label{RFexpr:bestCaseCost}
		n \cdot c(p^1_{2n})
		=n \left(  \sum_{i=1}^{2n-1} c(a^c_{i}) + c(a^d_{2n}) \right)
		=n \left( \sum_{i=1}^{n-1} c(a^c_{2i}) + c(a^d_{2n}) \right).
	\end{align}
	However, if $I\setminus \{2n\}\neq \emptyset$ holds, flow $\tilde{f}^1$ sends a number of units along different paths $p^1_{i}=sa^c_1v_1a^c_2\ldots a^c_{i-1}v_{i-1}a^{d}_{i}\ldots a^d_{2n}t_1$ for $i\in I\setminus \{2n\}$ (which use more detour instead of cross arcs). 
	If a detour arc $a^d_i$ is used instead of the corresponding cross arc $a^c_i$ for $i\in [2n-1]$, the cost rises by $ c(a^d_i)-c(a^c_i) = s_{i}-s_{i+1}+c(a^c_i) >0$. 
	Thus, comparing the cost incurred by sending $n$ units along path $p^1_{2n}$ with the cost incurred by sending $n$ units along different paths $p^1_{i}$ for $i\in I$, we obtain additional cost of
	\begin{align}\label{RFexpr:addidtionalCost}
		\sum_{i\in I} \left(s_i - s_{2n} + \sum_{j=i}^{2n-1} c(a^c_j)
		\right)\cdot \tilde{f}^1(p^1_i)
		&= \sum_{i\in I} \left(s_i - s_{2n} \right)\cdot\tilde{f}(p^1_i)+
		\sum^{2n-1}_{i=1}  c(a^c_i)
		\cdot \tilde{f}^1(a^d_i)\nonumber \\
		&= \sum_{i\in I} \left(s_i - s_{2n} \right)\cdot\tilde{f}(p^1_i) 
		+ \sum^{2n-1}_{i=1}   c(a^c_i)\cdot \left(n-\tilde{f}^1(a^c_i)\right) \nonumber\\
		&= \sum_{i\in I} \left(s_i - s_{2n} \right)\cdot\tilde{f}(p^1_i) 
		+  n\sum^{2n-1}_{i=1} c(a^c_i)
		- c^{\fixhelp}_{\fett{\tilde{f}}}.
	\end{align}
	Let $i^*$ be the index of the smallest integer, i.e., $i^* = \arg\min_{i\in [2n]} s_i$. 
	Combining expressions~\eqref{RFexpr:bestCaseCost} and~\eqref{RFexpr:addidtionalCost} and using the fact that $c^{\fixhelp} - 2w\geq c^{\fixhelp}_{\fett{\tilde{f}}}$ and $\sum_{i=1}^{n}s_{2i-1}<2w$ hold, we obtain a lower bound on the cost of flow $\tilde{f}^1$ as follows 
	\begin{align*}
		c(\tilde{f}^1) 	
		&= \sum_{i\in I}c(p^1_i)\cdot \tilde{f}^1(p^1_i)\\
		&=n \cdot c(p^1_{2n}) +
		\sum_{i\in I} \left(s_i - s_{2n} \right)\cdot\tilde{f}(p^1_i) 
		+  n\sum^{2n-1}_{i=1} c(a^c_i)
		- c^{\fixhelp}_{\fett{\tilde{f}}}
		\\
		&\geq  n \left( \sum_{i=1}^{n-1} c(a^c_{2i}) + c(a^d_{2n})  \right)  + n\cdot s_{i^*} -n \cdot s_{2n}  +n\sum_{i=1}^{n-1}c(a^c_{2i})- c^{\fixhelp}_{\fett{\tilde{f}}}\\
		&=  n \left( 2\sum_{i=1}^{n-1} c(a^c_{2i}) + c(a^d_{2n})-s_{2n} \right)+ n\cdot s_{i^*} - c^{\fixhelp}_{\fett{\tilde{f}}}\\
		&\geq n \left( 2\sum_{i=1}^{n-1} c(a^c_{2i})  + c(a^d_{2n})-s_{2n} \right)+ n\cdot s_{i^*}  - c^{\fixhelp} + 2w\\
		&> n \left( 2\sum_{i=1}^{n-1} c(a^c_{2i}) + c(a^d_{2n}) -s_{2n}\right)  - c^{\fixhelp} +  \sum_{i=1}^{n} s_{2i-1}\\
		&=  n \left( 2(2^{n-1}-1)nw + 2^0nw\right)  - nw(2^{n-1}n-2^n+1) + \sum_{i=1}^{n} s_{2i-1}  \\
		&= nw(2^{n}-1)+ n^2w(2^{n-1}-1) + \sum_{i=1}^{n} s_{2i-1} \\
		&=c(\hat{f}^1).
	\end{align*}
	We obtain $c(\fett{\tilde{f}})= \max\{c(\tilde{f}^1), c(\tilde{f}^2)\}\geq c(\tilde{f}^1) > c(\hat{f}^1) = c(\fett{\hat{f}})$, which contradicts the assumption. 
	
	Secondly, assume that $\fett{\tilde{f}}=(\tilde{f}^1,\tilde{f}^2)$ is given such that the cost incurred on the fixed arcs $c^{\fixhelp}_{\fett{\tilde{f}}}$ is greater than $c^{\fixhelp}$, i.e., $c^{\fixhelp}_{\fett{\tilde{f}}}>c^{\fixhelp}$.
	As the costs of the fixed arcs are $c(a^c_{2i-1})=0$ for $i\in [n]$ and $c(a^c_{2i})=2^{n-i-1}nw$ for $i\in [n-1]$, it holds $c^{\fixhelp}_{\fett{\tilde{f}}}-c^{\fixhelp}\geq 2^0nw= nw\geq 2w$. 
	In the following, we focus on the second scenario flow $\tilde{f}^2$.
	Let $i^*$ be the index of the smallest integer, i.e., $i^* = \arg\min_{i\in [2n]} s_i$. 
	We note that at least cost of $n (s_{i^*}+Fw)$ is incurred by the free shortcut arcs in the second scenario. %regardless of cross arcs are used or not,
	As $c^{\fixhelp}_{\fett{\tilde{f}}}\geq 2w+c^{\fixhelp}$ and 
	$\sum_{i=1}^{n} s_{2i-1}<2w$ hold, we obtain a lower bound on the cost of flow $\tilde{f}^2$ as follows 
	\begin{align*}
		c(\tilde{f}^2) 	
		&\geq  n( s_{i^*} + Fw) + c^{\fixhelp}_{\fett{\tilde{f}}}  
		\geq  n \cdot s_{i^*} + nFw + 2w + c^{\fixhelp}  
		>   nFw + \sum_{i=1}^{n} s_{2i-1} + c^{\fixhelp} \\
		&
		=  nw(2^{n}-1)+ n^2w(2^{n-1}-1) + \sum_{i=1}^{n} s_{2i-1}
		= c(f^2). 
	\end{align*}
	We obtain $c(\fett{\tilde{f}})= \max\{c(\tilde{f}^1), c(\tilde{f}^2)\}\geq c(\tilde{f}^2) > c(f^2) = c(\fett{f})$, which contradicts the assumption. 
	
	Consequently, the cost incurred by an optimal robust $\fett{b}$-flow $\fett{f}=(f^1,f^2)$ on the fixed arcs is $c^{\fixhelp}$.
\end{proof}

%\AuxiliaryLemmaFlowValueOnFixedArcs*
\begin{proof}[Proof of Lemma~\ref{RFlem:HilfslemmaFlowValueOnFixedArcsOfOptimalSol}]
	We assume that there exists an optimal robust $\fett{b}$-flow $\fett{\tilde{f}}=(\tilde{f}^1,\tilde{f}^2)$ with $\tilde{f}^1(a^c_{2i})=\tilde{f}^2(a^c_{2i})\neq n-i$ for at least one fixed arc $a^c_{2i}\in A_n^{\fixhelp}$.
	We define variables $x_i:=\tilde{f}^1(a^c_{2i})= \tilde{f}^2(a^c_{2i})$ for all $i\in [n-1]$. 
	On all fixed arcs, the flow is limited by the supply that is sent from the unique source to the single sink, i.e., $0\leq x_i\leq n$ for $i\in [n-1]$. 
	By Lemma~\ref{RFlem:BalanceConstraintsMaxSplitInstance}, it holds $x_i \geq x_{i+1}$ for $i\in [n-2]$.
	By Lemma~\ref{RFlem:HilfslemmaCostOfFixedArcsForOptimalSol}, the cost incurred on the fixed arcs is $c^{\text{fix}}$.
	As $c(a^c_{2i-1})=0$ holds for arcs $a^c_{2i-1}\in A_n^{\fixhelp}$, only the cost incurred by the arcs $a^c_{2i}\in A_n^{\fixhelp}$ contributes to $c^{\text{fix}}$.
	Consequently, the following must hold true
	\begin{align}
		\sum_{i=1}^{n-1} c(a^c_{2i}) \cdot  x_i 
		= \sum_{i=1}^{n-1} 2^{n-i-1}nw\cdot x_i  
		=  c^{\text{fix}} = nw(2^{n-1}n-2^n+1).\label{RFequ:onlyOneSolutionForFlowEqu}
	\end{align}
	Overall, variables $x_i$, $i\in [n-1]$ satisfy the requirements of Theorem~\ref{RFtheorem:Potenzen} in Appendix~\ref{RF:AppendixB}.
	We obtain that the only solution to equation~\eqref{RFequ:onlyOneSolutionForFlowEqu} is $x_i=(n-i)$ for all $i\in [n-1]$. 
	Thus, it holds $\tilde{f}^1(a^c_{2i}) =\tilde{f}^2(a^c_{2i}) = x_i = n-i$ for fixed arcs $a^c_{2i}\in \fix_n$, which contradicts the assumption. 
	By Lemma~\ref{RFlem:BalanceConstraintsMaxSplitInstance} and due to the flow balance constraints, it follows $\tilde{f}^1(a^c_{2i-1}) = \tilde{f}^2(a^c_{2i-1})  \in \{n-i,n-i+1\}$ for fixed arcs $a^c_{2i-1}\in A_n^{\fixhelp}$.
\end{proof}
\section{Appendix}
\label{RF:AppendixB}
\begin{theorem}\label{RFtheorem:Potenzen}
	Let non-negative variables $x_1,\ldots,x_{n-1}\in \{0,\ldots,n\}$ be given with a non-ascending restriction $x_i\geq x_{i+1}$ for all $i\in [n-2]$. 
	The equation 
	\begin{align*}
		\sum_{i=1}^{n-1} 2^{n-i-1}\cdot x_i = 2^{n-1}n-2^n+1
	\end{align*}
	has $x_i=n-i$ for all $i\in [n-1]$ as its only solution. 
\end{theorem}
To prove Theorem~\ref{RFtheorem:Potenzen}, we use the following auxiliary lemma. 
\begin{lemma}\label{RFlem:Aequivalenz}
The following statements are equivalent. 
	\begin{itemize}
			\item[(S1)] Let non-negative variables $x_1,\ldots,x_{n-1}\in \{0,\ldots,n\}$ be given with a non-ascending restriction $x_i\geq x_{i+1}$ for all $i\in [n-2]$. 
			The equation 
			\begin{align}\label{RFequ:x}
				\sum_{i=1}^{n-1} 2^{n-i-1}\cdot x_i = 2^{n-1}n-2^n+1 \tag{E1}
			\end{align}
			has $x_i=n-i$ for all $i\in [n-1]$ as its only solution. 
			\item[(S2)] Let non-negative variables $\overline{x}_0,\ldots,\overline{x}_{n-2}\in \{0,\ldots,n\}$ be given with a non-descending restriction $\overline{x}_i\leq \overline{x}_{i+1}$ for all $i\in \{0,\ldots,n-3\}$. 
			The equation 
			\begin{align}\label{RFequ:xquer}
				\sum_{i=0}^{n-2} 2^{i}\cdot \overline{x}_i = 2^{n-1}n-2^n+1 \tag{E2}
			\end{align}
			has $\overline{x}_i=i+1$ for all $i\in \{0,\ldots,n-2\}$ as its only solution.
			\item[(S3)] Let non-negative variables $y_0,\ldots,y_{n-2}\in \{0,\ldots,n\}$ be given with $\sum_{i=0}^{n-2}y_i \leq n $. The equation 
			\begin{align}\label{RFequ:y}
				2^{n-1}\sum_{i=0}^{n-2} y_i - \sum_{i=0}^{n-2} 2^i y_i= 2^{n-1}n-2^n+1 \tag{E3}
			\end{align}
			has $y_i=1$ for all $i\in \{0,\ldots,n-2\}$ as its only solution. 
	\end{itemize}
\end{lemma}
\begin{proof}
Firstly, by shifting the index of the left hand-side of equation~\eqref{RFequ:x} and summing up the other way around, we obtain 
\begin{align*}
\sum_{i=1}^{n-1} 2^{n-i-1}\cdot x_i 
= \sum_{i=0}^{n-2} 2^{n-i-2}\cdot x^\prime_i 
= \sum_{i=0}^{n-2} 2^{i}\cdot \overline{x}_i  
\end{align*}
with $x^\prime_i := x_{i+1}$ for all $i\in \{0,\ldots,n-2\}$ and $\overline{x}_i := x^\prime_{n-2-i}$ for all $i\in \{0,\ldots,n-2\}$. 
If $x_i \geq x_{i+1}$ is true for all $i\in [n-2]$, it holds $\overline{x}_i \leq \overline{x}_{i+1}$ for all $i\in \{0,\ldots,n-3\}$. 
The assignment $x_i=n-i$, $i\in [n-1]$ is the only solution to equation~\eqref{RFequ:x} if and only if $\overline{x}_i=i+1$, $i\in \{0,\ldots,n-2\}$ is the only solution to equation~\eqref{RFequ:xquer}. 
Consequently, statement~(S1) is equivalent to statement~(S2). 

We define new variables $y_0:=\overline{x}_0$ and $y_i := \overline{x}_i - \overline{x}_{i-1}$ for all $i\in [n-2]$. 
Because of $\overline{x}_0\geq 0$ and $\overline{x}_i\leq\overline{x}_{i+1}$ for $i\in \{0,\ldots,n-3\}$, it holds $y_i\geq 0$ for all $i\in \{0,\ldots,n-2\}$. 
Furthermore, for all $i\in \{0,\ldots,n-2\}$ it holds $\overline{x}_i = \sum_{j=0}^{i}y_j$
and as $\overline{x}_{n-2}\in \{0,\ldots,n\} $ we obtain $\sum_{i=0}^{n-2}y_i\leq n$. 
Using this, we obtain
\begin{align*}
\sum_{i=0}^{n-2} 2^{i}\cdot \overline{x}_i 
				=\sum_{i=0}^{n-2} 2^{i}\cdot \left(\sum_{j=0}^{i} y_j\right) 
				=\sum_{j=0}^{n-2}  \left(\sum_{i=j}^{n-2} 2^{i} \right)y_j 
				=\sum_{j=0}^{n-2}  \left(2^{n-1}- 2^{j} \right)y_j 
				=2^{n-1} \sum_{j=0}^{n-2}  y_j - \sum_{j=0}^{n-2} 2^{j} y_j.
\end{align*}
The assignment $\overline{x}_i=i+1$, $i\in \{0,\ldots,n-2\}$ is the only solution to equation \eqref{RFequ:xquer} if and only if $y_0=\overline{x}_0=1$ and $y_i = \overline{x}_i - \overline{x}_{i-1} = i+1 - i = 1 $ is the only solution to equation \eqref{RFequ:y}. 
Consequently, statement~(S2) is equivalent to statement~(S3). 
\end{proof}
%
% % % % % % % % % % % % % % % % % % % % % % % % % % % % % % % % % % % % % % % % % % % % % % % % % % % % % % % % % % % % %
In the following, we focus on statement~(S3) of Lemma~\ref{RFlem:Aequivalenz} to prove Theorem~\ref{RFtheorem:Potenzen}.  
For the proof, we use Algorithm~\ref{RFalgo:TransformationOfSolutions} to transform solutions to equation 
\begin{align}
	\sum_{i=0}^{n-2} 2^i y_i= 2^n-1. \label{RFequ:equAlg}
\end{align}
\begin{algorithm*}
	\caption{}\label{RFalgo:TransformationOfSolutions}
\begin{algorithmic}[1] 
	\State \algorithmicrequire \ solution $\fett{y}=(y_0,\ldots,y_{n-2})\in \mathbb{Z}^{n-1}_{\geq 0}$ and parameter $\ell \in \mathbb{Z}_{\geq 2}$
	\State Set $j=1$ and $\fett{y^1} : = \fett{y}$
	\Procedure{}{}
	\While{ $\exists \ k \in \{0,\ldots,n-3\}: \ y^j_k\geq 2$ }
	\State Let $k$ be minimal
	\State Set % \begin{align*}
	$y^{j+1}_i =
	\begin{cases}
		y^j_{k} -2 		& \text{ for }i=k \\
		y^j_{k+1} +1		& \text{ for }i=k+1\\
		y^j_i							& \text{ otherwise }
	\end{cases}$
	\State Set $j=j+1$
	\If {$y^{j+1}_{n-2} = \ell $}
	\State \textbf{break}
	\EndIf
	\EndWhile \\
	\Return $\fett{y^j}=(y^j_0,\ldots,y^j_{n-2})$
	\EndProcedure
\end{algorithmic}
\end{algorithm*}
More precisely, the algorithm transforms an arbitrary solution $\fett{y}= \fett{y^1}$ to equation~\eqref{RFequ:equAlg} after $j-1$ steps to a different solution $\fett{y^j}$ to equation~\eqref{RFequ:equAlg} with $y^j_{n-2}=\ell$ or $y^j_i\in \{0,1\}$ for $i\in \{0,\ldots,n-3\}$ and $y^j_{n-2}\in \{0,1,\ldots,\ell-1\}$ such that the sum of the entries is reduced by $j-1$, i.e., 
\begin{align*}
	\sum_{i=0}^{n-2} y^j_i = \sum_{i=0}^{n-2} y^1_i - (j-1). %\label{RFequ:AlgReducedEntries}
\end{align*} 
The algorithm is used in the next lemma to show that a solution assumed to exist does not satisfy equation~\eqref{RFequ:y}. 
% % % % % % % % % % % % % % % % % % % % % % % % % % % % % % % % % % % % % % % % % % % % % % % % % % % % % % % % % % % %
\begin{lemma}\label{RFlem:erstesHilfslemma}
	Let $\fett{y}=(y_0,\ldots,y_{n-2}) \in \{0,\ldots,n\}^{n-1}$ with $\sum_{i=0}^{n-2}y_i \leq n $ be a solution to the equation 
	\begin{align}\label{RFequ:erstesHilfslemma}
		2^{n-1}\sum_{i=0}^{n-2} y_i - \sum_{i=0}^{n-2} 2^i y_i= 2^{n-1}n-2^n+1. \tag{E3}
	\end{align} 
	Then, it holds $\sum_{i=0}^{n-2}y_i = n-1 $. 
\end{lemma}
\begin{proof}
	Assume the statement is false, then it either holds $\sum_{i=0}^{n-2} y_i \leq n-2$ or $\sum_{i=0}^{n-2} y_i = n$. 
	In the first case, we obtain 
	\begin{align*}
		2^{n-1}n - 2^n+1
		=2^{n-1} \sum_{i=0}^{n-2} y_i  -  \sum_{i=0}^{n-2} 2^i y_i 
		\leq 2^{n-1} (n-2)  -  \sum_{i=0}^{n-2} 2^i y_i 
		< 2^{n-1} (n-2) +1 
		= 2^{n-1}n - 2^n+1,
	\end{align*}
	which is a contradiction.
	In the second case, the equation~\eqref{RFequ:erstesHilfslemma} reduces to the following equivalent equation
	\begin{align}
			2^{n-1}n - \sum_{i=0}^{n-2} 2^i y_i= 2^{n-1}n-2^n+1
			\Leftrightarrow
			\sum_{i=0}^{n-2} 2^i y_i= 2^n-1.\label{RFequ:erstesHilfslemmaNewEqu}
	\end{align}
	For a solution $\boldsymbol{y}\in \mathbb{Z}^{n-1}_{\geq 0}$ to equation~\eqref{RFequ:erstesHilfslemmaNewEqu} there exists a $k\in \{0,\ldots,n-2\}$ such that $y_k\geq 2$ holds by the pigeonhole principle. 
	Let $k$ be minimal. 
	If $k=n-2$, we distinguish between the following two cases. 
	\begin{itemize}
		\item[]1. Case  $y_{n-2}\geq 4$:
		We obtain
		\begin{align*}
			 2^n-1 = \sum_{i=0}^{n-2} 2^i y_i \geq 2^{n-2} y_{n-2} = 2^{n-2}\cdot 4 = 2^n > 2^n -1,
		\end{align*}
		which is a contradiction. 
		\item[]2. Case $y_{n-2} \in \{2,3\}$:
		We note that $\fett{y^\prime}=(y^\prime_0, \ldots,y^\prime_{n-2})=(1,\ldots,1,3)$ satisfies $\sum_{i=0}^{n-2} 2^i y^\prime_i = 2^n -1 $ but it holds $\sum_{i=0}^{n-2} y^\prime_i = n-2+3= n+1$. 
		As $k=n-2$ is minimal, solution $\fett{y}$ consists of $y_i\in \{0,1\}$ for $i\in \{0,\ldots,n-3\}$ and $y_{n-2}\in \{2,3\}$. 
		As $\sum_{i=0}^{n-2} y_i = n< n+1 = \sum_{i=0}^{n-2} y^\prime_i $ is true, it holds $y_i=0$ for at least one $i\in \{0,\ldots,n-3\}$ or it holds $y_{n-2}=2$. 
		In both cases, we obtain
		\begin{align*}
			2^n-1
			=\sum_{i=0}^{n-2} 2^i y_i
			< \sum_{i=0}^{n-2} 2^i y^\prime_i
			= 2^{n}-1,
		\end{align*}
		which is a contradiction.
	\end{itemize}
	If $0\leq k<n-2$, we apply Algorithm~\ref{RFalgo:TransformationOfSolutions} on solution $\fett{y}$ for $\ell = 4$. 
	After $j-1>0$ steps, the algorithm terminates and provides a solution $\fett{y^j}\in \mathbb{Z}^ {n-1}_{\geq 0}$ with $y^j_{n-2}=\ell$ or $y^j_i\in \{0,1\}$ for $i\in \{0,\ldots,n-3\}$ and $y^j_{n-2}\in \{0,1,\ldots,\ell-1\}$ such that the following holds	
	$$\sum_{i=0}^{n-2} y^j_i = \sum_{i=0}^{n-2} y_i - (j-1) = n-j+1.$$
	If $y^j_{n-2}= \ell = 4$, we obtain analogous to Case 1
	\begin{align*}
		2^n-1 = \sum_{i=0}^{n-2} 2^i y_i = \sum_{i=0}^{n-2} 2^i y^j_i \geq 2^{n-2} y^j_{n-2} = 2^{n-2}\cdot 4 = 2^n > 2^n -1,
	\end{align*}
	which is a contradiction.
	Otherwise, solution $\fett{y^j}$ consists of $y^j_i\in \{0,1\}$ for $i\in \{0,\ldots,n-3\}$ and $y^j_{n-2}\in \{0,1,2,3\}$. 
	Analogues to Case 2, we obtain 
	\begin{align*}
		2^n-1
		=\sum_{i=0}^{n-2} 2^i y_i
		=\sum_{i=0}^{n-2} 2^i y^j_i
		< \sum_{i=0}^{n-2} 2^i y^\prime_i
		= 2^{n}-1,
	\end{align*}
	which is a contradiction.
	Overall, we have proven $\sum_{i=0}^{n-2}y_i=n-1$.
\end{proof}

\begin{lemma}\label{RFlem:LetzteStelleEins}
	Let non-negative variables $y_0,\ldots,y_{n-2}\in \{0,\ldots,n\}$ be given with $\sum_{i=0}^{n-2}y_i = n-1$. 
	The equation 
	\begin{align}\label{RFequ:LemmaMitY}
		2^{n-1}\sum_{i=0}^{n-2} y_i - \sum_{i=0}^{n-2} 2^i y_i= 2^{n-1}n-2^n+1 \tag{E3}
	\end{align}
	has $ y_i=1$ for all $i\in \{0,\ldots,n-2\}$ as its only solution.
\end{lemma}
\begin{proof}
	As $\sum_{i=0}^{n-2}y_i = n-1 $ holds, the equation~\eqref{RFequ:LemmaMitY} reduces to the following equivalent equation
	\begin{align}
		2^{n-1}(n-1) - \sum_{i=0}^{n-2} 2^i y_i= 2^{n-1}n-2^n+1 %\nonumber\\
		\Leftrightarrow
		\sum_{i=0}^{n-2} 2^i y_i= 2^{n-1}-1.\label{RFequ:HilfslemmaAequivalenz}
	\end{align}
	Assume now the statement is false, i.e., there exists a solution $\fett{y^1} = (y^1_0, \ldots, y^1_{n-2}) \in \{0,\ldots,n\}^{n-1}$ to equation~\eqref{RFequ:HilfslemmaAequivalenz} with $\sum_{i=0}^{n-2}y^1_i = n-1$ where $\fett{y^1}  \neq \fett{y^*}:=(1,\ldots,1)$ holds.
	There exists a $k\in \{0,\ldots,n-2\}$ such that $y^1_k\geq 2$ holds by the pigeonhole principle.
	Let $k$ be minimal.
	If $k=n-2$, we obtain 
	\begin{align}\label{RFequ:AppendixExprContradiction}
		2^{n-1}-1 =	\sum_{i=0}^{n-2} 2^i y^1_i \geq  2^{n-2}y^1_{n-2} \geq 2^{n-2}\cdot 2 = 2^{n-1} > 2^{n-1}-1, 
	\end{align}
	which is a contradiction.	
	If $0\leq k<n-2$, we apply Algorithm~\ref{RFalgo:TransformationOfSolutions} on solution $\fett{y^1}$ for $\ell=2$. 
	After $j-1>0$ steps, the algorithm terminates and provides a solution $\fett{y^j}\in \mathbb{Z}^{n-1}_{\geq 0}$ with $y^j_{n-2}=\ell$ or $y^j_i\in \{0,1\}$ for $i\in \{0,\ldots,n-3\}$ and $y^j_{n-2}\in \{0,1,\ldots,\ell-1\}$ such that the following holds
	$$\sum_{i=0}^{n-2} y^j_i = \sum_{i=0}^{n-2} y^1_i - (j-1) = n-j.$$ 
	If $y^j_{n-2}= \ell = 2$, we obtain a contradiction analogues to expression~\eqref{RFequ:AppendixExprContradiction}.
	Otherwise, solution $\fett{y^j}\in \mathbb{Z}^{n-1}_{\geq 0}$ consists of $y^j_i\in \{0,1\}$ for all $i\in \{0,\ldots,n-2\}$. 
	Furthermore, as at least one step of the algorithm has been performed, it holds $y^j_i=0$ for at least one $i\in \{0,\ldots,n-2\}$.
	We obtain
	\begin{align*}
		2^{n-1}-1 
		= \sum_{i=0}^{n-2} 2^i y^1_i 
		= \sum_{i=0}^{n-2} 2^i y^j_i	
		< \sum_{i=0}^{n-2} 2^i y^*_i
		= 2^{n-1}-1,
	\end{align*}
	which is a contradiction.
\end{proof}
\section{Appendix}
\label{RF:AppendixD}
\begin{lemma}\label{RFlem:SubgraphOfSPGraphsDefinedByTwoVertices}
	Let $G$ be an SP digraph and $v,w\in V(G)$ be two specified vertices such that there exists a $(v,w)$-path. 
	Let $G_{vw}$ be the subgraph spanned by vertices $v,w$. 
	Subgraph $G_{vw}$ is an SP digraph itself with origin $v$ and target $w$. 
\end{lemma}
\begin{proof}
	We prove the correctness of the statement by induction on the number of the digraph's arcs $m:=|A(G)|$. 
	For the beginning, if we consider an SP digraph consisting of only one arc, the statement is readily apparent. 
	In the next step, we assume the statement holds true for all SP digraphs consisting of at most $m$ arcs. 	
	In the following, we prove the statement for an SP digraph with $m + 1$ arcs.  	
	We distinguish whether SP digraph $G$ is composed serially or in parallel. 
	
	Firstly, we assume that $G$ is a series composition of SP digraphs $G_1$ and $G_2$. 
	The target of SP digraph $G_1$ and the origin of SP digraph $G_2$ are contracted to one vertex, denoted by $z$.
	Let $v,w\in V(G)$ be two specified vertices such that there exists a $(v,w)$-path. 
	If $v,w\in V(G_1)$ or $v,w\in V(G_2)$ holds, the statement results by the induction hypothesis. 
	Otherwise, it holds $v\in V(G_1)$ and $w\in V(G_2)$ due to the $(v,w)$-path.
	We separately consider SP digraphs $G_1$ and $G_2$ with specified vertices $v,z\in V(G_1)$ and $z,w\in V(G_2)$, respectively.
	Clearly, $(v,z)$- and $(z,w)$-paths exist, as the $(v,w)$-path must contain the connecting vertex $z$. 
	By the induction hypothesis, we obtain that the spanned subgraphs $G_{vz}$ and $G_{zw}$ are SP digraphs themselves. 
	If we serially compose subgraphs $G_{vz}$ and $G_{zw}$ at vertex $z$, we obtain subgraph $G_{vw}$ spanned by vertices $v$, $w$. 
	Subgraph $G_{vw}$ is an SP digraph as it is a series composition of two SP digraphs. 
	
	Secondly, we assume that $G$ is a parallel composition of SP digraphs $G_1$ and $G_2$. 
	Let $v,w\in V(G)$ be two specified vertices such that there exists a $(v,w)$-path. 
	It either holds $v,w\in V(G_1)$ or $v,w\in V(G_2)$ and thus the statement results by the induction hypothesis. 
\end{proof}

\end{document}